\documentclass[11pt]{article}

\usepackage[T1]{fontenc}
\usepackage{lmodern}
\usepackage[margin=1in]{geometry}
\usepackage{amsmath,amssymb,amsthm,mathtools,bm}
\usepackage{xcolor}
\usepackage{graphicx}
\usepackage{booktabs}
\usepackage{placeins}
\usepackage{enumitem}
\usepackage{tabularx}
\usepackage{array}
\usepackage{authblk}
\usepackage[round,authoryear]{natbib}
\usepackage[colorlinks=true,citecolor=blue,linkcolor=blue,urlcolor=blue]{hyperref}
\usepackage{microtype}

\graphicspath{{paper_results/}}

\newcolumntype{Y}{>{\raggedright\arraybackslash}X}
\newcolumntype{C}[1]{>{\centering\arraybackslash}p{#1}}

\numberwithin{equation}{section}
\allowdisplaybreaks
\theoremstyle{plain}
\newtheorem{theorem}{Theorem}[section]
\newtheorem{proposition}[theorem]{Proposition}
\newtheorem{lemma}[theorem]{Lemma}
\newtheorem{corollary}[theorem]{Corollary}
\theoremstyle{definition}
\newtheorem{remark}[theorem]{Remark}
\newtheorem{example}[theorem]{Example}
\newtheorem{definition}[theorem]{Definition}

\newcommand{\E}{\mathbb{E}}
\newcommand{\Var}{\operatorname{Var}}
\newcommand{\Cov}{\operatorname{Cov}}
\newcommand{\R}{\mathbb{R}}

\newcommand{\Proj}{\Pi}

\title{Semiparametric Efficient Inference under Non-Informative Complex Survey Designs}
\author[1]{Hiroki Chiba\thanks{Email: \texttt{chiba@iastate.edu}}}
\author[1]{Kosuke Morikawa\thanks{Email: \texttt{morikawa@iastate.edu}}}
\affil[1]{Department of Statistics, Iowa State University}
\date{}

\hypersetup{
  pdftitle={Semiparametric Efficient Inference under Non-Informative Complex Survey Designs},
  pdfauthor={Hiroki Chiba and Kosuke Morikawa}
}

\begin{document}
\maketitle

\begin{abstract}
Under non-informative complex survey designs, dependence among sampling indicators prevents the direct application of standard i.i.d.\ semiparametric efficiency theory. We introduce a \emph{reference Poisson experiment}, a hypothetical i.i.d.\ experiment constructed from the limiting inclusion probabilities of the original dependent design. Under suitable conditions, the actual observed-data experiment is in fact locally asymptotically normal, with the same tangent-space geometry as the reference Poisson experiment. Consequently, the efficient influence function derived by orthogonal projection in this hypothetical i.i.d.\ setting characterizes efficient estimators under the original dependent design. This yields a systematic approach to efficiency analysis for general semiparametric full-data models and pathwise-differentiable targets under complex survey designs. Under the joint superpopulation--design law, the efficiency bounds depend on the design only through its limiting inclusion probabilities. For finite-population targets whose first-order expansions are governed by the full-data efficient influence function, the same estimators are efficient under both superpopulation and finite-population centering, with bounds linked by a Pythagorean identity yielding a generalized finite-population correction. We establish efficiency of cross-fitted mean and conditional-mean regression estimators under explicit design-specific sufficient conditions. Simulations and California Academic Performance Index data illustrate the theory.
\end{abstract}



\section{Introduction}
\label{sec-intro}

Complex survey designs use stratification, clustering, unequal inclusion probabilities, and multiple stages to collect data efficiently, reducing costs or improving precision \citep{Neyman1934,HorvitzThompson1952,SarndalSwenssonWretman1992,Fuller2009}. Such designs, however, make the sampled units dependent: fixing the sample size, overall or within strata, can correlate the sampling indicators negatively, and selecting whole clusters or primary sampling units (PSUs) can correlate them positively within groups. When covariates are observed for every unit, the observed data can be regarded as a missing-data problem in which the missingness is induced by a dependent sampling design.

How to use such data efficiently is the subject of semiparametric efficiency theory, which provides, for a given model and target, the smallest attainable asymptotic variance and the efficient influence function (EIF) characterizing the estimators that attain it \citep{BickelKlaassenRitovWellner1993,VanDerVaart1998}. For missing data, \citet{RobinsRotnitzkyZhao1994} and \citet{Tsiatis2006} derive the observed-data EIF from the full-data EIF by a projection argument in the observed-data tangent space. Because the argument uses only the geometry of the tangent space and not the particular parameter, it gives a systematic derivation of the EIF for general parameters. Both the general theory and the missing-data theory, however, are built for i.i.d.\ observations and for targets that are parameters of the superpopulation, such as a population mean or a regression coefficient. Survey sampling departs from this setting on both counts. The observed data are dependent, and the targets of interest are often characteristics of the realized finite population, such as the finite-population mean or a census regression coefficient, which are random under the superpopulation model and depend on unobserved outcomes. Even under independent sampling, therefore, it is unclear, beyond means and totals, what the efficiency bound of such a target is or how an efficient estimator should be constructed.

Dependent designs and finite-population targets have been studied, largely for means and totals. Under dependent designs, \citet{GodambeJoshi1965} derived the anticipated-variance lower bound for finite-population totals, and \citet{IsakiFuller1982} studied optimal design--estimator pairs attaining the bound under linear regression models. \citet{BreidtOpsomer2000} proved that local-polynomial model-assisted estimators attain the bound asymptotically. \citet{AnDagdougHazizaTille2026} proved the same for cross-fitted machine-learning model-assisted estimators, but their results exclude designs with non-vanishing positive conditional covariances among sampling indicators, as can occur under cluster and two-stage sampling. Other work on inference under dependent survey designs includes linearization \citep{Binder1983}, calibration and generalized regression estimation \citep{DevilleSarndal1992}, rejective and high-entropy sampling \citep{Hajek1964,Tan2013}, and functional and uniform limit theorems \citep{BoistardLopuhaaRuizGazen2017,HanWellner2021}. These studies address estimation methods, sampling designs, and asymptotic properties of estimators rather than deriving general semiparametric efficiency bounds.

Dependence has also been studied in experimental design, where treatment assignment is dependent across units. Efficiency bounds have been derived through least-favorable submodels of the corresponding i.i.d.\ problem, assuming a particular form of the EIF beyond mean targets \citep{Armstrong2022}; as convolution bounds for dependent assignments with a known marginal propensity and nuisance-free moment-condition targets \citep{BaiLiuShaikhTabordMeehan2026}; for the average treatment effect under covariate-adaptive randomization \citep{Rafi2023}; and as average-propensity bounds in sequential experiments \citep{LiSimchiLevi2026}. All of these bounds concern superpopulation targets. These studies do not provide a unified projection framework accommodating general full-data model restrictions and pathwise-differentiable targets under dependent survey designs.

In survey sampling and related two-phase studies, semiparametric efficiency results for superpopulation parameters have been obtained in several settings. \citet{BreslowRobinsWellner2000} established the efficiency of standard logistic regression slope estimators under case-control sampling, while \citet{BreslowMcNeneyWellner2003} derived information bounds and established the efficiency of maximum likelihood estimators for regression models under outcome-dependent two-phase sampling with Bernoulli second-phase selection. For weighted likelihood estimators, \citet{BreslowWellner2007} and \citet{SaegusaWellner2013} treated stratified sampling without replacement at the second phase and showed that the asymptotic variance differs from that under Bernoulli sampling. \citet{ZengLin2014} developed efficient estimators of regression coefficients in semiparametric transformation models under two-phase cohort designs, including nested case-control sampling with dependent second-phase selection. For two-phase studies with continuous or discrete outcomes, \citet{TaoZengLin2017} developed efficient sieve likelihood estimators that accommodate continuous first-phase covariates, and \citet{TaoLotspeichAmorimShawShepherd2021} extended this approach to linear regression with measurement errors in both outcomes and covariates. Under informative Poisson sampling, \citet{MorikawaTeradaKim2025} derived efficiency bounds and proposed adaptive estimators for superpopulation parameters. Related to finite-population inference, \citet{Zhang2005} derived information bounds and a convolution theorem for additive random targets under an i.i.d.\ model. For non-i.i.d.\ experiments, \citet{Milbrodt1985} established local asymptotic normality (LAN) under Poisson and rejective sampling in a class of superpopulation regression models, while \citet{McNeney2000} developed general tools for applying convolution theorems to non-i.i.d.\ semiparametric models. These contributions do not, however, directly provide a unified efficiency theory that accommodates both dependent survey designs and random finite-population targets.

We establish a unified efficiency theory for general targets under dependent designs, covering finite-population as well as superpopulation targets. The key step is to show that an independent Poisson sampling experiment exists whose tangent space and inner product coincide with those of the dependent survey experiment. We work under the joint law of the superpopulation model and the sampling design \citep{RubinBleuerSchiopuKratina2005}, for non-informative designs whose first-order inclusion probabilities converge uniformly to a function of the covariates alone, which we call the limiting inclusion probability; simple random, stratified, cluster, probability-proportional-to-size, rejective, and stratified two-stage sampling all qualify. We introduce the \emph{reference Poisson experiment}, which samples each unit independently with its limiting inclusion probability. As an ordinary i.i.d.\ missing-data experiment, it has an explicit tangent space and inner product, involving the design only through the limiting inclusion probability. Our main result is that the dependent experiment is LAN with exactly this tangent space and inner product.

The consequence is that standard semiparametric theory applies unchanged. For any pathwise-differentiable target in any semiparametric full-data model, the EIF under the dependent design is obtained by the standard missing-at-random projection of \citet{RobinsRotnitzkyZhao1994} and \citet{Tsiatis2006}, carried out in the reference Poisson experiment, and characterizes the efficient estimators. EIFs can thus be derived systematically for general parameters, and working under Poisson sampling is justified rather than assumed: for a given model and target, all designs in our class with the same limiting inclusion probabilities share the same efficient influence function and efficiency bound. 

The second half of the theory concerns finite-population targets. For targets whose deviation from the superpopulation target is, to first order, the population average of the full-data EIF, such as finite-population moments and census regression coefficients, convolution theorems under both centerings show that the same estimators are efficient for the finite-population and the superpopulation target. The superpopulation bound equals the finite-population bound plus the full-data bound, which can be interpreted as the Pythagorean identity yielding a generalized finite-population correction. For the finite-population mean, the finite-population bound is the large-sample limit of the Godambe--Joshi bound.

We then construct cross-fitted estimators of means and conditional-mean regression coefficients that attain the bounds, with consistent variance estimators. Dependent sampling breaks the independence between nuisance functions fitted outside a fold and the indicators in that fold. Because cross-fitting exploits this independence in the i.i.d.\ setting \citep{ChernozhukovEtAl2018}, the nuisance remainder needs a design-specific analysis. For the mean, under the stated sampling and fold conditions, mean-square consistency of the regression estimator suffices for the individual-unit sampling designs considered here, and under cluster and two-stage sampling it suffices that the largest group size times the empirical mean-square error vanishes. Regression targets need no additional group-size condition.

The remainder of the paper is organized as follows. Section~\ref{sec-basic-setup} introduces the setup and outlines the efficiency framework, Section~\ref{sec-lan-efficiency} establishes the efficiency results, and Section~\ref{sec-efficient-estimation} develops the cross-fitted estimators. Section~\ref{sec-num_exp} reports simulations, Section~\ref{sec-real-data} analyzes California Academic Performance Index data, and Section~\ref{sec-discussion} concludes. Proofs, verification of the design assumptions, and additional numerical results are given in Appendices~\ref{app-technical-details} and~\ref{app-additional-numerical-details}.

\section{Basic Setup}
\label{sec-basic-setup}
\label{sec-setup}
\subsection{Data sources and observed likelihoods}
\label{sec-datasources}
Consider a finite population indexed by \(\{1,\ldots,N\}\). For \(i=1,\ldots,N\), let \(X_i\) and \(Y_i\) be the covariate and study variable of unit \(i\), respectively. The full-data units \(L_i=(X_i,Y_i)\) are independent and identically distributed. Let \(\boldsymbol{\delta}_N=(\delta_1,\ldots,\delta_N)\), where \(\delta_i\) indicates whether unit \(i\) is sampled. The observed data are defined by \(O_i=(\delta_i,X_i,Y_i)\) if \(\delta_i=1\) and \(O_i=(\delta_i,X_i)\) if \(\delta_i=0\).

Let \(\{P_{\theta,\eta}^F\mid \theta\in\Theta\subset\mathbb R^p,\ \eta\in\Xi\}\) be a semiparametric model for the full-data units \(L_i=(X_i,Y_i)\). Here \(\theta\) is the finite-dimensional superpopulation parameter of interest, and \(\eta\) is a possibly infinite-dimensional nuisance parameter specifying the remaining features of the distribution. For example, when \(\theta\) is the mean of \(Y\), \(\eta\) may index the joint distribution of \((X,Y)\) subject to this mean constraint. The true parameter is \((\theta_0,\eta_0)\). Let \(f_{\theta,\eta}(x,y)\) be the full-data density.
Throughout this paper, \(P_0^F\) denotes the true full-data distribution, and \(\E\) denotes expectation under the true distribution. The symbols \(\xrightarrow[]{d}\) and \(\xrightarrow[]{p}\) denote convergence in distribution and convergence in probability, respectively.

For \((\theta,\eta)\in\Theta\times\Xi\), let \(P_{\theta,\eta}^N\) denote the distribution induced by the observed data \((O_i)_{i=1}^N\) and \(P_0^N\) denote the true observed-data distribution.
We assume throughout that the sampling mechanism is non-informative,
\begin{equation}
\boldsymbol{\delta}_N\perp\bm Y_N\mid\bm X_N,
\label{eq-non-informative}
\end{equation}
where \(\bm X_N=(X_1,\dots,X_N)\) and \(\bm Y_N=(Y_1,\dots,Y_N)\). Under \eqref{eq-non-informative}, the observed-data likelihood factors as
\begin{equation}
q_N(\boldsymbol{\delta}_N\mid \bm{X}_N)
\prod_{i=1}^N
\left\{
f_{\theta,\eta}(X_i,Y_i)^{\delta_i}
\left(
\int f_{\theta,\eta}(X_i,y)\,dy
\right)^{1-\delta_i}
\right\},
\label{eq-obs-lik}
\end{equation}
where \(q_N(\boldsymbol{\delta}_N\mid \bm{X}_N)\) is the known conditional law of \(\boldsymbol{\delta}_N\) given \(\bm{X}_N\) and does not depend on \((\theta,\eta)\). Let
$\pi_{N,i}
=
P_0^N(\delta_i=1\mid \bm X_N)
>0$
be the first-order inclusion probability of unit \(i\) given \(\bm X_N\). The designs considered below are summarized in Table~\ref{tab-sampling-examples}, which includes Poisson sampling, simple random sampling without replacement (SRSWOR), stratified sampling, cluster sampling, probability-proportional-to-size (PPS) sampling, rejective sampling, and stratified two-stage sampling. Clusters or PSUs may be formed using covariates.

Let \(\theta_0=\psi(P_0^F)\in\mathbb R^p\) be the superpopulation target. For finite-population inference, let \(\theta_N\) be an \(\mathbb R^p\)-valued finite-population target measurable with respect to \(\sigma(\bm X_N,\bm Y_N)\), such as a finite-population mean, and let \(Q_{\theta,\eta}^N\) be the joint law of \((\boldsymbol{\delta}_N,\bm X_N,\bm Y_N)\). Write \(Q_0^N\) for the true joint law and reuse \(\E\) for its expectation.
Our finite-population targets are characterized by
\begin{equation}
\sqrt N(\theta_N-\theta_0)
=
\frac1{\sqrt N}\sum_{i=1}^N\phi^F(L_i)
+
o_{Q_0^N}(1).
\label{eq-A3-finite-pop-target-expansion}
\end{equation}
Here \(\phi^F\) is the full-data EIF in Assumption~\ref{ass-A2}.
This expansion requires the finite-population target to behave asymptotically as an efficient full-data estimator of the superpopulation target.

The correspondence depends on the full-data model. Under an unrestricted joint distribution of $(X,Y)$, the superpopulation \(r\)th moment \(\theta_0=\mathbb E[Y^r]\) is paired with the finite-population $r$th moment \(\theta_N=N^{-1}\sum_{i=1}^N Y_i^r\), with \(\phi^F(L)=Y^r-\theta_0\). For regression targets, the best linear projection coefficient defined by $\mathbb E[X(Y-X^\top\beta_0)]=0$ is paired with census ordinary least squares (OLS), defined by $N^{-1}\sum_{i=1}^N X_i(Y_i-X_i^\top\beta_N)=0$. When $\mathbb E[XX^\top]$ is nonsingular and the standard moment conditions hold, the full-data EIF is $\{\mathbb E[XX^\top]\}^{-1}X(Y-X^\top\beta_0)$, and $\beta_N$ satisfies \eqref{eq-A3-finite-pop-target-expansion} \citep{VanDerVaart1998}.
In the linear conditional-mean model $\mathbb E[Y\mid X]=X^\top\beta_0$, the full-data EIF in \eqref{eq-regression-eifs} generally differs from the influence function of census OLS under heteroskedasticity. Example~\ref{ex-eif-regression} instead pairs $\beta_0$ with the inverse-variance-weighted census target in \eqref{eq-regression-finite-population-target}, whose expansion has the required full-data EIF.

\subsection{Reference Poisson experiment}
\label{sec-tangent-space}
When the sampling indicators are dependent, the observed-data score functions and tangent spaces are not determined from a single observation.
To systematically derive the observed-data EIF using i.i.d.\ projection arguments, we introduce the \emph{reference Poisson experiment}, a hypothetical i.i.d.\ statistical experiment.
We consider designs whose first-order inclusion probabilities satisfy
\begin{equation}
\sup_{1\le i\le N}|\pi_{N,i}-\pi(X_i)|=o_{P_0^N}(1),
\label{eq-pi-convergence}
\end{equation}
where \(\pi\) is a deterministic function of the covariates giving the limiting inclusion probability.
For example, under simple random sampling without replacement of \(n_N\) units from a population of size \(N\), the inclusion probability is \(\pi_{N,i}=n_N/N\) for every unit.
Assuming that \(n_N/N\to\rho\in(0,1)\), condition~\eqref{eq-pi-convergence} holds with \(\pi(x)\equiv\rho\).
Table~\ref{tab-sampling-examples} gives the corresponding limits for other sampling designs.

\begin{table}[t]
\centering
\setlength{\tabcolsep}{4pt}
\renewcommand{\arraystretch}{1}
\caption{Representative sampling designs satisfying Assumptions~\ref{ass-B1}--\ref{ass-B4} under the conditions verified in Appendix~\ref{app-sampling-examples}.}
\label{tab-sampling-examples}
\begin{tabularx}{\linewidth}{@{}>{\raggedright\arraybackslash}p{0.24\linewidth} C{0.34\linewidth} Y@{}}
\toprule
Sampling design & Limiting inclusion probability \(\pi(x)\) & Notes \\
\midrule

Poisson sampling
&
\(\displaystyle \pi(x)\)
&
\(\pi\) is any function satisfying Assumption~\ref{ass-B4}. \\
\addlinespace[1pt]

Simple random sampling without replacement
&
\(\displaystyle \pi(x)\equiv \rho\)
&
\(\rho=\lim n_N/N\). \\
\addlinespace[1pt]

Cluster sampling
&
\(\displaystyle \pi(x)\equiv \rho\)
&
\(m_N/M_N\to\rho\), and selected clusters are fully observed. The largest cluster contains \(o_p(N)\) units. \\
\addlinespace[1pt]

PPS sampling with replacement
&
\(\displaystyle
\pi(x)
=
1-\exp\!\left\{
-\rho\,\frac{r(x)}{\mathbb{E}[r(X)]}
\right\}
\)
&
\(r\) is the size measure, and
\(\rho=\lim n_N/N\). \\
\addlinespace[1pt]

Rejective sampling
&
\(\displaystyle
\pi(x)
=
\rho\,\frac{r(x)}{\mathbb{E}[r(X)]}
\)
&
Canonical Bernoulli probabilities \(p_{N,i}\) are chosen as in \citet{Hajek1964}. \\
\addlinespace[1pt]

Stratified sampling
&
\(\displaystyle
\pi(x)
=
\sum_{j=1}^J
p_j\,\mathbf 1\{x\in\mathcal X_j\}
\)
&
Independent SRSWOR within \(J\) fixed strata \(\mathcal X=\bigsqcup_{j=1}^J\mathcal X_j\), with limiting fractions \(p_j\). \\
\addlinespace[1pt]

Stratified two-stage sampling
&
\(\displaystyle \pi^{(1)}(x)\pi^{(2)}(x)\)
&
The number of strata is fixed. The largest PSU contains \(o_p(N)\) units. \\
\addlinespace[1pt]

\bottomrule
\end{tabularx}
\par\smallskip
\begin{minipage}{\linewidth}
In the table, \(n_N\) is the sample size for fixed-size unit sampling or the number of draws with replacement, with \(n_N/N\to\rho\). Under cluster sampling, \(M_N\) and \(m_N\) are the total and sampled numbers of clusters. Under stratification, \(\mathcal X_j\) are the strata and \(p_j\) is the limiting within-stratum sampling fraction. For a unit with covariates $x$, $\pi^{(1)}(x)$ is the limiting inclusion probability of the PSU containing that unit, and $\pi^{(2)}(x)$ is the limiting conditional inclusion probability of the unit given selection of its PSU. 
\end{minipage}
\end{table}

Using this limit \(\pi(X)\), consider hypothetical Poisson sampling defined by
\[
L=(X,Y)\sim P_{\theta,\eta}^F,\qquad \delta\mid X\sim\operatorname{Bernoulli}(\pi(X)),\qquad \delta\perp Y\mid X.
\]
Let \(\mathbb Q^\pi_{\theta,\eta}\) be the joint law of \((\delta,L)\) and \(\mathbb P^\pi_{\theta,\eta}\) the induced observed-data law. The reference Poisson experiment is defined by \(\{\mathbb P^\pi_{\theta,\eta}\mid(\theta,\eta)\in\Theta\times\Xi\}\). This statistical experiment is used to calculate the efficiency bound and derive the EIF for the original dependent survey designs.

In this i.i.d.\ setting, we construct the observed-data score function and tangent space.
At \((\theta_0,\eta_0)\), we abbreviate the laws as \(\mathbb Q^\pi\) and \(\mathbb P^\pi\), with expectations \(\mathbb E_{\mathbb Q^\pi}\) and \(\mathbb E_{\mathbb P^\pi}\), respectively.
Equip \(L^2_0(P_0^F)\) with \(\langle g_1,g_2\rangle_F=\mathbb E[g_1(L)g_2(L)]\) and induced norm \(\|\cdot\|_F\). For any \(g\in L^2_0(P_0^F)\), let
\[
(\mathcal Lg)(O)=\mathbb E_{\mathbb Q^\pi}[g(L)\mid O]=\delta g(L)+(1-\delta)\mathbb E[g(L)\mid X].
\]
The operator \(\mathcal L\) takes the conditional expectation of a full-data function given the observed data under \(\mathbb Q^\pi\): it retains the function value for sampled units and averages over the unobserved outcome \(Y\) given \(X\) for unsampled units.
In particular, \(\mathcal L\) maps full-data scores to observed-data scores under the reference Poisson experiment (see Figure~\ref{fig:eif-projection} in Section~\ref{sec-lan-efficiency}).
For \(h_1,h_2\in L^2_0(\mathbb P^\pi)\), write \(\langle h_1,h_2\rangle=\mathbb E_{\mathbb P^\pi}[h_1(O)h_2(O)]\), with induced norm \(\|\cdot\|\). For \(h_j=\mathcal Lg_j\) with \(g_j\in L^2_0(P_0^F)\), \(j=1,2\), the inner product and norm are
\begin{align}
\langle h_1,h_2\rangle&=\mathbb E\!\left[\pi(X)\operatorname{Cov}\{g_1(L),g_2(L)\mid X\}+\mathbb E[g_1(L)\mid X]\mathbb E[g_2(L)\mid X]\right],\label{eq-obs-inner-product}\\
\|\mathcal Lg_1\|^2&=\mathbb E\!\left[\pi(X)\operatorname{Var}\{g_1(L)\mid X\}+\{\mathbb E[g_1(L)\mid X]\}^2\right].\label{eq-obs-norm}
\end{align}

A full-data score describes a first-order change of the distribution along a smooth submodel, and applying $\mathcal L$ gives the corresponding observed-data score. The observed-data tangent set collects these scores with the full-data fourth-moment restriction in Assumption~\ref{ass-A1}, and its closure in mean square is the tangent space, which also contains limits of these local directions:
\begin{align}
\mathcal T=\mathcal L\{\dot{\mathcal P}_{P_0^F}\cap L^4(P_0^F)\}\subset L^2_0(\mathbb P^\pi),
\quad
\mathcal H=\overline{\mathcal T}\quad\text{in }L^2_0(\mathbb P^\pi),
\label{eq-obs-tangent-set}
\end{align}
where \(\dot{\mathcal P}_{P_0^F}\) is the full-data tangent set in Assumption~\ref{ass-A1}.
Define \(\mathcal M=\mathcal L^*\mathcal L\) on \(L^2_0(P_0^F)\) by
\(
(\mathcal Ma)(L)=\mathbb E_{\mathbb Q^\pi}[(\mathcal La)(O)\mid L]=\pi(X)a(L)+\{1-\pi(X)\}\mathbb E[a(L)\mid X],
\)
where \(\mathcal L^*\) is the adjoint of \(\mathcal L\). The operator \(\mathcal M\) is used in the equation for the observed-data EIF below.

\subsection{Outline of efficiency theory}
\label{sec-efficiency-overview}
Efficiency theory for $\theta_0$ combines LAN of the experiment, pathwise differentiability of the target, and local regularity of estimators \citep{McNeney2000}. For a smooth full-data submodel $t\mapsto P_{t,g}^F$ with score $g$ and $\mathcal Lg\in\mathcal T$, let $P^N_{t,\mathcal Lg}$ denote its observed-data law. For fixed $u\in\mathbb R$, LAN gives
\begin{equation}
\begin{aligned}
\log\frac{dP^N_{u/\sqrt N,\mathcal Lg}}{dP_0^N}
&=u\Delta_N(\mathcal Lg)-\frac{u^2}{2}\|\mathcal Lg\|^2+o_{P_0^N}(1),\\
\Delta_N(\mathcal Lg)
&=N^{-1/2}\sum_{i=1}^N(\mathcal Lg)(O_i)\xrightarrow[]{d}N(0,\|\mathcal Lg\|^2),
\end{aligned}
\label{eq-lan-overview}
\end{equation}
under $P_0^N$. Theorem~\ref{thm-obs-eif} establishes this expansion with the tangent space $\mathcal H$ and inner product of Section~\ref{sec-tangent-space}. The squared norm $\|\mathcal Lg\|^2$ measures how well the data distinguish local changes along $g$.

Pathwise differentiability gives a continuous linear operator $\dot\psi:\mathcal H\to\mathbb R^p$ for the target's first-order changes. The EIF $\phi\in\mathcal H^p$ uniquely represents this derivative through
\begin{equation}
\dot\psi(\mathcal Lg)=\mathbb E[\phi^F(L)g(L)]=\mathbb E_{\mathbb P^\pi}[\phi(O)(\mathcal Lg)(O)],\qquad g\in\mathcal H^F.
\label{eq-eif-derivative-overview}
\end{equation}
Local regularity requires the $\sqrt N$-scaled error about the perturbed target to have the same limiting law under these changes. The convolution theorem decomposes this limit into a centered Gaussian vector with covariance $\mathbb E_{\mathbb P^\pi}[\phi(O)\phi(O)^\top]$ and independent noise. This covariance is the efficiency bound, attained when the noise vanishes. The bound is therefore determined by $\mathcal H$, its inner product, and $\dot\psi$.

\subsection{Assumptions}
\label{sec-assumptions}
Assumptions~\ref{ass-A1}--\ref{ass-A3} concern the full-data model, the superpopulation target, and the finite-population target, whereas Assumptions~\ref{ass-B1}--\ref{ass-B4} concern the sampling mechanism.

\begin{enumerate}[label=(A\arabic*),ref=A\arabic*,labelsep=0.5em]
\item\label{ass-A1}
\textbf{Full-data DQM and a dense linear tangent set.}
The full-data model is differentiable in quadratic mean (DQM) at \(P_0^F\), with tangent set \(\dot{\mathcal P}_{P_0^F}\) and tangent space \(\mathcal H^F\). Every smooth one-dimensional submodel \(t\mapsto P_{t,g}^F\), \(g\in\dot{\mathcal P}_{P_0^F}\), satisfies \(P_{t,g}^F\ll P_0^F\) for all sufficiently small \(|t|\), with score \(g\) characterized by
\[
\int\left(\sqrt{\frac{dP_{t,g}^F}{dP_0^F}}-1-\frac{t}{2}g\right)^2\,dP_0^F=o(t^2),\qquad t\to0.
\]
Moreover, \(\dot{\mathcal P}_{P_0^F}\cap L^4(P_0^F)\) is a linear subspace of \(L^2_0(P_0^F)\), and
\(
\overline{\dot{\mathcal P}_{P_0^F}\cap L^4(P_0^F)}
=
\mathcal H^F
\)
in \(L^2(P_0^F)\).

\item\label{ass-A2}
\textbf{Pathwise differentiability of the superpopulation target.}
The target map \(P\mapsto\psi(P)\in\mathbb R^p\), with
\(\psi(P_{\theta,\eta}^F)=\theta\), is pathwise differentiable at
\(P_0^F=P_{\theta_0,\eta_0}^F\). Specifically, there exists
\(\phi^F\in(\mathcal H^F)^p\) such that, for every smooth
one-dimensional submodel \(t\mapsto P_{t,g}^F\) through \(P_0^F\)
with score \(g\in\dot{\mathcal P}_{P_0^F}\),
\[
\lim_{t\to0}
\frac{\psi(P_{t,g}^F)-\psi(P_0^F)}{t}
=
\mathbb E[\phi^F(L)g(L)].
\]
The derivative extends continuously and linearly to \(\mathcal H^F\)
as \(\dot\psi^F(g)=\mathbb E[\phi^F(L)g(L)]\).
The function \(\phi^F\), uniquely determined in \((\mathcal H^F)^p\),
is the full-data EIF.

\item\label{ass-A3}
\textbf{Asymptotic linearity of the finite-population target.}
The finite-population target \(\theta_N\) satisfies \eqref{eq-A3-finite-pop-target-expansion}.
\end{enumerate}

\begin{enumerate}[label=(B\arabic*),ref=B\arabic*,labelsep=0.5em]
\item\label{ass-B1}
\textbf{Non-informative sampling.}
The sampling design satisfies \eqref{eq-non-informative}.

\item\label{ass-B2}
\textbf{Uniform convergence of the first-order inclusion probabilities.}
The inclusion probabilities satisfy \eqref{eq-pi-convergence}.

\item\label{ass-B3}
\textbf{Vanishing variance under the sampling design.}
For every measurable function \(h\) with \(\mathbb E[h(X)^2]<\infty\),
\(
\operatorname{Var}\!\left(N^{-1}\sum_{i=1}^N(\delta_i-\pi_{N,i})h(X_i)\,\middle|\,\bm X_N\right)=o_{P_0^N}(1).
\)

\item\label{ass-B4}
\textbf{Sampling positivity.}
There exists a constant \(\pi_->0\) such that 
\(
\pi_-\le \pi(X)\le 1,
\quad\text{a.s.}
\)
\end{enumerate}

The fourth-moment condition in Assumption~\ref{ass-A1} is mainly used in the conditional Lindeberg--Feller argument in the LAN proof. The denseness condition holds in regular parametric models whose scores have finite fourth moments, while bounded mean-zero perturbations provide the required dense class in unrestricted nonparametric models.
Assumption~\ref{ass-A2} allows the first-order sensitivity of the target
to local changes in the full-data distribution to be represented by
the full-data EIF.
Assumption~\ref{ass-A3} is explained in Section~\ref{sec-datasources}.

Assumption~\ref{ass-B1} is automatic when, conditional on $\bm X_N$, the sampling
indicators are generated by auxiliary randomization independent of $\bm Y_N$. Assumptions~\ref{ass-B2} and~\ref{ass-B3} require uniform approximation of first-order inclusion probabilities and vanishing design variance of centered weighted averages. The reference Poisson experiment uses the function \(\pi\) in Assumption~\ref{ass-B2}. Assumption~\ref{ass-B2} may fail, for example, when within-stratum sampling fractions do not converge. Assumption~\ref{ass-B3} may fail under systematic sampling when periodicity in the frame aligns with the sampling interval, so that different random starts yield different limiting averages. Under cluster sampling, it may fail when a cluster contains a non-vanishing fraction of the population: selecting or omitting that cluster can then have a non-vanishing effect on the weighted average. Assumption~\ref{ass-B4} is the standard positivity condition. Table~\ref{tab-sampling-examples} summarizes representative designs satisfying these assumptions, with verification in Appendix~\ref{app-sampling-examples}.

\section{Semiparametric Efficiency}
\label{sec-lan-efficiency}
This section develops the efficiency results outlined in Section~\ref{sec-efficiency-overview}. The following theorem establishes the LAN expansion \eqref{eq-lan-overview} under dependent sampling with the tangent space and inner product of the reference Poisson experiment and derives the EIF characterized by \eqref{eq-eif-derivative-overview} through projection.

\begin{theorem}[Observed-data EIF under dependent sampling]
\label{thm-obs-eif}
Under Assumptions~\ref{ass-A1}--\ref{ass-A2} and~\ref{ass-B1}--\ref{ass-B4}, the observed-data experiment is LAN with tangent space $\mathcal H$ and inner product $\langle\cdot,\cdot\rangle$ of the reference Poisson experiment. The EIF for the superpopulation target is
\[
\phi(O)=(\mathcal L \phi^{*})(O)=\delta \phi^{*}(L)+(1-\delta)\mathbb E[\phi^{*}(L)\mid X],
\]
where $\phi^{*}=(\phi^{*}_1,\ldots,\phi^{*}_p)^\top\in(\mathcal H^F)^p$ is the unique solution of
\begin{equation}
\Proj[\mathcal M \phi^{*}_j\mid\mathcal H^F]=\phi_j^F,
\qquad j=1,\ldots,p.
\label{eq-obs-eif-equation}
\end{equation}
Here $\Proj[\,\cdot\mid\mathcal H^F]$ denotes orthogonal projection onto $\mathcal H^F$ in $L^2_0(P_0^F)$.
\end{theorem}

Theorem~\ref{thm-obs-lan} proves the LAN expansion \eqref{eq-lan-overview} under Assumptions~\ref{ass-A1} and~\ref{ass-B1}--\ref{ass-B3}, including the joint Gaussian convergence of the central sequences with the covariance structure of the reference Poisson experiment. The key to the LAN argument is that \((\mathcal Lg)(O_i)=\mathbb E[g(L_i)\mid X_i]+\delta_i\{g(L_i)-\mathbb E[g(L_i)\mid X_i]\}\), so the sampling indicators enter the central sequence \(\Delta_N(\mathcal Lg)\) in \eqref{eq-lan-overview} only through conditionally centered terms. Under Assumption~\ref{ass-B1} and the i.i.d.\ full-data assumption, these terms are independent and centered conditional on \((\bm X_N,\boldsymbol\delta_N)\). The conditional variance of their normalized sum is \(N^{-1}\sum_{i=1}^N\delta_i\operatorname{Var}\{g(L_i)\mid X_i\}\), whose convergence follows from Assumptions~\ref{ass-B2} and~\ref{ass-B3}. This reduces the effect of sampling dependence to a law of large numbers for the conditional variance, allowing us to combine a conditional central limit theorem for the residual terms with the ordinary i.i.d.\ central limit theorem for the covariate terms. By contrast, inverse-probability-weighted estimators can involve sampling indicators multiplying terms that are not conditionally centered, so their asymptotic variance can depend on the design beyond \(\pi\) \citep{BreslowWellner2007,SaegusaWellner2013}.

Theorem~\ref{thm-pathwise-obs} then establishes the observed-data pathwise derivative, and Proposition~\ref{prop-obs-eif} derives the projection equation through the Riesz representation theorem. Theorem~\ref{thm-obs-eif} justifies the standard missing-at-random projection \citep{RobinsRotnitzkyZhao1994,Tsiatis2006} in the reference Poisson experiment even when the original observed data are dependent. The full-data model restrictions enter through \(\mathcal H^F\), and Figure~\ref{fig:eif-projection} illustrates the geometry of the projection equation \eqref{eq-obs-eif-equation}.

\begin{figure}[!tb]
\centering
\includegraphics[width=\linewidth,pagebox=cropbox]{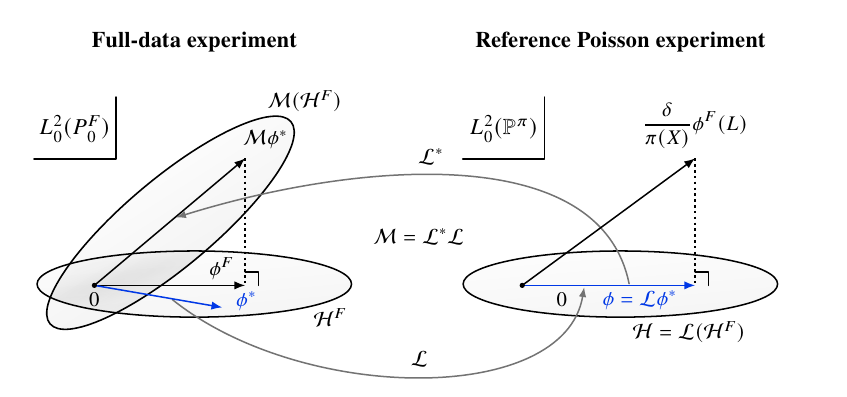}
\caption{Geometry of the projection equation \eqref{eq-obs-eif-equation}, illustrated for a scalar target in a restricted full-data model. The reference Poisson experiment uses the limiting inclusion probability $\pi(X)$ in \eqref{eq-pi-convergence}. On the left, $\phi^F$ is the orthogonal projection of $\mathcal M\phi^*$ onto $\mathcal H^F$. On the right, the observed-data EIF $\phi=\mathcal L\phi^*$ is the orthogonal projection of $\delta\phi^F(L)/\pi(X)$ onto $\mathcal H=\mathcal L(\mathcal H^F)$. The curved arrows show $\phi^*\mapsto\phi$ under $\mathcal L$ and $\phi\mapsto\mathcal M\phi^*$ under $\mathcal L^*$.}
\label{fig:eif-projection}
\end{figure}

Having identified the observed-data EIF, we now characterize efficient estimators through their first-order expansions under superpopulation and finite-population centering. This characterization will be used to establish the efficiency of cross-fitted estimators in Section~\ref{sec-efficient-estimation}. Estimators are required to be measurable with respect to $\mathcal O_N=\sigma(O_1,\ldots,O_N)$. The formal regularity and convolution results are given in Definition~\ref{def-local-regularity} and Theorem~\ref{thm-convolution}.

\begin{theorem}[Efficient estimators]
\label{thm-efficiency}
\label{prop-obs-efficiency-characterization}
Under Assumptions~\ref{ass-A1}--\ref{ass-A2} and~\ref{ass-B1}--\ref{ass-B4}, an estimator $\hat\theta_N$ measurable with respect to the observed data is locally regular and efficient for the superpopulation target if and only if \eqref{eq-best-regular-sp} holds. Under Assumption~\ref{ass-A3}, these conditions are also equivalent to local regularity and efficiency for the finite-population target and to \eqref{eq-best-regular-fp}.
\begin{align}
\sqrt N(\hat\theta_N-\theta_0)
&=\frac1{\sqrt N}\sum_{i=1}^N\phi(O_i)+o_{P_0^N}(1),
\label{eq-best-regular-sp}\\
\sqrt N(\hat\theta_N-\theta_N)
&=\frac1{\sqrt N}\sum_{i=1}^N\{\phi(O_i)-\phi^F(L_i)\}+o_{Q_0^N}(1).
\label{eq-best-regular-fp}
\end{align}
The limiting distributions in \eqref{eq-best-regular-sp} and \eqref{eq-best-regular-fp} are $N_p(0,\Sigma_{\theta_0})$ and $N_p(0,\Sigma_{\theta_N})$, respectively, where the efficiency bounds are
\begin{equation}
\Sigma_{\theta_0}=\mathbb E_{\mathbb P^\pi}[\phi(O)\phi(O)^\top],
\qquad
\Sigma_{\theta_N}=\Sigma_{\theta_0}-\Sigma_F,
\label{eq-cov-identity}
\end{equation}
with $\Sigma_F=\mathbb E[\phi^F(L)\phi^F(L)^\top]$, which is the efficiency bound for estimating $\theta_0$ when all $N$ units are fully observed.
\end{theorem}

\begin{remark}[Finite-population correction as the Pythagorean theorem]
\label{rem-finite-population-correction-Pythagorean-theorem}
The identity \(\Sigma_{\theta_N}=\Sigma_{\theta_0}-\Sigma_F\) can be interpreted as a generalized finite-population correction.
In \(L^2_0(\mathbb Q^\pi)\), the EIF equation \eqref{eq-obs-eif-equation} gives \(\Proj[\phi(O)\mid\mathcal H^F]=\phi^F(L)\) coordinatewise. Hence \(\phi(O)-\phi^F(L)\) is orthogonal to \(\mathcal H^F\), and taking covariance matrices yields the Pythagorean identity in \(L^2_0(\mathbb Q^\pi)\),
\[
\Sigma_{\theta_0}=\Sigma_F+\mathbb E_{\mathbb Q^\pi}\!\left[\{\phi(O)-\phi^F(L)\}\{\phi(O)-\phi^F(L)\}^{\top}\right]
=\Sigma_F+\Sigma_{\theta_N}.
\]
\end{remark}

Thus, the same estimators attain both bounds, and the sampling design enters the bounds only through $\pi(X)$. Appendix~\ref{sec-optimal-design} uses this identity to characterize optimal limiting inclusion probabilities. The following examples illustrate the resulting EIFs and efficiency bounds for mean targets under a nonparametric model and regression coefficients under a conditional-mean model.

\begin{example}[Mean targets in the nonparametric full-data model]
\label{ex-eif-mean}
Consider the superpopulation target \(\mu_0=\mathbb E[Y]\) and finite-population target \(\bar Y_N=N^{-1}\sum_{i=1}^N Y_i\) with \(\mathcal H^F=L^2_0(P_0^F)\). Put \(m_0(X)=\mathbb E[Y\mid X]\). Then \(\phi^F(L)=Y-\mu_0\), and
\begin{align}
\phi(O)&=m_0(X)-\mu_0+\frac{\delta}{\pi(X)}\{Y-m_0(X)\},\label{eq-eif-mean-example}\\
\Sigma_{\theta_0}=\operatorname{Var}\{m_0(X)\}+&\mathbb E\!\left[\frac{\operatorname{Var}(Y\mid X)}{\pi(X)}\right],\quad
\Sigma_{\theta_N}=\mathbb E\!\left[\left\{\frac1{\pi(X)}-1\right\}\operatorname{Var}(Y\mid X)\right].\notag
\end{align}
Equation \eqref{eq-cov-identity} holds because
\(
\Sigma_{\theta_0} - \Sigma_{\theta_N}
=
\operatorname{Var}(Y)
=
\Sigma_F.
\)

On the \(\sqrt N\) scale, the anticipated-variance lower bound of \citet{GodambeJoshi1965} for \(\bar Y_N\) is \(N^{-1}\sum_{i=1}^N(1/\pi_{N,i}-1)\operatorname{Var}(Y_i\mid X_i)\). Under Assumptions~\ref{ass-B2} and~\ref{ass-B4}, it converges to \(\Sigma_{\theta_N}\), so Theorem~\ref{thm-efficiency} recovers its large-sample limit.
\end{example}

\begin{example}[Conditional-mean regression]
\label{ex-eif-regression}
Suppose that \(Y\) is scalar and the full-data model is otherwise nonparametric subject to \(\mathbb E[Y\mid X]=\mu(X,\beta_0)\), with the superpopulation target \(\beta_0\in\mathbb R^p\). Put \(\varepsilon=Y-\mu(X,\beta_0)\), \(\varepsilon_i=Y_i-\mu(X_i,\beta_0)\), and \(\sigma^2(X)=\mathbb E[\varepsilon^2\mid X]\). Let \(\dot\mu_\beta(X)=\partial\mu(X,\beta)/\partial\beta\) and \(\ddot\mu_\beta(X)=\partial\dot\mu_\beta(X)/\partial\beta^\top\), with subscript \(0\) denoting evaluation at \(\beta_0\). Assume \(0<\sigma^2(X)<\infty\) a.s. Define
\begin{equation}
I_F=\mathbb E\!\left[\frac{\dot\mu_0(X)\dot\mu_0(X)^\top}{\sigma^2(X)}\right],\qquad I_\pi=\mathbb E\!\left[\pi(X)\frac{\dot\mu_0(X)\dot\mu_0(X)^\top}{\sigma^2(X)}\right].
\label{eq-regression-information}
\end{equation}

Assume these matrices are finite and nonsingular and that \(\mathcal H^F\) consists exactly of the functions \(g\in L^2_0(P_0^F)\) satisfying \(\mathbb E[\varepsilon g(L)\mid X]=\dot\mu_0(X)^\top\dot\beta_g\) for some \(\dot\beta_g\in\mathbb R^p\), with \(\dot\psi^F(g)=\dot\beta_g\). Then
\begin{equation}
\phi^F(L)=I_F^{-1}\frac{\dot\mu_0(X)}{\sigma^2(X)}\varepsilon,
\qquad
\phi(O)=I_\pi^{-1}\frac{\delta\dot\mu_0(X)}{\sigma^2(X)}\varepsilon,
\label{eq-regression-eifs}
\end{equation}
and \(\Sigma_F=I_F^{-1}\), \(\Sigma_{\theta_0}=I_\pi^{-1}\), and \(\Sigma_{\theta_N}=I_\pi^{-1}-I_F^{-1}\). The finite-population target \(\beta_N\) is a solution, in a neighborhood of \(\beta_0\), of
\begin{equation}
\frac1N\sum_{i=1}^N\frac{\dot\mu_{\beta_N}(X_i)}{\sigma^2(X_i)}\{Y_i-\mu(X_i,\beta_N)\}=0.
\label{eq-regression-finite-population-target}
\end{equation}
The target \(\beta_N\) is the coefficient that an oracle observing the entire population and knowing the true conditional variance function \(\sigma^2(\cdot)\) would compute, although this function remains an unknown nuisance in the full-data model. In the local-alternative analysis, the function \(\sigma^2(\cdot)\) defining \(\beta_N\) is held fixed at its value under \(P_0^F\). Under homoskedasticity, the target reduces to the corresponding census least-squares coefficient and, for a linear mean, to census OLS. Under Condition~\ref{cond-regression-regularity} in Section~\ref{sec-regression-target-estimation}, \(\beta_N\) satisfies Assumption~\ref{ass-A3}, so Theorem~\ref{thm-efficiency} applies with these bounds.
\end{example}

\section{Efficient Estimation}
\label{sec-efficient-estimation}

We first give high-level sufficient conditions for cross-fitted estimators of a general target to attain the efficiency bounds and consistently estimate their asymptotic variances. We then verify these conditions for mean and regression estimators under the sampling designs considered here. Appendix~\ref{app-efficient-estimation} gives the proofs and design-specific conditions.

Write \(\phi_i=\phi(O_i)\) and \(\phi_i^F=\phi^F(L_i)\). Fix \(K\ge2\) and a partition \(\{1,\ldots,N\}=I_1\sqcup\cdots\sqcup I_K\), with \(\mathcal I_N=\sigma(I_1,\ldots,I_K)\subseteq\sigma(\bm X_N)\). Let \(k(i)\) denote the fold containing unit \(i\). The information available for nuisance fitting outside \(I_k\) is \(\mathcal T_{N,k}=\sigma\!\left(\mathcal I_N,\bm X_N,\{(\delta_j,\delta_jY_j)\mid j\in I_k^c\}\right)\). Let \(\hat\phi_{i}(\theta)\) denote the estimating contribution formed from the EIF, evaluated at \(O_i\) and a candidate target value \(\theta\in\Theta\), with nuisance functions fitted using \(\mathcal T_{N,k(i)}\). We use the estimating map \(\hat\Psi_N(\theta)=N^{-1}\sum_{i=1}^N\hat\phi_{i}(\theta)\).

\begin{remark}[Randomized folds]
\label{rem-randomized-folds}
An independent parameter-free seed \(\zeta_N\perp(L_1,\ldots,L_N)\) may generate folds or PSU partitions before sampling. We adjoin it to \(\sigma(\bm X_N)\) and \(\mathcal O_N\), retaining the same notation. All laws and Assumptions~\ref{ass-B1}--\ref{ass-B4} then refer to this augmented experiment.
\end{remark}

\subsection{General results}
\label{sec-efficient-general-result}

For variance estimation, let \(\hat\phi_i^F\) be the estimated full-data EIF on sampled units. We use the following conditions.
\begin{enumerate}[label=(C\arabic*),ref=C\arabic*,nosep,labelsep=0.5em]
\item\label{cond-general-root}
\textbf{Uniform convergence and identification.}
There is a convex neighborhood \(U\subset\Theta\) of \(\theta_0\) such that \(P_0^N(\hat\theta_N\in U)\to1\) and \(\|\hat\Psi_N(\hat\theta_N)\|=o_{P_0^N}(N^{-1/2})\). For a deterministic map \(\Psi:U\to\mathbb R^p\) with \(\Psi(\theta_0)=0\), \(\sup_{\theta\in U}\|\hat\Psi_N(\theta)-\Psi(\theta)\|=o_{P_0^N}(1)\), and \(\inf_{\theta\in U,\ \|\theta-\theta_0\|\ge\epsilon}\|\Psi(\theta)\|>0\) for every \(\epsilon>0\).

\item\label{cond-general-jacobian}
\textbf{Convergence of gradients.}
Each \(\hat\phi_{i}\) is differentiable on \(U\). For every deterministic \(r_N\downarrow0\),
\(
\sup_{\theta\in U,\ \|\theta-\theta_0\|\le r_N}
\left\|
N^{-1}\sum_{i=1}^N
\frac{\partial}{\partial\theta^\top}\hat\phi_{i}(\theta)
+I_p
\right\|
=o_{P_0^N}(1).
\)

\item\label{cond-general-score}
\textbf{First-order effect of nuisance estimation.}
\(N^{-1/2}\sum_{i=1}^N\{\hat\phi_{i}(\theta_0)-\phi_i\}=o_{P_0^N}(1)\).

\item\label{cond-general-variance}
\textbf{Mean-square consistency of the estimated EIFs.}
\(N^{-1}\sum_{i=1}^N\|\hat\phi_{i}(\hat\theta_N)-\phi_i\|^2=o_{P_0^N}(1)\) and \(N^{-1}\sum_{i=1}^N(\delta_i/\pi_{N,i})\|\hat\phi_i^F-\phi_i^F\|^2=o_{P_0^N}(1)\).
\end{enumerate}
Conditions~\ref{cond-general-root} and~\ref{cond-general-jacobian} are familiar regularity conditions for estimating equations: the former ensures consistency of the approximate root, and the latter permits a local linear expansion with limiting Jacobian \(-I_p\).
Condition~\ref{cond-general-score} depends on the EIF and the relation between training data and validation-fold sampling indicators.
Condition~\ref{cond-general-variance} is needed only for variance estimation.

Put \(\hat\phi_{i}=\hat\phi_{i}(\hat\theta_N)\). Estimate \(\Sigma_{\theta_0}\), \(\Sigma_F\), and \(\Sigma_{\theta_N}\) by
\begin{align}
\hat\Sigma_{\theta_0,N}
=\frac1N\sum_{i=1}^N\hat\phi_{i}\{\hat\phi_{i}\}^\top, \quad
\hat\Sigma_{F,N}
=\frac1N\sum_{i=1}^N\frac{\delta_i}{\pi_{N,i}}\hat\phi_i^F\{\hat\phi_i^F\}^\top
\label{eq-general-variance-estimators}
\end{align}
and
\(
\widetilde\Sigma_{\theta_N,N}
=\hat\Sigma_{\theta_0,N}-\hat\Sigma_{F,N},
\)
respectively. 
To obtain a positive semidefinite estimator of \(\Sigma_{\theta_N}\) in general, we can truncate any negative eigenvalues of \(\widetilde\Sigma_{\theta_N,N}\) at zero without affecting consistency. For means, Section~\ref{sec-mean-target-estimation} gives a nonnegative estimator.

\begin{theorem}[A general theorem for cross-fitted estimators]
\label{thm-efficient-estimation}
Under Assumptions~\ref{ass-A1}--\ref{ass-A2} and~\ref{ass-B1}--\ref{ass-B4}, the following hold.

\noindent\textup{(a)} Conditions~\ref{cond-general-root}--\ref{cond-general-score} imply the expansion \eqref{eq-best-regular-sp} and, under Assumption~\ref{ass-A3}, \eqref{eq-best-regular-fp}.
Thus \(\hat\theta_N\) is locally regular and efficient for the superpopulation target and, under Assumption~\ref{ass-A3}, for the finite-population target.

\noindent\textup{(b)} Condition~\ref{cond-general-variance} implies \(\hat\Sigma_{\theta_0,N}\xrightarrow[]{p}\Sigma_{\theta_0}\) and, under Assumption~\ref{ass-A3}, \(\widetilde\Sigma_{\theta_N,N}\xrightarrow[]{p}\Sigma_{\theta_N}\).
\end{theorem}

\subsection{Mean targets}
\label{sec-mean-target-estimation}

We work under the setting in Example~\ref{ex-eif-mean}. We fit \(\hat m_{-k}\) outside \(I_k\). Set \(\hat\phi_{i}(\mu)=\hat m_{-k(i)}(X_i)-\mu+(\delta_i/\pi_{N,i})\{Y_i-\hat m_{-k(i)}(X_i)\}\), whose empirical mean has root
\begin{equation}
\hat\mu_N
=\frac1N\sum_{i=1}^N\left[
\hat m_{-k(i)}(X_i)+\frac{\delta_i}{\pi_{N,i}}\{Y_i-\hat m_{-k(i)}(X_i)\}
\right].
\label{eq-mean-estimator-appendix}
\end{equation}
Under dependent sampling, training outside a fold need not make \(\hat m_{-k}\) independent of that fold's sampling indicators. We use the following conditions for mean estimation with cross-fitting. 
\begin{enumerate}[label=(M\arabic*),ref=M\arabic*,nosep,labelsep=0.5em,leftmargin=*,widest=1-G,align=left]
\item\label{cond-mean-regression}
\textbf{Mean-square consistency.}
For \(i\in I_k\), \(\hat m_{-k}(X_i)\) is \(\mathcal T_{N,k}\)-measurable. Moreover, \(N^{-1}\sum_{i=1}^N\{\hat m_{-k(i)}(X_i)-m_0(X_i)\}^2=o_{P_0^N}(1)\).

\item\label{cond-mean-sampling-term}
\textbf{Dependent remainder control.}
\(N^{-1/2}\sum_{i=1}^N(\delta_i-\pi_{N,i})\{\hat m_{-k(i)}(X_i)-m_0(X_i)\}/\pi_{N,i}=o_{P_0^N}(1)\).

\renewcommand{\theenumi}{M1-G}
\renewcommand{\labelenumi}{(\theenumi)}
\item\label{cond-mean-group-regression}
\textbf{Errors summed within groups.}
\(N^{-1}\sum_G[\sum_{i\in G}\{\hat m_{-k(i)}(X_i)-m_0(X_i)\}]^2=o_{P_0^N}(1)\), where \(G\) ranges over clusters or PSUs.
\end{enumerate}
Known inclusion probabilities remove the rate requirement on the product of the estimation errors for the inclusion probabilities and the regression function, but Condition~\ref{cond-mean-sampling-term} can still impose design-specific rates.
The superpopulation variance estimator specializes \eqref{eq-general-variance-estimators}, while the finite-population variance estimator below takes a residual-based form that is nonnegative by construction. With \(\hat\phi_{i}=\hat\phi_{i}(\hat\mu_N)\), define
\begin{align}
\hat\Sigma_{\theta_0,N}
&=\frac1N\sum_{i=1}^N\{\hat\phi_{i}\}^2,
\label{eq-sp-var-estimator}\\
\hat\Sigma_{\theta_N,N}
&=\frac1N\sum_{i=1}^N\frac{\delta_i}{\pi_{N,i}}
\left(\frac1{\pi_{N,i}}-1\right)
\{Y_i-\hat m_{-k(i)}(X_i)\}^2.
\label{eq-fp-var-estimator-mean-positive}
\end{align}

The following results specialize Theorem~\ref{thm-efficient-estimation} to the cross-fitted mean estimator.

\begin{theorem}[Efficient estimation of the mean]
\label{thm-mean-efficient-estimation}
Under Assumptions~\ref{ass-A1}--\ref{ass-A2} and~\ref{ass-B1}--\ref{ass-B4}, the following hold.

\noindent\textup{(a)} Conditions~\ref{cond-mean-regression}--\ref{cond-mean-sampling-term} imply the two expansions in Theorem~\ref{thm-efficiency}, with the EIFs in Example~\ref{ex-eif-mean}. Hence \(\hat\mu_N\) is locally regular and efficient for \(\mu_0\) and \(\bar Y_N\).

\noindent\textup{(b)} Condition~\ref{cond-mean-regression} implies \(\hat\Sigma_{\theta_0,N}\xrightarrow[]{p}\Sigma_{\theta_0}\) and \(\hat\Sigma_{\theta_N,N}\xrightarrow[]{p}\Sigma_{\theta_N}\).
\end{theorem}

We next give design-specific sufficient conditions for Condition~\ref{cond-mean-sampling-term}, which controls the interaction between nuisance estimation and dependent sampling.

\begin{theorem}[Efficiency under the sampling designs]
\label{thm-mean-sampling-designs}
Under Assumptions~\ref{ass-A1}--\ref{ass-A2} and the corresponding sampling and fold conditions specified in Appendices~\ref{app-sampling-examples} and~\ref{app-mean-design-results}, the following hold.
\begin{enumerate}[label=\textup{(\roman*)},nosep,labelsep=0.5em]
\item \textbf{Non-positive individual correlations.} For Poisson sampling, SRSWOR, stratified sampling, PPS sampling with replacement, and rejective sampling, Condition~\ref{cond-mean-regression} suffices for~\ref{cond-mean-sampling-term}.
\item \textbf{Non-positive group correlations.} For cluster sampling and stratified two-stage sampling with Poisson or rejective selection of PSUs, Conditions~\ref{cond-mean-regression} and~\ref{cond-mean-group-regression} suffice for~\ref{cond-mean-sampling-term}.
\end{enumerate}
In either case, all conclusions of Theorem~\ref{thm-mean-efficient-estimation} hold.
\end{theorem}

\begin{proposition}[Sufficient rates under cluster sampling and stratified two-stage sampling]
\label{prop-cluster-two-stage-regression-rates}
Let \(G\) range over the clusters or first-stage PSUs, and suppose that \(\hat m_{-k}(X_i)\) is \(\mathcal T_{N,k}\)-measurable for \(i\in I_k\). Conditions~\ref{cond-mean-regression} and~\ref{cond-mean-group-regression} hold if
\begin{equation}
\max_G|G|\,\frac1N\sum_{i=1}^N\{\hat m_{-k(i)}(X_i)-m_0(X_i)\}^2=o_{P_0^N}(1).
\label{eq-group-size-regression-rate}
\end{equation}
\end{proposition}

Thus an empirical root-mean-square error (RMSE) of \(O_p(N^{-1/4})\) suffices when \(\max_G|G|=o_p(\sqrt N)\), a stronger restriction than the \(o_p(N)\) condition used for the bounds. 

\subsection{Regression targets}
\label{sec-regression-target-estimation}

We work under the setting in Example~\ref{ex-eif-regression}. Unlike the mean estimator, the regression estimator is defined implicitly by an estimating equation with estimated inverse-variance weights. We therefore impose smoothness and identification conditions in addition to consistency of the variance estimator. We fit \(\hat\sigma^2_{-k}\) outside \(I_k\), and let \(\hat\beta_N\) be an approximate root of
\begin{equation}
\hat U_N(\beta)
=\frac1N\sum_{i=1}^N\delta_i
\frac{\dot\mu_\beta(X_i)}{\hat\sigma^2_{-k(i)}(X_i)}
\{Y_i-\mu(X_i,\beta)\}.
\label{eq-regression-estimating-equation}
\end{equation}
We impose the following conditions on the regression model and the conditional-variance estimators.
\begin{enumerate}[label=(R\arabic*),ref=R\arabic*,nosep,labelsep=0.5em]
\item\label{cond-regression-regularity}
\textbf{Smoothness and moments.}
There is a compact convex neighborhood \(U\) of \(\beta_0\) on which \(\beta\mapsto\mu(x,\beta)\) is twice continuously differentiable. An envelope \(b(X)\) satisfies \(\sup_{\beta\in U}\{|\mu(X,\beta)|+\|\dot\mu_\beta(X)\|+\|\ddot\mu_\beta(X)\|\}\le b(X)\) and \(\mathbb E[Y^4+b(X)^4]<\infty\). For constants \(0<\sigma_-^2\le\sigma_+^2<\infty\), \(\sigma_-^2\le\sigma^2(X)\le\sigma_+^2\) a.s. The matrices \(I_F\) and \(I_\pi\) in \eqref{eq-regression-information} are positive definite. The selected solution \(\beta_N\) in \eqref{eq-regression-finite-population-target} exists with probability tending to one and satisfies \(\beta_N\xrightarrow[]{p}\beta_0\) under \(Q_0^N\).

\item\label{cond-regression-variance-estimation}
\textbf{Cross-fitted conditional-variance estimation.}
For \(i\in I_k\), \(\hat\sigma^2_{-k}(X_i)\) is \(\mathcal T_{N,k}\)-measurable. For constants \(0<c<C<\infty\) with \(c\le\sigma_-^2\le\sigma_+^2\le C\), the bounds \(c\le\hat\sigma^2_{-k(i)}(X_i)\le C\), \(1\le i\le N\), hold with probability tending to one, and, with \(b(X)\) denoting the envelope in Condition~\ref{cond-regression-regularity},
\begin{equation}
\frac1N\sum_{i=1}^N b(X_i)^2\{\hat\sigma^2_{-k(i)}(X_i)-\sigma^2(X_i)\}^2=o_{P_0^N}(1).
\label{eq-regression-variance-L2}
\end{equation}
The bounds can be enforced by truncation.

\item\label{cond-regression-root}
\textbf{Approximate root and identification.}
Define \(\Psi_R(\beta)=\mathbb E[\pi(X)\dot\mu_\beta(X)\{\mu(X,\beta_0)-\mu(X,\beta)\}/\sigma^2(X)]\). The estimator satisfies \(P_0^N(\hat\beta_N\in U)\to1\) and \(\|\hat U_N(\hat\beta_N)\|=o_{P_0^N}(N^{-1/2})\). For every \(\epsilon>0\), \(\inf_{\beta\in U,\ \|\beta-\beta_0\|\ge\epsilon}\|\Psi_R(\beta)\|>0\).
\end{enumerate}

Define the empirical information matrices
\begin{equation}
\begin{aligned}
\hat I_{\pi,N}
=\frac1N\sum_{i=1}^N\delta_i
\frac{\dot\mu_{\hat\beta_N}(X_i)\dot\mu_{\hat\beta_N}(X_i)^\top}
{\hat\sigma^2_{-k(i)}(X_i)},
\quad
\hat I_{F,N}
=\frac1N\sum_{i=1}^N
\frac{\dot\mu_{\hat\beta_N}(X_i)\dot\mu_{\hat\beta_N}(X_i)^\top}
{\hat\sigma^2_{-k(i)}(X_i)}.
\end{aligned}
\label{eq-regression-full-information-estimator}
\end{equation}
The estimated EIF is
\begin{equation}
\hat\phi_{i}(\beta)
=\hat I_{\pi,N}^{-1}\delta_i
\frac{\dot\mu_\beta(X_i)}{\hat\sigma^2_{-k(i)}(X_i)}
\{Y_i-\mu(X_i,\beta)\}.
\label{eq-regression-estimated-eif}
\end{equation}
The variance estimators are \(\hat\Sigma_{\theta_0,N}=\hat I_{\pi,N}^{-1}\) and \(\hat\Sigma_{\theta_N,N}=\hat I_{\pi,N}^{-1}-\hat I_{F,N}^{-1}\). When both information matrices are positive definite, \(\hat I_{\pi,N}\preceq\hat I_{F,N}\) makes the finite-population variance estimator positive semidefinite.

\begin{theorem}[Efficient estimation of a regression target]
\label{thm-regression-efficient-estimation}
Under Assumptions~\ref{ass-A1}--\ref{ass-A2} and~\ref{ass-B1}--\ref{ass-B4}, together with Conditions~\ref{cond-regression-regularity}--\ref{cond-regression-root}, the following hold.

\noindent\textup{(a)} The estimator \(\hat\beta_N\) satisfies the two expansions in Theorem~\ref{thm-efficiency}, with the EIFs in Example~\ref{ex-eif-regression}. Hence it is locally regular and efficient for the superpopulation target \(\beta_0\) and the finite-population target \(\beta_N\).

\noindent\textup{(b)} The variance estimators satisfy \(\hat\Sigma_{\theta_0,N}\xrightarrow[]{p}I_\pi^{-1}\) and \(\hat\Sigma_{\theta_N,N}\xrightarrow[]{p}I_\pi^{-1}-I_F^{-1}\).
\end{theorem}

Conditional on \(\mathcal T_{N,k}\) and \(\boldsymbol\delta_N\), validation-fold residuals remain independent and centered. No design-specific remainder verification analogous to Theorem~\ref{thm-mean-sampling-designs} is therefore needed. For each design in Table~\ref{tab-sampling-examples}, Conditions~\ref{cond-regression-regularity}--\ref{cond-regression-root} suffice by Theorem~\ref{thm-regression-efficient-estimation} and the design verifications in Appendix~\ref{app-sampling-examples}.

\section{Numerical Experiments}
\label{sec-num_exp}

The experiments address three questions: whether the first-order behavior of the efficient estimator depends on the design only through the limiting inclusion probabilities (Section~\ref{sec-num-design-invariance}), whether the group-size condition of Proposition~\ref{prop-cluster-two-stage-regression-rates} matters in finite samples under cluster sampling (Section~\ref{sec-num-group-size}), and when the efficiency gain over conventional estimators is substantively large (Section~\ref{sec-num-regression}). Appendix~\ref{app-details-sections-five-six} gives implementation details.

For $B$ Monte Carlo repetitions, let $\theta_b$ be the superpopulation or realized finite-population target, and let $\Sigma_{\mathrm{bound},b}$ be the corresponding empirical reference variance on the $\sqrt N$ scale, formed from the generating conditional moments and the exact inclusion probabilities. The empirical SD is the sample standard deviation of the $B$ values $\sqrt N(\hat\theta_b-\theta_b)$, with divisor $B-1$ in the sample variance. We report
\begin{equation}
\mathrm{SD/bound}
=
\frac{\text{empirical SD}}
{\{B^{-1}\sum_{b=1}^B\Sigma_{\mathrm{bound},b}\}^{1/2}}.
\label{eq-standardized-simulation-error}
\end{equation}
Coverage is the fraction of intervals $\hat\theta_b\pm1.96\,\widehat{\mathrm{SE}}_b$ containing $\theta_b$, with an estimated standard error. For an efficient asymptotically normal estimator, the SD ratio should approach one and coverage should approach 0.95. The Monte Carlo standard error of an SD ratio near one is about $(2B)^{-1/2}$, that is, 0.011 at $B=4{,}000$ and 0.022 at $B=1{,}000$.

\subsection{Design invariance}
\label{sec-num-design-invariance}

We use $N=10{,}000$. Stratum membership $H\in\{1,\ldots,5\}$ is generated independently with probabilities $(0.10,0.15,0.20,0.25,0.30)$, so that the realized stratum counts are random. Given $H=h$, let $X_1=\mu_h+U_1$ with $U_1\sim N(0,1)$ and $(\mu_1,\ldots,\mu_5)=(-1,-0.5,0,0.5,1)$, and let $X_2\sim\operatorname{Unif}(-1,1)$ independently. With $X=(H,X_1,X_2)$, the outcome is conditionally Gaussian with
\[
m_0(X)=X_1-\tfrac14+0.75(X_2^2-1/3),
\qquad
\Var(Y\mid X)=\exp(0.35U_1+0.15X_2),
\]
so that $\mu_0=0$. The limiting inclusion probabilities are $\pi(X)=\rho_H$ with $(\rho_1,\ldots,\rho_5)=(0.20,0.20,0.25,0.28,0.30)$.

Four designs share these limits: Poisson sampling with probability $\rho_h$ in stratum $h$; stratified SRSWOR; one-stage cluster sampling, in which whole PSUs are selected by SRSWOR within strata; and the self-weighted two-stage design of Corollary~\ref{cor-self-weighted-two-stage} with Poisson sampling of PSUs, in which a PSU is selected with probability proportional to its size and ten of its units are then drawn by SRSWOR, so that the ultimate inclusion probability is $\rho_h$. One-stage cluster sampling is the case $\pi^{(2)}\equiv1$ of stratified two-stage sampling in Table~\ref{tab-sampling-examples}. Each stratum is partitioned into PSUs of about 25 units, either at random, so that PSU membership is unrelated to the covariates, or by sorting the units on $U_1$ and grouping consecutive units, so that PSUs are homogeneous in a covariate that predicts the outcome, as geographic clusters typically are.

We compare the cross-fitted estimator \eqref{eq-mean-estimator-appendix} with the stratified H\'ajek estimator $\sum_h(N_{N,h}/N)\bar Y_{s,h}$, where $N_{N,h}$ is the population size of stratum $h$ and $\bar Y_{s,h}$ is the mean of the sampled outcomes in that stratum; it uses the design weights but no outcome regression. The regression $\hat m_{-k}$ is a weighted least-squares fit on an intercept, stratum indicators, $U_1$, $X_2$, and $X_2^2$, which contains the true mean, with five folds of whole PSUs under all designs.

\begin{table}[!tb]
\centering
\small
\setlength{\tabcolsep}{2.7pt}
\caption{Design invariance for the mean, $N=10{,}000$ and 4,000 Monte Carlo repetitions. Entries are SD/bound; parentheses give estimated-SE 95\% Wald coverage for the cross-fitted EIF. The designs share the same limiting inclusion probabilities. PSUs are random groups or consecutive blocks in $U_1$. For Poisson and stratified sampling, random-PSU results are shown because PSUs only define the folds. SP and FP denote superpopulation and finite-population targets.}
\label{tab:fixed-strata-design-invariance}
\begin{tabular*}{\linewidth}{@{\extracolsep{\fill}}llrrrr@{}}
\toprule
& & \multicolumn{2}{c}{Stratified H\'ajek} & \multicolumn{2}{c}{Cross-fitted EIF} \\
\cmidrule(lr){3-4}\cmidrule(l){5-6}
Design & PSUs & SP & FP & SP & FP \\
\midrule
Stratum-specific Poisson & --- & 1.245 & 1.407 & 1.015 (0.950) & 1.014 (0.948) \\
Stratified SRSWOR & --- & 1.228 & 1.387 & 1.016 (0.945) & 1.003 (0.954) \\
One-stage cluster & random & 1.235 & 1.381 & 1.020 (0.947) & 1.000 (0.953) \\
One-stage cluster & ordered & 3.758 & 4.966 & 1.043 (0.940) & 1.036 (0.940) \\
Two-stage self-weighted & random & 1.242 & 1.422 & 1.001 (0.950) & 1.013 (0.945) \\
Two-stage self-weighted & ordered & 1.911 & 2.419 & 1.000 (0.950) & 1.010 (0.944) \\
\bottomrule
\end{tabular*}
\end{table}

In Table~\ref{tab:fixed-strata-design-invariance}, the SD ratio of the cross-fitted estimator lies between 1.000 and 1.043 and its coverage between 0.940 and 0.954 in all configurations. The H\'ajek estimator behaves differently. With random PSUs its SD ratio for the finite-population mean is between 1.381 and 1.422 under all four designs, because cluster and two-stage samples of random PSUs resemble samples of individual units. With covariate-ordered PSUs it rises to 4.966 under one-stage cluster sampling and 2.419 under two-stage sampling, while the cross-fitted estimator is unaffected. The design thus affects an estimator that is not efficient beyond $\pi$, as noted after Theorem~\ref{thm-obs-eif}, whereas the first-order behavior of the efficient estimator is governed by the limiting inclusion probabilities under the joint superpopulation--design law. These results do not assert equality of design-conditional variances for a fixed population. The same pattern holds for the regression estimator of Section~\ref{sec-regression-target-estimation}: for a linear conditional mean with a bounded heteroskedastic variance, the SD ratio of the feasible generalized least squares (GLS) estimator of a slope coefficient lies between 0.983 and 1.046 and its coverage between 0.939 and 0.956 across the same designs, PSU constructions, and targets (Appendix~\ref{app-fixed-strata-implementation}).

\subsection{Group-size condition under cluster sampling}
\label{sec-num-group-size}

For designs that sample groups, Theorem~\ref{thm-mean-sampling-designs} requires Condition~\ref{cond-mean-group-regression} in addition to mean-square consistency of $\hat m_{-k}$, and by Proposition~\ref{prop-cluster-two-stage-regression-rates} it suffices that $\max_G|G|$ times the empirical mean-square error vanishes. To examine the finite-sample role of this product, we keep the model and inclusion probabilities of Section~\ref{sec-num-design-invariance} and compare stratified SRSWOR of units with one-stage cluster sampling of covariate-ordered PSUs, using the same whole-PSU folds. The PSU size is 5, 10, 25, or 50 for $N=10{,}000$ and 10, 25, 50, or 100 for $N=40{,}000$. Besides the low-dimensional regression of Section~\ref{sec-num-design-invariance}, we use a rougher learner that replaces the linear term in $U_1$ by a penalized 40-knot linear spline; its mean-square error is at least six times larger.

\begin{figure}[!tb]
\centering
\includegraphics[width=\linewidth]{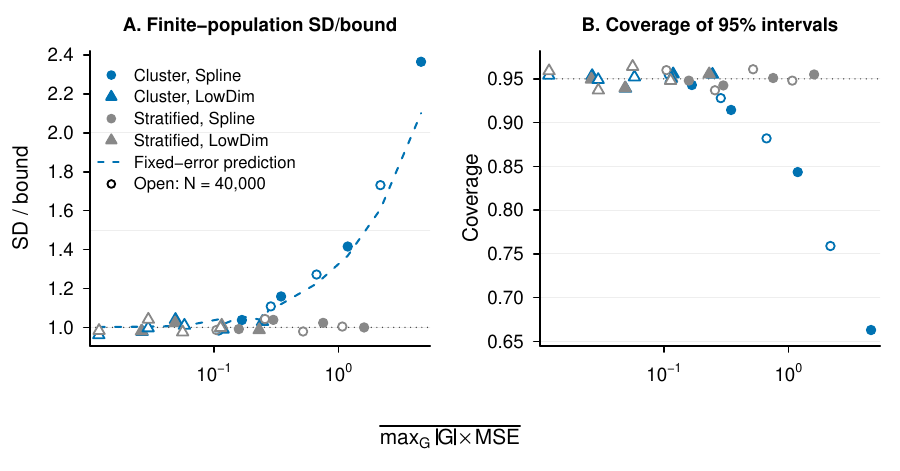}
\caption{Group-size condition. SD ratio \eqref{eq-standardized-simulation-error} (Panel A) and coverage of nominal 95\% intervals (Panel B) of the cross-fitted estimator for the finite-population mean, against $\max_G|G|$ times the empirical mean-square error of $\hat m_{-k}$. Blue symbols are one-stage cluster sampling of covariate-ordered PSUs and gray symbols are stratified SRSWOR with the same folds. Circles are the spline learner and triangles the low-dimensional learner. Filled symbols are $N=10{,}000$ (2,000 repetitions) and open symbols $N=40{,}000$ (1,000 repetitions). The dashed curve is the prediction from the design variance of the nuisance remainder with the fitted errors held fixed.}
\label{fig-group-size}
\end{figure}

Figure~\ref{fig-group-size} plots the SD ratio and the coverage for the finite-population mean against the left side of \eqref{eq-group-size-regression-rate}. Under stratified sampling the SD ratio stays between 0.975 and 1.043 and coverage between 0.937 and 0.964 for both learners and all PSU sizes, in line with Theorem~\ref{thm-mean-sampling-designs}(i). Under cluster sampling the results for both learners and both population sizes fall close to a single curve: the SD ratio is at most 1.040 when the product is below 0.25, which includes all configurations with the low-dimensional learner, and it increases to 2.366, with coverage 0.663, when the product is 4.6. For a fixed PSU size, a larger population reduces the regression error and moves the SD ratio toward one: with PSUs of 25 units and the spline learner, it falls from 1.416 at $N=10{,}000$ to 1.108 at $N=40{,}000$. The dashed curve is obtained from the design variance of the remainder in Condition~\ref{cond-mean-sampling-term} with the cross-fitted errors held fixed, which is the group-sum quantity of Condition~\ref{cond-mean-group-regression} weighted by the design (Appendix~\ref{app-fixed-strata-implementation}); it follows the observed ratios except at the largest values, where it is too small. Thus, with covariate-homogeneous PSUs, the product in the sufficient condition of Proposition~\ref{prop-cluster-two-stage-regression-rates} tracks finite-sample performance under cluster sampling, whereas no such dependence on the group size appears when units are sampled individually.

\subsection{When are efficiency gains large?}
\label{sec-num-regression}

Two experiments with $N=8{,}000$, Poisson sampling, and 5,000 repetitions per setting indicate when the gain over conventional estimators is large. For the mean, let $Y=\sqrt{R^2}\,f(X)+\sqrt{1-R^2}\,\varepsilon$ with $\varepsilon\sim N(0,1)$, where $f$ is a standardized, partly quadratic function of two uniform covariates (Appendix~\ref{app-efficiency-gain-implementation}), $R^2\in\{0,0.25,0.50,0.75,0.90\}$, and the inclusion probability is a constant $\rho\in\{0.10,0.25,0.50\}$. The Horvitz--Thompson (HT) estimator has asymptotic variance $1/\rho$ for the superpopulation mean and $(1-\rho)/\rho$ for the finite-population mean, and the efficiency bounds are $R^2+(1-R^2)/\rho$ and $(1-\rho)(1-R^2)/\rho$. The theoretical relative variance reductions are therefore
\begin{equation}
(1-\rho)R^2
\quad\text{for the superpopulation target},
\qquad
R^2
\quad\text{for the finite-population target}.
\label{eq-mean-efficiency-gain-formula}
\end{equation}
The absolute reduction is the same for both targets; the relative reduction is smaller for the superpopulation target because $\Sigma_F=\Var(Y)=1$ is common to both of its variances, as in Remark~\ref{rem-finite-population-correction-Pythagorean-theorem}.

\begin{figure}[!tb]
\centering
\includegraphics[width=\linewidth]{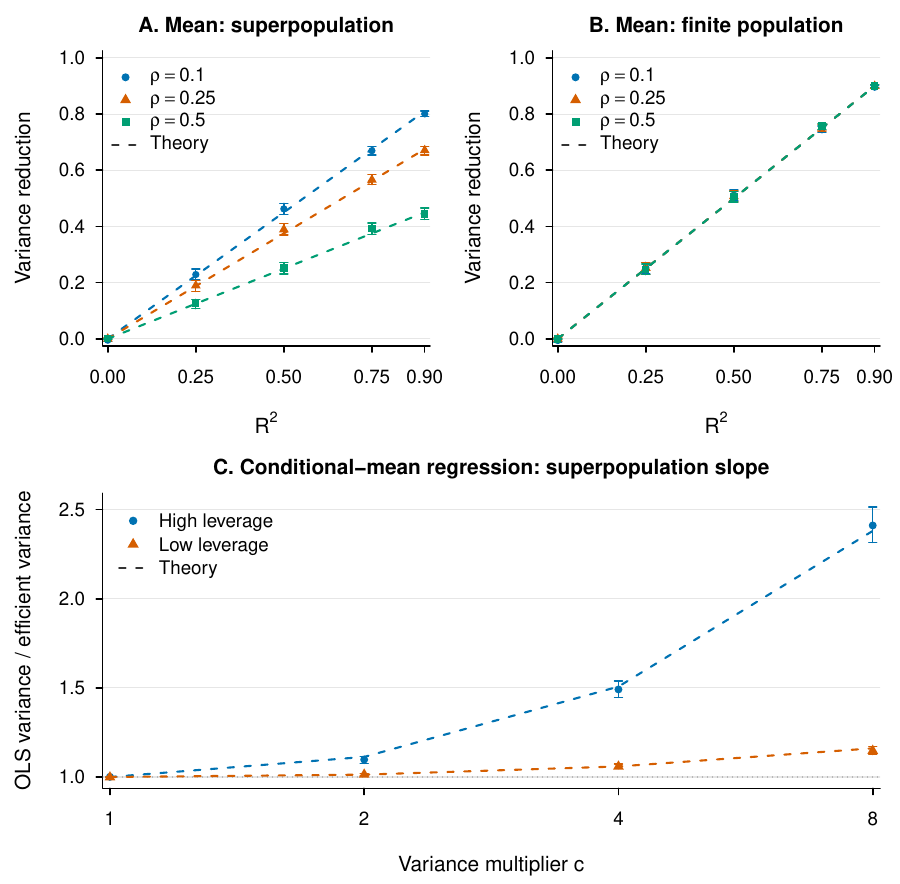}
\caption{Efficiency-gain map ($N=8{,}000$; 5,000 repetitions per setting). Panels A and B show the variance reduction of the cross-fitted estimator relative to HT for the superpopulation and finite-population means; in Panel B the three sampling fractions coincide. Panel C shows the OLS-to-GLS variance ratio for the superpopulation slope when heteroskedasticity is concentrated at high- or low-leverage covariate values. Curves are theoretical values, points are Monte Carlo estimates, and error bars are 95\% paired bootstrap intervals.}
\label{fig-efficiency-gain-map}
\end{figure}

For regression, let $X\sim\operatorname{Unif}(-1,1)$, $Z=(1,X)^\top$, and $Y=Z^\top\beta_0+\varepsilon$ with conditionally Gaussian errors, $\rho=0.25$ and, for $c\in\{1,2,4,8\}$, either $\sigma^2(X)=1+(c-1)\mathbf1\{|X|>0.7\}$ or $\sigma^2(X)=1+(c-1)\mathbf1\{|X|<0.3\}$. The two variance functions have the same contrast, but the first places the excess variance at high-leverage values for the slope and the second near the center of the covariate distribution. We compare unweighted OLS on the sampled units with feasible GLS, in which the two variance regions are known and their levels are estimated by cross-fitting. The theoretical OLS-to-GLS variance ratio for the slope is the ratio of the corresponding diagonal entries of $A_\pi^{-1}B_\pi A_\pi^{-1}$ and $I_\pi^{-1}$, where $A_\pi=\E\{\pi(X)ZZ^\top\}$ and $B_\pi=\E\{\pi(X)\sigma^2(X)ZZ^\top\}$.

The Monte Carlo estimates in Figure~\ref{fig-efficiency-gain-map} follow the theoretical values. At $R^2=0.75$ and $\rho=0.25$, the cross-fitted estimator reduces the variance relative to HT by 56.6\% for the superpopulation mean and 75.1\% for the finite-population mean, against theoretical reductions of 56.25\% and 75\%. For regression with $c=8$, the variance ratio is 2.411 (theory 2.380) when the excess variance is at high-leverage values, but only 1.148 (theory 1.161) when it is at low-leverage values. Regression adjustment is thus most valuable when the covariates are strongly predictive, and variance weighting when heteroskedasticity occurs at high-leverage covariate values. These comparisons use correctly specified, low-dimensional nuisance models and do not assess the cost of learning a flexible variance function.

\section{Real Data Analysis}
\label{sec-real-data}

We analyze the California Academic Performance Index (API) data distributed with the \texttt{survey} package \citep{Lumley2004}. We treat \texttt{apipop}, containing $N=6{,}194$ schools, as the finite population and \texttt{apistrat}, a stratified sample of $n=200$ schools, as the observed sample. Units are selected by SRSWOR independently within school-type strata, with population size \(N_h\), sample size \(n_h\), and inclusion probability \(\pi_{N,i}=n_h/N_h\) in stratum \(h\). The outcome is the 2000 API. Outcomes outside \texttt{apistrat} are masked during this sample's estimation; the complete outcome vector is used only to evaluate population quantities and to draw repeated samples. The covariates are the 1999 API, percentages eligible for subsidized meals and classified as English learners, student mobility, average parental education, percentages of fully qualified and emergency-credentialed teachers, log enrollment, and school-type indicators. Five folds are formed within school type using only covariate information. Appendix~\ref{app-real-data-implementation} gives implementation details.

\subsection{Mean}
\label{sec-real-data-mean}

We compare four estimators. HT uses no outcome regression. The generalized regression (GREG) estimator uses a survey-weighted linear regression fitted on the full observed sample. The cross-fitted linear augmentation estimator and the cross-fitted spline EIF estimator have the form \eqref{eq-mean-estimator-appendix}, with a linear regression on the same covariates and a penalized additive spline regression, respectively, fitted outside each fold. Finite-population variances use the residual-linearization formula for stratified sampling, and superpopulation variances add an HT estimate of the full-data contribution; for the cross-fitted linear estimator, the resulting standard errors differ from those based on \eqref{eq-sp-var-estimator} and \eqref{eq-fp-var-estimator-mean-positive} by less than 1\% (Appendix~\ref{app-api-mean-variance}).

\begin{table}[!tb]
\centering
\setlength{\tabcolsep}{4pt}
\caption{California API mean analysis with linear and spline augmentation diagnostics. Lower and Upper are 95\% Wald limits.}
\label{tab:api-mean-actual-diagnostic}
\begin{tabularx}{\linewidth}{@{}Yrrrr@{}}
\toprule
Estimator & Estimate & SE & Lower & Upper \\
\midrule
\multicolumn{5}{@{}l}{\textit{Superpopulation mean}} \\
HT & 662.287 & 9.538 & 643.594 & 680.981 \\
GREG & 664.474 & 2.431 & 659.710 & 669.238 \\
Cross-fitted linear augmentation & 664.179 & 2.601 & 659.081 & 669.278 \\
Cross-fitted spline EIF & 664.244 & 2.616 & 659.116 & 669.372 \\
\addlinespace
\multicolumn{5}{@{}l}{\textit{Finite-population mean}} \\
HT & 662.287 & 9.409 & 643.846 & 680.729 \\
GREG & 664.474 & 1.862 & 660.824 & 668.124 \\
Cross-fitted linear augmentation & 664.179 & 2.080 & 660.102 & 668.256 \\
Cross-fitted spline EIF & 664.244 & 2.099 & 660.131 & 668.357 \\
\bottomrule
\end{tabularx}
\par\smallskip\parbox{0.95\linewidth}{Finite-population standard errors use the same residual-linearization formula for stratified sampling. Superpopulation variances additionally include the estimated full-data contribution. GREG residuals are evaluated in sample, whereas the two cross-fitted methods use out-of-fold predictions. The realized finite-population mean is 664.713.}
\end{table}

All four point estimates in Table~\ref{tab:api-mean-actual-diagnostic} are close to the realized finite-population mean, 664.713, and the three regression-adjusted estimators are much more precise than HT. The two centerings illustrate the generalized finite-population correction of Remark~\ref{rem-finite-population-correction-Pythagorean-theorem}. For HT the superpopulation and finite-population standard errors nearly coincide (9.538 and 9.409), as the classical correction $1-n_h/N_h\ge0.93$ suggests. For the regression-adjusted estimators the finite-population standard error is about 20\% smaller (2.601 versus 2.080 for the cross-fitted linear estimator), because the estimated full-data term $\hat V_Y/N=2.44$ is then of the same order as the remaining variance. Predictive covariates thus make the distinction between the two targets material even at a sampling fraction of 3\%.

GREG has the smallest reported standard error, but this alone does not establish greater efficiency, because its residuals are evaluated on the same observations used to fit the regression. The weighted RMSEs are 24.04 for in-sample GREG, against 26.79 and 27.01 for the out-of-fold predictions of the cross-fitted linear and spline estimators. In 1,000 repeated stratified samples from the fixed API population, the empirical SDs of GREG and the two cross-fitted estimators are 2.08, 2.10, and 2.15 API points, the ratios of average estimated to empirical SD are 0.935, 1.010, and 1.003, and coverage is 0.924, 0.936, and 0.939 (Table~\ref{tab:api-mean-repeated-sampling}, on the $\sqrt N$ scale). The three estimators are thus almost equally precise, cross-fitting improves the agreement between estimated and empirical SDs, and, in this small and nearly linear application, the simple working model performs at least as well as the flexible learner. This diagnostic is conditional on the realized population and does not assess the joint superpopulation--design efficiency bound.

\subsection{Conditional-mean regression}
\label{sec-real-data-regression}

Under the linear conditional-mean working model with the covariates above, we compare unweighted OLS on the sampled schools with the feasible GLS estimator \eqref{eq-regression-estimating-equation} under two cross-fitted variance models (Appendix~\ref{app-api-regression-implementation}). Because the API is a school-level average, its conditional variance should decrease with school size; in the population, the residual standard deviation of the working regression falls from 34.3 in the lowest enrollment decile to 20.6 in the highest. The first model, GLS-size, is therefore a Gamma log-link regression of squared residuals on log enrollment alone, with two parameters. The second, GLS-flex, uses linear and quadratic terms of all continuous covariates and school-type indicators, with 19 parameters. Table~\ref{tab:api-regression-main} reports the coefficient of the standardized 1999 API in the realized sample and in 5,000 repeated stratified samples from the fixed population.

\begin{table}[!tb]
\centering
\small
\setlength{\tabcolsep}{2.7pt}
\caption{California API regression: coefficient of standardized 1999 API. The realized sample is \texttt{apistrat}; repeated-sampling results use 5,000 stratified samples from the fixed population. GLS-size uses a two-parameter log-enrollment variance model; GLS-flex uses the 19-parameter model. Var./OLS is empirical variance relative to OLS; SE/SD is average estimated SE divided by empirical SD. The census OLS coefficient is 113.2. The SE formulas are for superpopulation inference under the working model; the empirical SDs are calculated by repeated sampling from the fixed API population.}
\label{tab:api-regression-main}
\begin{tabular*}{\linewidth}{@{\extracolsep{\fill}}lrrrrrrr@{}}
\toprule
& \multicolumn{3}{c}{Realized sample} & \multicolumn{4}{c}{Repeated sampling} \\
\cmidrule(lr){2-4}\cmidrule(l){5-8}
& & Plug-in & Sandwich & Emp. & Var./ & Plug-in & Sandwich \\
Estimator & Estimate & SE & SE & SD & OLS & SE/SD & SE/SD \\
\midrule
OLS & 117.716 & --- & 4.593 & 5.141 & 1.000 & --- & 0.954 \\
GLS-size & 117.940 & 4.187 & 4.555 & 4.926 & 0.918 & 0.903 & 0.968 \\
GLS-flex & 132.615 & 2.811 & 4.787 & 6.691 & 1.694 & 0.524 & 0.760 \\
\bottomrule
\end{tabular*}
\end{table}

With the size model, variance weighting yields a modest but stable reduction in repeated-sampling variance. The realized estimate is close to that of OLS and varies only between 116.8 and 118.9 over 200 random fold assignments. In repeated sampling, the variance is smaller than that of OLS for every slope coefficient, by 2--12\%, and the sandwich SE/SD ratios are similar to those of OLS. The average plug-in standard error, based on the superpopulation variance estimator in Theorem~\ref{thm-regression-efficient-estimation}(b), is about 10\% below the empirical SD for the 1999 API coefficient. The flexible model is unstable at this sample size: its estimate is 132.6 with the fold assignment used here and between 118.6 and 131.1 over the 200 assignments, its variance is 1.46--2.14 times that of OLS, and its average plug-in standard errors are 34--57\% below the corresponding empirical SDs. Among these two variance models, only GLS-size reduces the repeated-sampling variance.

Two caveats apply. The estimating equation \eqref{eq-regression-estimating-equation} does not use the design weights, which is what efficiency requires under a correctly specified conditional mean; under misspecification, however, OLS targets a projection weighted by $\pi(X)$, while GLS additionally uses its limiting inverse-variance weights. For OLS in the API population, the difference is small: the census OLS coefficient of the 1999 API is 113.2 without and 114.8 with the weights $\pi_{N,i}$, a difference of about one third of the sample standard error. Moreover, the sandwich standard errors treat the fitted weights as given and do not by themselves establish valid inference under mean-model misspecification.

\section{Discussion}
\label{sec-discussion}

The reference Poisson experiment provides a common framework for efficiency analysis under dependent non-informative survey designs. The numerical experiments and API analysis suggest several practical implications. For mean estimation, regression adjustment is most valuable when the covariates are strongly predictive of the outcome, with the relative gain for the finite-population target persisting even at larger sampling fractions. For conditional-mean regression, variance weighting is most beneficial when heteroskedasticity occurs at covariate values with high leverage for the coefficient of interest. At moderate sample sizes, however, these asymptotic gains must be balanced against the cost of nuisance estimation. A simple working model may outperform a more flexible learner, while cross-fitting can still materially improve standard-error calibration. Under cluster sampling, the accuracy required of the outcome regression grows with the cluster size through the sufficient condition \eqref{eq-group-size-regression-rate}, which also favors parsimonious learners.

Potential applications within the present observation scheme include the cube method for balanced selection \citep{DevilleTille2004} and the local pivotal method for spatially balanced selection \citep{GrafstromLundstromSchelin2012}, with balancing variables or locations included among the covariates observed for every population unit. Under the full-data and target assumptions, verifying Assumptions~\ref{ass-B1}--\ref{ass-B4} for these designs would make the efficiency theory in Section~\ref{sec-lan-efficiency} applicable. Establishing attainment by cross-fitted estimators would additionally require suitable conditions on nuisance estimation and fold construction, as in Section~\ref{sec-efficient-estimation}. Deriving explicit attainment conditions for other pathwise-differentiable targets is another direction within this framework.

Two-phase rejective sampling \citep{YangDing2025}, in contrast, observes auxiliary covariates only in a first-phase sample. This changes the observed-data experiment, so extending the theory requires an observed-data score operator and joint-limit analysis for the new observation scheme, rather than verification of the existing sampling assumptions alone. Further directions include informative sampling, outcome dependence within clusters, sampling clusters containing a non-vanishing fraction of the population, and design-conditional efficiency for fixed populations. For instance, in the API population the residuals of the working regression have an intraclass correlation of about 0.18 within school districts, which points to district-level structure not captured by the covariates, so that district cluster samples would call for such an extension. For finite-population targets whose first-order expansions are not governed by the full-data EIF, both the equivalence of efficiency under the two centerings and the form of the variance correction should be re-examined.

\section*{Acknowledgments}
ChatGPT (OpenAI) was used to assist with language editing, code development for numerical experiments and real-data analysis, and LaTeX preparation. The authors take full responsibility for the accuracy and integrity of the manuscript.

\clearpage
\appendix
\numberwithin{equation}{subsection}
\numberwithin{theorem}{subsection}
\renewcommand{\theHequation}{\thesection.\arabic{subsection}.\arabic{equation}}
\renewcommand{\theHtheorem}{\thesection.\arabic{subsection}.\arabic{theorem}}
\numberwithin{table}{section}
\numberwithin{figure}{section}
\section{Technical proofs and design verifications}
\label{app-technical-details}
\subsection{Efficiency results and proofs}\label{app-proofs}

The notation and assumptions are those of Sections~\ref{sec-basic-setup} and~\ref{sec-lan-efficiency}. This subsection establishes the LAN, pathwise differentiability, projection, and convolution results underlying Theorems~\ref{thm-obs-eif} and~\ref{thm-efficiency} and characterizes optimal limiting inclusion probabilities.

An estimator \(\hat\theta_N\) is required to be measurable with respect to \(\mathcal O_N=\sigma(O_1,\ldots,O_N)\).

\subsubsection{LAN}\label{sec-lan}

\begin{theorem}[LAN for the observed-data experiment]\label{thm-obs-lan}
Assume Assumption~\ref{ass-A1} and Assumptions~\ref{ass-B1}--\ref{ass-B3}. Then the observed-data model satisfies LAN at \(P_0^N\) in the sense of \citet{McNeney2000}, indexed by the linear tangent set \(\mathcal T\) in \eqref{eq-obs-tangent-set}. More precisely, let \(t\mapsto P_{t,g}^F\), \(g\in\dot{\mathcal P}_{P_0^F}\cap L^4(P_0^F)\), be a smooth one-dimensional submodel through \(P_0^F\), and let \(P^N_{t,\mathcal Lg}\) and \(Q^N_{t,g}\) be the corresponding observed-data and joint laws. For fixed \(u\in\mathbb R\),
\begin{align}
\log\frac{dP^N_{u/\sqrt N,\mathcal L g}}{dP_0^N}
&=u\,\Delta_N(\mathcal Lg)
-
\frac{u^2}{2}\|\mathcal Lg\|^2
+o_{P_0^N}(1),
\label{eq-LAN}\\
\text{where}\qquad
\Delta_N(\mathcal Lg)&=\frac1{\sqrt N}\sum_{i=1}^N\left[\delta_i g(L_i)+(1-\delta_i)\mathbb E[g(L_i)\mid X_i]\right].
\label{eq-central-seq}
\end{align}
For every finite collection \(g_1,\ldots,g_m\in\dot{\mathcal P}_{P_0^F}\cap L^4(P_0^F)\),
\begin{equation}
\left(\Delta_N(\mathcal Lg_1),\ldots,\Delta_N(\mathcal Lg_m)\right)^\top\xrightarrow[]{d}N_m\!\left(0,\left\{\langle\mathcal Lg_j,\mathcal Lg_\ell\rangle\right\}_{j,\ell=1}^m\right).
\label{eq-central-sequence-joint-limit}
\end{equation}
\end{theorem}

\begin{remark}
The theorem shows that the LAN central sequences have the same limiting covariance structure as those in the corresponding i.i.d.\ Poisson experiment with inclusion probability \(\pi(X)\). Together with the convolution theorem, it implies that, under the joint law of the superpopulation and sampling design, the first-order efficiency bound depends on the design only through \(\pi(X)\). The dependence structure of the original designs does not enter the first-order semiparametric efficiency analysis.
\end{remark}

\begin{lemma}[DQM for marginal and conditional distributions]
\label{lem-dqm-factorization-obs}
Let \(t\mapsto P_t^F\) be a one-dimensional submodel through \(P_0^F\) with joint density \(f_t(x,y)\) and marginal density \(m_t(x)=\int f_t(x,y)\,dy\). Suppose the submodel is differentiable in quadratic mean (DQM) at \(0\) with score \(g\in L^4(P_0^F)\), and \(P_t^F\ll P_0^F\) for all sufficiently small \(|t|\). Put \(k(x)=\mathbb E[g(L)\mid X=x]\) and \(s(x,y)=g(x,y)-k(x)\). Then the marginal path of \(X\) is DQM with score \(k\), and the conditional path of \(Y\mid X\) is DQM on average with score \(s\).
\end{lemma}

\begin{proof}
This is \citet[Theorem~4.3]{Inagaki1983} applied to the factorization \(f_t(x,y)=m_t(x)f_t(y\mid x)\).
\end{proof}

\begin{lemma}[Sampled conditional likelihood expansion]
\label{lem-sampled-conditional-likelihood-expansion}
Assume the setting of Lemma~\ref{lem-dqm-factorization-obs} and Assumption~\ref{ass-B1}. Put \(v(X)=\E[s(L)^2\mid X]\). Let \(u\in\R\) be fixed. Then, under \(P_0^N\),
\[
\sum_{i=1}^N\delta_i
\log\frac{f_{u/\sqrt N}(Y_i\mid X_i)}{f_0(Y_i\mid X_i)}
=
\frac{u}{\sqrt N}\sum_{i=1}^N\delta_i s(L_i)
-
\frac{u^2}{2N}\sum_{i=1}^N\delta_i v(X_i)
+
o_{P_0^N}(1).
\]
\end{lemma}

\begin{proof}
If \(u=0\), the statement is obvious. Suppose \(u \neq 0\), and set \(f_t(\cdot\mid x)=f_0(\cdot\mid x)\) whenever \(m_t(x)=0<m_0(x)\).
Let \(R_t(L)=\left\{\frac{f_t(Y\mid X)}{f_0(Y\mid X)}\right\}^{1/2}\), \(a_t(L)=\frac{R_t(L)-1-\frac{t}{2}s(L)}{t}\) for \(t\ne0\), and
\[
\omega_{N,i}=R_{u/\sqrt N}(L_i)-1
=
\frac{u}{2\sqrt N}s(L_i)+\frac{u}{\sqrt N}a_{u/\sqrt N}(L_i).
\]
Thus
\begin{equation}
R_t(L)=1+\frac{t}{2}s(L)+ta_t(L),
\quad
\log\frac{f_t(Y\mid X)}{f_0(Y\mid X)}=2\log R_t(L).
\label{eq-lem-sampled-cond-identities}
\end{equation}
Lemma~\ref{lem-dqm-factorization-obs} gives, as \(t \to 0\),
\begin{equation}
\E[a_t(L)^2]\to0,
\quad
\E[s(L)\mid X]=0,
\quad
\E[s(L)^2\mid X]=v(X).
\label{eq-lem-sampled-cond-dqm}
\end{equation}
Moreover,
\begin{equation}
\E|s(L)|^4
\le
8\E|g(L)|^4+8\E|k(X)|^4
\le
16\E|g(L)|^4
<\infty,
\label{eq-lem-sampled-cond-s-fourth}
\end{equation}
and
\begin{equation}
\E|s(L)a_t(L)|
\le
\{\E s(L)^2\}^{1/2}\{\E a_t(L)^2\}^{1/2}
\to0.
\label{eq-lem-sampled-cond-sa}
\end{equation}
Since \(\E[R_t(L)^2\mid X]=1\),
\begin{align}
0
&=
\E[R_t(L)^2-1\mid X] \notag\\
&=
2t\E[a_t(L)\mid X]
+
\frac{t^2}{4}v(X)
+
t^2\E[s(L)a_t(L)\mid X]
+
t^2\E[a_t(L)^2\mid X].
\label{eq-lem-sampled-cond-normalization}
\end{align}
By \eqref{eq-lem-sampled-cond-dqm} and \eqref{eq-lem-sampled-cond-sa},
\begin{align}
\E\left|
\frac{u^2}{N}\sum_{i=1}^N\delta_i
\E[s(L_i)a_{u/\sqrt N}(L_i)\mid X_i]
\right|
&\le
u^2\E|s(L)a_{u/\sqrt N}(L)|
=o(1),
\notag\\
\E\left[
\frac{u^2}{N}\sum_{i=1}^N\delta_i
\E[a_{u/\sqrt N}(L_i)^2\mid X_i]
\right]
&\le
u^2\E[a_{u/\sqrt N}(L)^2]
=o(1),
\notag\\
\E\left|
\frac{u^2}{N}\sum_{i=1}^N\delta_i s(L_i)a_{u/\sqrt N}(L_i)
\right|
&\le
u^2\E|s(L)a_{u/\sqrt N}(L)|
=o(1),
\notag\\
\E\left[
\frac{u^2}{N}\sum_{i=1}^N\delta_i a_{u/\sqrt N}(L_i)^2
\right]
&\le
u^2\E[a_{u/\sqrt N}(L)^2]
=o(1).
\label{eq-lem-sampled-cond-a-basic-rem}
\end{align}
Moreover,
\begin{align}
&\E\left[
\left\{
\frac{2u}{\sqrt N}\sum_{i=1}^N\delta_i
\left[
a_{u/\sqrt N}(L_i)-\E[a_{u/\sqrt N}(L_i)\mid X_i]
\right]
\right\}^2
\,\middle|\,
\bm X_N,\boldsymbol\delta_N
\right]\notag\\
&\quad\le
4\frac{u^2}{N}\sum_{i=1}^N\delta_i
\E[a_{u/\sqrt N}(L_i)^2\mid X_i].
\label{eq-lem-sampled-cond-a-centered}
\end{align}
Hence \eqref{eq-lem-sampled-cond-a-basic-rem} and \eqref{eq-lem-sampled-cond-a-centered} imply
\begin{equation}
\frac{2u}{\sqrt N}\sum_{i=1}^N\delta_i
\left[
a_{u/\sqrt N}(L_i)-\E[a_{u/\sqrt N}(L_i)\mid X_i]
\right]
=o_{P_0^N}(1).
\label{eq-lem-sampled-cond-a-centered-op}
\end{equation}
Summing \eqref{eq-lem-sampled-cond-normalization} over sampled units and using
\eqref{eq-lem-sampled-cond-a-basic-rem} gives
\begin{equation}
\frac{2u}{\sqrt N}\sum_{i=1}^N\delta_i
\E[a_{u/\sqrt N}(L_i)\mid X_i]
=
-\frac{u^2}{4N}\sum_{i=1}^N\delta_i v(X_i)
+
o_{P_0^N}(1).
\label{eq-lem-sampled-cond-a-mean-sum}
\end{equation}
Combining \eqref{eq-lem-sampled-cond-a-mean-sum} and
\eqref{eq-lem-sampled-cond-a-centered-op},
\begin{equation}
\frac{2u}{\sqrt N}\sum_{i=1}^N\delta_i a_{u/\sqrt N}(L_i)
=
-\frac{u^2}{4N}\sum_{i=1}^N\delta_i v(X_i)
+
o_{P_0^N}(1).
\label{eq-lem-sampled-cond-a-sum}
\end{equation}
On the other hand, by \eqref{eq-lem-sampled-cond-a-basic-rem},
\begin{align}
\sum_{i=1}^N\delta_i\omega_{N,i}^2
&=
\sum_{i=1}^N\delta_i
\left\{
\frac{u}{2\sqrt N}s(L_i)+\frac{u}{\sqrt N}a_{u/\sqrt N}(L_i)
\right\}^2 \notag\\
&=
\frac{u^2}{4N}\sum_{i=1}^N\delta_i s(L_i)^2
+
\frac{u^2}{N}\sum_{i=1}^N\delta_i s(L_i)a_{u/\sqrt N}(L_i)
+
\frac{u^2}{N}\sum_{i=1}^N\delta_i a_{u/\sqrt N}(L_i)^2 \notag\\
&=
\frac{u^2}{4N}\sum_{i=1}^N\delta_i s(L_i)^2
+
o_{P_0^N}(1).
\label{eq-lem-sampled-cond-omega-square-1}
\end{align}
By \eqref{eq-lem-sampled-cond-dqm} and \eqref{eq-lem-sampled-cond-s-fourth},
\begin{align}
&\E\left[
\left\{
\frac{u^2}{N}\sum_{i=1}^N\delta_i
\left[
s(L_i)^2-v(X_i)
\right]
\right\}^2
\,\middle|\,
\bm X_N,\boldsymbol\delta_N
\right]
\le
\frac{u^4}{N^2}\sum_{i=1}^N\delta_i
\E[s(L_i)^4\mid X_i].
\notag
\end{align}
Also, \(\E\left[ \frac{u^4}{N^2}\sum_{i=1}^N\delta_i \E[s(L_i)^4\mid X_i] \right] \le \frac{u^4}{N}\E|s(L)|^4 \to0\). Therefore
\begin{equation}
\frac{u^2}{N}\sum_{i=1}^N\delta_i
\{s(L_i)^2-v(X_i)\}
=
o_{P_0^N}(1).
\label{eq-lem-sampled-cond-s2-v}
\end{equation}
Combining \eqref{eq-lem-sampled-cond-omega-square-1} and \eqref{eq-lem-sampled-cond-s2-v},
\begin{equation}
\sum_{i=1}^N\delta_i\omega_{N,i}^2
=
\frac{u^2}{4N}\sum_{i=1}^N\delta_i v(X_i)
+
o_{P_0^N}(1).
\label{eq-lem-sampled-cond-omega-square}
\end{equation}
Moreover, \(\E\left[\frac{u^2}{N}\sum_{i=1}^N\delta_i v(X_i)\right] \le u^2\E[v(X)] <\infty\), so \(\sum_{i=1}^N\delta_i\omega_{N,i}^2=O_{P_0^N}(1)\). For every \(\varepsilon>0\),
\begin{align}
&P_0^N\left(\max_{1\le i\le N}\delta_i|\omega_{N,i}|>\varepsilon\right)\notag\\
&\quad\le
\sum_{i=1}^N
\E\left[
\delta_i\mathbf 1\{|\frac{u}{\sqrt N}s(L_i)|>\varepsilon\}
\right]\notag\\
&\qquad{}+
\sum_{i=1}^N
\E\left[
\delta_i\mathbf 1\left\{|\frac{u}{\sqrt N}a_{u/\sqrt N}(L_i)|>\frac{\varepsilon}{2}\right\}
\right] \notag\\
&\quad\le
N\E\left[
\mathbf 1\{|\frac{u}{\sqrt N}s(L)|>\varepsilon\}
\right]
+
N\E\left[
\mathbf 1\left\{|\frac{u}{\sqrt N}a_{u/\sqrt N}(L)|>\frac{\varepsilon}{2}\right\}
\right] \notag\\
&\quad\le
\frac{u^4}{\varepsilon^4 N}\E|s(L)|^4
+
\frac{4u^2}{\varepsilon^2}\E[a_{u/\sqrt N}(L)^2]
\to0.
\label{eq-lem-sampled-cond-max}
\end{align}
For \(0<\eta<1\), set
\[
A_{N,\eta}
=
\left\{
\max_{1\le i\le N}\delta_i|\omega_{N,i}|\le\eta
\right\},
\quad
c(\eta)
=
\sup_{0<|z|\le\eta}
\left|
\frac{2\log(1+z)-2z+z^2}{z^2}
\right|.
\]
Then
\begin{equation}
c(\eta)\to0
\quad \text{as } \eta\downarrow0.
\label{eq-lem-sampled-cond-log-taylor}
\end{equation}
Moreover,
\begin{align}
&\left|
\sum_{i=1}^N\delta_i
\left[
2\log(1+\omega_{N,i})
-
2\omega_{N,i}
+
\omega_{N,i}^2
\right]
\right|
\mathbf 1_{A_{N,\eta}}
\le
c(\eta)
\sum_{i=1}^N\delta_i\omega_{N,i}^2.
\label{eq-lem-sampled-cond-log-rem-bound}
\end{align}
Hence, for every \(\varepsilon>0\),
\begin{align}
&\limsup_{N\to\infty}
P_0^N\left(
\left|
\sum_{i=1}^N\delta_i
\left[
2\log(1+\omega_{N,i})
-
2\omega_{N,i}
+
\omega_{N,i}^2
\right]
\right|>\varepsilon
\right) \notag\\
&\quad\le
\limsup_{N\to\infty}P_0^N(A_{N,\eta}^c)
+
\limsup_{N\to\infty}
P_0^N\left(
c(\eta)\sum_{i=1}^N\delta_i\omega_{N,i}^2>\varepsilon
\right)
\to0
\end{align}
as \(\eta\downarrow0\), by \eqref{eq-lem-sampled-cond-omega-square},
\eqref{eq-lem-sampled-cond-max}, and \eqref{eq-lem-sampled-cond-log-taylor}. Therefore
\begin{equation}
\sum_{i=1}^N\delta_i
\left[
2\log(1+\omega_{N,i})
-
2\omega_{N,i}
+
\omega_{N,i}^2
\right]
=
o_{P_0^N}(1).
\label{eq-lem-sampled-cond-log-rem}
\end{equation}
Finally, by \eqref{eq-lem-sampled-cond-identities}, \eqref{eq-lem-sampled-cond-a-sum}, \eqref{eq-lem-sampled-cond-omega-square}, and \eqref{eq-lem-sampled-cond-log-rem},
\begin{align*}
&\sum_{i=1}^N\delta_i
\log\frac{f_{u/\sqrt N}(Y_i\mid X_i)}{f_0(Y_i\mid X_i)}\\
&\quad=
2\sum_{i=1}^N\delta_i\omega_{N,i}
-
\sum_{i=1}^N\delta_i\omega_{N,i}^2
+
o_{P_0^N}(1) \\
&\quad=
\frac{u}{\sqrt N}\sum_{i=1}^N\delta_i s(L_i)
+
\frac{2u}{\sqrt N}\sum_{i=1}^N\delta_i a_{u/\sqrt N}(L_i)\\
&\qquad{}-
\frac{u^2}{4N}\sum_{i=1}^N\delta_i v(X_i)
+
o_{P_0^N}(1) \\
&\quad=
\frac{u}{\sqrt N}\sum_{i=1}^N\delta_i s(L_i)
-
\frac{u^2}{2N}\sum_{i=1}^N\delta_i v(X_i)
+
o_{P_0^N}(1).
\end{align*}
\end{proof}

\begin{lemma}[Conditional Lindeberg--Feller central limit theorem]
\label{lem-conditional-lindeberg-feller}
Let \(\mathcal G_N\) be sub-\(\sigma\)-fields and, conditionally on \(\mathcal G_N\), let \(\xi_{N,1},\dots,\xi_{N,N}\) be independent mean-zero \(d\)-vectors. Suppose \(\sum_i\mathbb E[\xi_{N,i}\xi_{N,i}^\top\mid\mathcal G_N]\xrightarrow[]{p}V\) for a nonrandom matrix \(V\), and \(\sum_i\mathbb E[\|\xi_{N,i}\|^2\mathbf1\{\|\xi_{N,i}\|>\varepsilon\}\mid\mathcal G_N]\xrightarrow[]{p}0\) for every \(\varepsilon>0\). Then
\[
\mathbb E\left[\exp\left\{it^\top\sum_{i=1}^N\xi_{N,i}\right\}\middle|\mathcal G_N\right]
\xrightarrow[]{p}\exp(-t^\top Vt/2),\qquad t\in\mathbb R^d.
\]
Consequently, \(\sum_i\xi_{N,i}\xrightarrow[]{d}N_d(0,V)\), stably with respect to \(\mathcal G_N\).
\end{lemma}

\begin{proof}
Apply \citet[Theorem~1 and Corollary~3]{Bulinski2017ConditionalCLT} conditionally on \(\mathcal G_N\). Stable convergence follows from convergence of the conditional characteristic functions.
\end{proof}

\begin{remark}[Use of stable convergence]
\label{rem-stable-convergence-use}
The normalized sum converges jointly with every convergent sequence of \(\mathcal G_N\)-measurable statistics, and their limits are independent. We use this with \(\mathcal G_N=\sigma(\boldsymbol\delta_N,\bm X_N)\) in Theorem~\ref{thm-obs-lan}, and with \(\mathcal G_N=\mathcal O_N\) for full-data residuals in Lemma~\ref{lem-joint-full-observed-eif}.
\end{remark}

\begin{proof}[Proof of Theorem~\ref{thm-obs-lan}]
Put \(k(X)=\mathbb E[g(L)\mid X], \quad s(L)=g(L)-k(X), \quad v(X)=\mathbb E[s(L)^2\mid X]\). Then \(\mathbb E[s(L)\mid X]=0\), \(\mathbb E[g(L)^2]=\mathbb E[k(X)^2]+\mathbb E[v(X)]\), and \(v\in L^2((P_0^F)^X)\) because \(g\in L^4(P_0^F)\). Lemma~\ref{lem-dqm-factorization-obs} gives the factorization \(f_t(x,y)=m_t(x)f_t(y\mid x)\) with marginal score \(k\) for \(X\) and conditional score \(s\) for \(Y\mid X\).

Since \(q_N(\boldsymbol\delta_N\mid\bm X_N)\) does not depend on \(t\), the observed-data likelihood ratio under \(u/\sqrt N\) is
\begin{align}
\log\frac{dP^N_{u/\sqrt N,\mathcal Lg}}{dP_0^N} &= \sum_{i=1}^N \log\frac{m_{u/\sqrt N}(X_i)}{m_0(X_i)} + \sum_{i=1}^N \delta_i \log\frac{f_{u/\sqrt N}(Y_i\mid X_i)}{f_0(Y_i\mid X_i)}.
\label{eq-proof-lan-factorization}
\end{align}
The DQM of the marginal distribution yields
\begin{equation}
\sum_{i=1}^N \log\frac{m_{u/\sqrt N}(X_i)}{m_0(X_i)} = \frac{u}{\sqrt N}\sum_{i=1}^Nk(X_i) - \frac{u^2}{2}\mathbb E[k(X)^2] +o_{P_0^N}(1).
\label{eq-proof-lan-marginal}
\end{equation}
For the conditional likelihood, Lemma~\ref{lem-sampled-conditional-likelihood-expansion} gives
\begin{equation}
\sum_{i=1}^N\delta_i \log\frac{f_{u/\sqrt N}(Y_i\mid X_i)}{f_0(Y_i\mid X_i)} = \frac{u}{\sqrt N}\sum_{i=1}^N\delta_i s(L_i) - \frac{u^2}{2N}\sum_{i=1}^N\delta_i v(X_i) + o_{P_0^N}(1).
\label{eq-proof-lan-conditional}
\end{equation}
We identify the limit of the quadratic term in \eqref{eq-proof-lan-conditional}. We have
\begin{align*}
\frac1N\sum_{i=1}^N\delta_i v(X_i) &= \frac1N\sum_{i=1}^N\pi_{N,i}v(X_i) + \frac1N\sum_{i=1}^N(\delta_i-\pi_{N,i})v(X_i).
\end{align*}
Assumption~\ref{ass-B3}, applied to \(v\in L^2((P_0^F)^X)\), makes the second term \(o_{P_0^N}(1)\). Assumption~\ref{ass-B2} and the law of large numbers give \(N^{-1}\sum_{i=1}^N\pi_{N,i}v(X_i) = N^{-1}\sum_{i=1}^N\pi(X_i)v(X_i)+o_{P_0^N}(1) \xrightarrow[]{p} \mathbb E[\pi(X)v(X)]\). Combining this limit with \eqref{eq-proof-lan-factorization}, \eqref{eq-proof-lan-marginal}, and \eqref{eq-proof-lan-conditional},
\begin{align*}
\log\frac{dP^N_{u/\sqrt N,\mathcal Lg}}{dP_0^N}
&=\frac{u}{\sqrt N}\sum_{i=1}^N\{k(X_i)+\delta_i s(L_i)\}\\
&\quad-\frac{u^2}{2}\left[\mathbb E[k(X)^2]+\mathbb E[\pi(X)v(X)]\right]+o_{P_0^N}(1).
\end{align*}
Since \((\mathcal Lg)(O_i)=k(X_i)+\delta_i s(L_i)\), and since \eqref{eq-obs-norm} gives \(\|\mathcal Lg\|^2 = \mathbb E[k(X)^2]+\mathbb E[\pi(X)v(X)]\), the LAN expansion \eqref{eq-LAN} follows.

We now prove the convergence in distribution of the central sequence. Conditional on \((\boldsymbol\delta_N,\bm X_N)\), the variables \(\delta_i s(L_i)\) are independent, centered, and have conditional variance \(\delta_i v(X_i)\). The preceding convergence of \(N^{-1}\sum_i\delta_i v(X_i)\), the Lindeberg condition implied by \(g\in L^4(P_0^F)\), and Lemma~\ref{lem-conditional-lindeberg-feller} show that \(\frac1{\sqrt N}\sum_{i=1}^N\delta_i s(L_i)\) converges stably with respect to \(\sigma(\boldsymbol\delta_N,\bm X_N)\) to a random variable with distribution \(N(0,\mathbb E[\pi(X)v(X)])\). The ordinary i.i.d.\ central limit theorem gives \(\frac1{\sqrt N}\sum_{i=1}^Nk(X_i) \xrightarrow[]{d} N(0,\mathbb E[k(X)^2])\). By Remark~\ref{rem-stable-convergence-use}, the stable convergence above is joint with the \(\bm X_N\)-measurable sequence \(N^{-1/2}\sum_i k(X_i)\). The two Gaussian limits are independent. Hence \(\Delta_N(\mathcal Lg) \xrightarrow[]{d} N\!\left(0, \mathbb E[k(X)^2]+\mathbb E[\pi(X)v(X)] \right) = N(0,\|\mathcal Lg\|^2)\). 
For a finite collection \(g_1,\dots,g_m\), apply the preceding scalar result to \(\sum_{j=1}^m c_jg_j\) for arbitrary \(c\in\mathbb R^m\). The Cram\'er--Wold theorem gives the joint limit in \eqref{eq-central-sequence-joint-limit}.
\end{proof}

\subsubsection{Pathwise differentiability}\label{sec-pathwise-observed}

For the observed-data experiment, we write \(\psi(P_{\theta,\eta}^N)=\psi(P_{\theta,\eta}^F)=\theta\).

\begin{theorem}[Pathwise differentiability under the observed-data model]
\label{thm-pathwise-obs}
Under Assumptions~\ref{ass-A1}--\ref{ass-A2} and~\ref{ass-B1}--\ref{ass-B4}, \(\psi\) is pathwise differentiable at \(P_0^N\) with tangent set \(\mathcal T\). Its derivative extends uniquely and continuously to \(\mathcal H\) and satisfies \(\dot\psi(\mathcal Lg)=\dot\psi^F(g)\) for all \(g\in\mathcal H^F\).
\end{theorem}

\begin{proof}[Proof of Theorem~\ref{thm-pathwise-obs}]
For \(g\in\mathcal H^F\), Assumption~\ref{ass-B4} and \eqref{eq-obs-norm} give
\begin{equation}
\pi_-\|g\|_F^2\le\|\mathcal Lg\|^2\le\|g\|_F^2.
\label{eq-eif-proof-coercivity}
\end{equation}
Thus \(\mathcal L\) is injective, has a continuous inverse on its range, and has closed range. Consequently, \(\mathcal H=\mathcal L\mathcal H^F\). Define \(\dot\psi(\mathcal Lg)=\dot\psi^F(g)\). For each coordinate,
\[
|\dot\psi_j(\mathcal Lg)|=|\langle\phi_j^F,g\rangle_F|
\le\pi_-^{-1/2}\|\phi_j^F\|_F\|\mathcal Lg\|.
\]
The derivative is therefore continuous on \(\mathcal H\). Along every admissible submodel, Assumption~\ref{ass-A2} gives
\[
\sqrt N\{\psi(P^N_{u/\sqrt N,\mathcal Lg})-\theta_0\}
=\sqrt N\{\psi(P^F_{u/\sqrt N,g})-\theta_0\}
\longrightarrow u\dot\psi^F(g)=u\dot\psi(\mathcal Lg).
\]
\end{proof}

\subsubsection{Observed-data EIF}
\label{sec-observed-eif}

\begin{proposition}[Observed-data EIF]
\label{prop-obs-eif}
Under Assumptions~\ref{ass-A1}--\ref{ass-A2} and~\ref{ass-B1}--\ref{ass-B4}, the observed-data EIF is \(\phi(O)=(\mathcal L \phi^{*})(O)\), where \(\phi^{*}\in(\mathcal H^F)^p\) is the unique coordinatewise solution of \eqref{eq-obs-eif-equation}. Its covariance under \(\mathbb P^\pi\) is \(\Sigma_{\theta_0}\).
\end{proposition}

\begin{proof}[Proof of Proposition~\ref{prop-obs-eif}]
Theorem~\ref{thm-pathwise-obs} gives a continuous derivative on \(\mathcal H=\mathcal L\mathcal H^F\). The Riesz representation theorem gives a unique \(\phi_j\in\mathcal H\) representing each coordinate. By \eqref{eq-eif-proof-coercivity}, there is a unique \(\phi^{*}_j\in\mathcal H^F\) with \(\phi_j=\mathcal L \phi^{*}_j\). For every \(g\in\mathcal H^F\),
\begin{equation}
\langle\mathcal M \phi^{*}_j,g\rangle_F
=\langle\mathcal L \phi^{*}_j,\mathcal Lg\rangle
=\dot\psi_j(\mathcal Lg)
=\dot\psi_j^F(g)=\langle\phi_j^F,g\rangle_F.
\label{eq-eif-proof-normal-equation}
\end{equation}
This proves \(\Proj[\mathcal M \phi^{*}_j\mid\mathcal H^F]=\phi_j^F\). Conversely, any solution of \eqref{eq-obs-eif-equation} represents the derivative, and uniqueness follows from \eqref{eq-eif-proof-coercivity}.

The inverse-probability-weighted function $\delta\phi_j^F(L)/\pi(X)$ belongs to $L^2_0(\mathbb P^\pi)$ by Assumption~\ref{ass-B4}. For every $g\in\mathcal H^F$,
\begin{align*}
\mathbb E_{\mathbb P^\pi}\!\left[\frac{\delta}{\pi(X)}\phi_j^F(L)(\mathcal Lg)(O)\right]
&=\mathbb E[\phi_j^F(L)g(L)]
=\langle\phi_j,\mathcal Lg\rangle.
\end{align*}
Since $\mathcal H=\mathcal L\mathcal H^F$, this also gives $\phi_j=\Proj[\delta\phi_j^F(L)/\pi(X)\mid\mathcal H]$, as illustrated in Figure~\ref{fig:eif-projection}.
\end{proof}

\begin{proof}[Proof of Theorem~\ref{thm-obs-eif}]
The LAN assertion follows from Theorem~\ref{thm-obs-lan}, and the EIF characterization follows from Proposition~\ref{prop-obs-eif}.
\end{proof}

\subsubsection{Local regularity and convolution theorems}\label{sec-convolution-theorems}

\begin{definition}[Local regularity]
\label{def-local-regularity}
\label{def-superpopulation-local-regularity}
\label{def-finite-population-local-regularity}
\textup{(a)} We say that \(\hat\theta_N\) is \emph{locally regular at \(P_0^F\) for the superpopulation target \(\theta_0\)} if \(\sqrt N\{\hat\theta_N-\psi(P^F_{u/\sqrt N,g})\}\) converges in distribution under \(P^N_{u/\sqrt N,\mathcal Lg}\) to a limit law independent of \(u\) and \(g\), for every fixed \(u\in\mathbb R\) and \(g\in\dot{\mathcal P}_{P_0^F}\cap L^4(P_0^F)\). \textup{(b)} We say that it is \emph{locally regular at \(P_0^F\) for the finite-population target \(\theta_N\)} if the same requirement holds for \(\sqrt N(\hat\theta_N-\theta_N)\) under \(Q^N_{u/\sqrt N,g}\).
\end{definition}

\begin{theorem}[Equivalence of local regularity]
\label{thm-local-regularity-equivalence}
Under Assumptions~\ref{ass-A1}--\ref{ass-A3} and~\ref{ass-B1}--\ref{ass-B4}, the two notions of regularity in Definition~\ref{def-local-regularity}\textup{(a)} and~\textup{(b)} are equivalent.
\end{theorem}

\begin{theorem}[Convolution under the two centerings]
\label{thm-convolution}
Assume Assumptions~\ref{ass-A1}--\ref{ass-A2} and~\ref{ass-B1}--\ref{ass-B4}, and let \(\hat\theta_N\) be locally regular for the superpopulation target. The first convergence \eqref{eq-conv-sp} below holds under \(P_0^N\), and, if Assumption~\ref{ass-A3} also holds, both \eqref{eq-conv-sp} and \eqref{eq-conv-fp} hold jointly under \(Q_0^N\).
\begin{alignat}{2}
\sqrt N(\hat\theta_N-\theta_0)
&\xrightarrow[]{d}G_{\mathrm{sp}}+W,
&\qquad G_{\mathrm{sp}}&\sim N_p(0,\Sigma_{\theta_0}),
\label{eq-conv-sp}\\
\sqrt N(\hat\theta_N-\theta_N)
&\xrightarrow[]{d}G_{\mathrm{fp}}+W,
&\qquad G_{\mathrm{fp}}&\sim N_p(0,\Sigma_{\theta_N}),
\label{eq-conv-fp}\\
\Sigma_{\theta_N}
=\Sigma_{\theta_0}&-\Sigma_F\succeq0,
&\qquad \Sigma_F&=\mathbb E[\phi^F(L)\phi^F(L)^\top].
\label{eq-conv-covariance}
\end{alignat}
The same residual \(W\) is independent of the Gaussian vector \((G_{\mathrm{sp}},G_{\mathrm{fp}})\). The estimator is efficient for the superpopulation target if \(W=0\) a.s., which under Assumption~\ref{ass-A3} is also the condition for efficiency for the finite-population target.
\end{theorem}

Theorems~\ref{thm-local-regularity-equivalence} and~\ref{thm-convolution} show that the same estimators are locally regular and efficient for both targets but the lower bounds are different. The covariance identity has the Pythagorean interpretation given in Remark~\ref{rem-finite-population-correction-Pythagorean-theorem}.

\begin{theorem}[Specialization of \citet{McNeney2000}, Theorem~2.4]
\label{thm-mcneney-wellner-convolution}
Let \(B=\mathbb R^p\), \(R_n=\sqrt n I_p\), and \(P_{n,0}=P_{n,\theta_0}\). Suppose the experiments \(\mathcal P_n=\{P_{n,\theta}\mid\theta\in\Theta\}\) are LAN at \(\theta_0\), indexed by a linear subspace \((H,\langle\cdot,\cdot\rangle)\) of a Hilbert space. Suppose \(\{\nu_n\}\) is differentiable, with \(\sqrt n\{\nu_n(P_{n,\theta_n(h)})-\nu_n(P_{n,0})\}\to\dot\nu(h)\) for every \(h\in H\), where \(\dot\nu\) is continuous and linear from \(H\) to \(B\). Use the same symbol for its continuous extension to \(\overline H\). If \(T_n\) is locally regular, meaning \(\sqrt n\{T_n-\nu_n(P_{n,\theta_n(h)})\}\xrightarrow[]{d}Z\) under \(P_{n,\theta_n(h)}\) with a law independent of \(h\), then \(Z\stackrel{d}{=}Z_0+W\), where \(Z_0\perp W\), \(P\{Z_0\in\dot\nu(\overline H)\}=1\), and, for every \(b^*\in B^*\),
\[
b^*Z_0\sim N\!\left(0,\|\dot\nu_{b^*}^{T}\|^2\right),\qquad
b^*\dot\nu(h)=\langle\dot\nu_{b^*}^{T},h\rangle,\quad h\in H.
\]
Here \(\dot\nu_{b^*}^{T}\in\overline H\) is the unique Riesz representer of \(h\mapsto b^*\dot\nu(h)\). In particular, \(Z_0\) is centered Gaussian.
\end{theorem}

\begin{proof}[Proof of Theorem~\ref{thm-convolution}, superpopulation part]
Apply Theorem~\ref{thm-mcneney-wellner-convolution} with \(n=N\), \(H=\mathcal T\), \(B=\mathbb R^p\), \(R_N=\sqrt N I_p\), and \(\nu_N(P_{\theta,\eta}^N)=\psi(P_{\theta,\eta}^N)\). For \(h=u\mathcal Lg\), the local law is \(P^N_{u/\sqrt N,\mathcal Lg}\), with central sequence \(\Delta_N(h)=u\Delta_N(\mathcal Lg)\). Theorems~\ref{thm-obs-lan} and~\ref{thm-pathwise-obs} verify LAN and differentiability, and Definition~\ref{def-local-regularity}\textup{(a)} supplies local regularity. Since \(\overline H=\mathcal H\), Proposition~\ref{prop-obs-eif} identifies \(b^\top\phi\) as the Riesz representer of \(b^\top\dot\psi\), with \(\|b^\top\phi\|^2=b^\top\Sigma_{\theta_0}b\) for every \(b\in\mathbb R^p\). The asserted convolution for the superpopulation target follows with \(G_{\mathrm{sp}}=Z_0\sim N_p(0,\Sigma_{\theta_0})\).
\end{proof}

\begin{lemma}[Central sequences for \(L^2\)-limits of observed scores]
\label{lem-central-sequence-L2-limits}
Under Assumptions~\ref{ass-A1} and~\ref{ass-B1}--\ref{ass-B4}, for every fixed \(a_1,\ldots,a_m\in\mathcal H^F\),
\[
\left(\Delta_N(\mathcal La_1),\ldots,\Delta_N(\mathcal La_m)\right)
\xrightarrow[]{d}N_m\!\left(0,\{\langle\mathcal La_j,\mathcal La_\ell\rangle\}_{j,\ell=1}^m\right).
\]
If \(a_{j,r}\in\dot{\mathcal P}_{P_0^F}\cap L^4(P_0^F)\) and \(\|a_{j,r}-a_j\|_F\to0\), then
\[
\lim_{r\to\infty}\limsup_{N\to\infty}\mathbb E|\Delta_N\{\mathcal L(a_{j,r}-a_j)\}|^2=0.
\]
\end{lemma}

\begin{proof}
For \(b\in\mathcal H^F\), write \(k_b(X)=\mathbb E[b(L)\mid X]\) and \(s_b(L)=b(L)-k_b(X)\). Conditional centering and independence give
\begin{equation}
\mathbb E|\Delta_N(\mathcal Lb)|^2
=\mathbb E[k_b^2]+\mathbb E\left[\frac1N\sum_{i=1}^N\delta_i\mathbb E[s_b(L_i)^2\mid X_i]\right]
\le\mathbb E[k_b^2]+\mathbb E[s_b^2]=\|b\|_F^2.
\label{eq-lem-central-L2-bound}
\end{equation}
This proves the approximation assertion. By Assumption~\ref{ass-A1}, choose admissible \(a_{j,r}\to a_j\) in \(L^2(P_0^F)\). For fixed \(c\in\mathbb R^m\), put \(a^c=\sum_jc_ja_j\) and \(a_r^c=\sum_jc_ja_{j,r}\). Theorem~\ref{thm-obs-lan} gives \(\Delta_N(\mathcal La_r^c)\xrightarrow[]{d}N(0,\|\mathcal La_r^c\|^2)\). Also, by \eqref{eq-obs-norm}, \(\|\mathcal La_r^c-\mathcal La^c\|\le\|a_r^c-a^c\|_F\to0\). For every \(t\in\mathbb R\), equation~\eqref{eq-lem-central-L2-bound} yields
\begin{align*}
&\limsup_{N\to\infty}\left|\mathbb E e^{it\Delta_N(\mathcal La^c)}-e^{-t^2\|\mathcal La^c\|^2/2}\right|\\
&\quad\le |t|\|a_r^c-a^c\|_F
+\left|e^{-t^2\|\mathcal La_r^c\|^2/2}-e^{-t^2\|\mathcal La^c\|^2/2}\right|\longrightarrow0.
\end{align*}
The Cram\'er--Wold theorem gives the joint limit.
\end{proof}

For a finite-dimensional vector \(\varphi\) of observed-data functions, write \(\Delta_N(\varphi)=N^{-1/2}\sum_i\varphi(O_i)\).

\begin{lemma}[Joint convolution with central sequences]
\label{lem-joint-convolution-observed-central}
Under Assumptions~\ref{ass-A1}--\ref{ass-A2} and~\ref{ass-B1}--\ref{ass-B4}, let \(\hat\theta_N\) be locally regular for the superpopulation target and put \(A_N=\sqrt N(\hat\theta_N-\theta_0)\). For fixed \(r_1,\ldots,r_q\in\mathcal H^F\) and \(\varphi=(\mathcal Lr_1,\ldots,\mathcal Lr_q)^\top\),
\[
(A_N,\Delta_N(\phi),\Delta_N(\varphi))
\xrightarrow[]{d}(G_{\mathrm{sp}}+W,G_{\mathrm{sp}},G_\varphi),
\]
where \((G_{\mathrm{sp}},G_\varphi)\) is centered Gaussian with covariance induced by \(\langle\cdot,\cdot\rangle\), and \(W\perp(G_{\mathrm{sp}},G_\varphi)\). The assertion also holds with no \(\varphi\)-coordinate.
\end{lemma}

\begin{proof}
Write \(\phi_j=\mathcal L \phi^{*}_j\), with \(\phi^{*}_j\in\mathcal H^F\). Lemma~\ref{lem-central-sequence-L2-limits} applied to \(\phi^{*}_1,\ldots,\phi^{*}_p,r_1,\ldots,r_q\) gives the stated Gaussian limit of the last two coordinates. Local regularity at \(u=0\) gives tightness of \(A_N\). Along any subsequence, take a further subsequence on which
\begin{equation}
(A_N,\Delta_N(\phi),\Delta_N(\varphi))
\xrightarrow[]{d}(A,G_{\mathrm{sp}},G_\varphi).
\label{eq-lem-joint-conv-subseq-limit}
\end{equation}
Fix \(\alpha,\beta\in\mathbb R^p\) and \(\gamma\in\mathbb R^q\). To identify this joint law, put \(b=\sum_j\beta_j\phi^{*}_j+\sum_\ell\gamma_\ell r_\ell\) and \(h=\mathcal Lb=\beta^\top\phi+\gamma^\top\varphi\). Choose admissible \(b_m\to b\) in \(L^2(P_0^F)\) and put \(h_m=\mathcal Lb_m\). Equation~\eqref{eq-eif-proof-coercivity} and Theorem~\ref{thm-pathwise-obs} give \(\|h_m-h\|\to0\) and \(\dot\psi(h_m)\to\dot\psi(h)\). Moreover, \eqref{eq-lem-central-L2-bound} implies
\begin{equation}
\sup_N\mathbb E|\Delta_N(h_m)-\Delta_N(h)|^2\le\|b_m-b\|_F^2\to0.
\label{eq-lem-joint-conv-central-approx}
\end{equation}
For fixed \(m\), every further subsequential limit \((A,H_m)\) of \((A_N,\Delta_N(h_m))\) has \(H_m\sim N(0,\|h_m\|^2)\). LAN, pathwise differentiability, local regularity, and Le Cam's third lemma give, for every \(s\in\mathbb R\),
\begin{equation}
\mathbb E\exp\left\{i\alpha^\top A+sH_m-\frac{s^2}{2}\|h_m\|^2\right\}
=e^{is\alpha^\top\dot\psi(h_m)}\mathbb E e^{i\alpha^\top A}.
\label{eq-lem-joint-conv-real-tilt-hm}
\end{equation}
Both sides extend to entire functions of \(s\), since \(H_m\) is Gaussian. Evaluating at \(s=i\) identifies every subsequential characteristic function, giving, along the subsequence in \eqref{eq-lem-joint-conv-subseq-limit},
\begin{equation}
\lim_{N\to\infty}\mathbb E e^{i\alpha^\top A_N+i\Delta_N(h_m)}
=e^{-\alpha^\top\dot\psi(h_m)-\|h_m\|^2/2}\mathbb E e^{i\alpha^\top A}.
\label{eq-lem-joint-conv-imag-tilt-hm}
\end{equation}
By \eqref{eq-lem-joint-conv-central-approx} and \(|e^{ix}-e^{iy}|\le|x-y|\), letting \(m\to\infty\) yields
\begin{equation}
\mathbb E e^{i\alpha^\top A+i\beta^\top G_{\mathrm{sp}}+i\gamma^\top G_\varphi}
=\mathbb E e^{i\alpha^\top A}\exp\left\{-\alpha^\top\mathbb E_{\mathbb P^\pi}[\phi h]-\frac12\|h\|^2\right\}.
\label{eq-lem-joint-conv-char-limit}
\end{equation}
Here \(\mathbb E_{\mathbb P^\pi}[\phi h]=\operatorname{Cov}(G_{\mathrm{sp}},\beta^\top G_{\mathrm{sp}}+\gamma^\top G_\varphi)\) and \(\|h\|^2=\operatorname{Var}(\beta^\top G_{\mathrm{sp}}+\gamma^\top G_\varphi)\). Replacing \(\beta\) by \(\beta-\alpha\) in \eqref{eq-lem-joint-conv-char-limit} therefore gives
\begin{align}
&\mathbb E e^{i\alpha^\top(A-G_{\mathrm{sp}})+i\beta^\top G_{\mathrm{sp}}+i\gamma^\top G_\varphi}\notag\\
&\quad=\mathbb E e^{i\alpha^\top A}e^{\alpha^\top\Sigma_{\theta_0}\alpha/2}
\mathbb E e^{i\beta^\top G_{\mathrm{sp}}+i\gamma^\top G_\varphi}.
\label{eq-lem-joint-conv-factorization}
\end{align}
Taking \((\beta,\gamma)=(0,0)\) identifies the first two factors on the right as \(\mathbb E e^{i\alpha^\top W}\), where \(W=A-G_{\mathrm{sp}}\). Thus \(W\perp(G_{\mathrm{sp}},G_\varphi)\). Local regularity fixes the law of \(A\), and \(\mathbb E e^{i\alpha^\top W}=\mathbb E e^{i\alpha^\top A}e^{\alpha^\top\Sigma_{\theta_0}\alpha/2}\) fixes the law of \(W\). Every subsequential joint limit is therefore the same.
\end{proof}

For \(g\in L^2(P_0^F)\), write \(\mathbb G_N(g)=N^{-1/2}\sum_i\{g(L_i)-\mathbb E[g(L)]\}\), with expectation under \(P_0^F\) even under local alternatives. For \(g\in\mathcal H^F\), this equals \(N^{-1/2}\sum_i g(L_i)\). The corresponding centered Gaussian limits satisfy \(\operatorname{Cov}\{\mathbb G(g),\mathbb G(h)\}=\mathbb E[gh]\) for \(g,h\in\mathcal H^F\). Both notations apply coordinatewise, and only finite-dimensional convergence is used.

Conditional-expectation decompositions also appear in the i.i.d.\ random-sum analysis of \citet{Zhang2005}. The following lemmas establish the joint limits needed for the finite-population target under the dependent design.

\begin{lemma}[Joint observed-data and full-data limits]
\label{lem-joint-full-observed-eif}
Under Assumptions~\ref{ass-A1}--\ref{ass-A2} and~\ref{ass-B1}--\ref{ass-B4}, put \(\Sigma_R=\mathbb E[\{1-\pi(X)\}\operatorname{Var}\{\phi^F(L)\mid X\}]\). Under \(Q_0^N\),
\[
(\Delta_N(\phi),\Delta_N(\mathcal L\phi^F),\mathbb G_N(\phi^F)-\Delta_N(\mathcal L\phi^F))
\xrightarrow[]{d}(G_{\mathrm{sp}},G_L,G_R),
\]
where \((G_{\mathrm{sp}},G_L,G_R)\) is centered Gaussian, \(G_R\perp(G_{\mathrm{sp}},G_L)\), and
\[
\operatorname{Var}(G_{\mathrm{sp}})=\Sigma_{\theta_0},\quad
\operatorname{Cov}(G_{\mathrm{sp}},G_L)=\Sigma_F,\quad
\operatorname{Var}(G_L)=\Sigma_F-\Sigma_R,\quad
\operatorname{Var}(G_R)=\Sigma_R.
\]
Thus \(\mathbb G(\phi^F)=G_L+G_R\) has variance \(\Sigma_F\) and cross-covariance \(\operatorname{Cov}\{G_{\mathrm{sp}},\mathbb G(\phi^F)\}=\Sigma_F\). The residual convergence is stable with respect to \(\mathcal O_N\). Specifically, for every \(\mathcal O_N\)-measurable finite-dimensional \(Z_N\xrightarrow[]{d}Z\),
\begin{equation}
(Z_N,\mathbb G_N(\phi^F)-\Delta_N(\mathcal L\phi^F))\xrightarrow[]{d}(Z,G_R),\qquad G_R\perp Z.
\label{eq-lem-joint-full-observed-stable-consequence}
\end{equation}
\end{lemma}

\begin{proof}
Put \(m_F(X)=\mathbb E[\phi^F(L)\mid X]\) and \(V_F(X)=\operatorname{Var}\{\phi^F(L)\mid X\}\). Then
\begin{equation}
\mathbb G_N(\phi^F)-\Delta_N(\mathcal L\phi^F)
=\frac1{\sqrt N}\sum_{i=1}^N(1-\delta_i)\{\phi^F(L_i)-m_F(X_i)\}.
\label{eq-lem-joint-full-observed-identities}
\end{equation}
Conditionally on \(\mathcal O_N\), the summands are independent and centered, with conditional covariance sum \(N^{-1}\sum_i(1-\delta_i)V_F(X_i)\). For every fixed \(M<\infty\), Assumptions~\ref{ass-B2}--\ref{ass-B3} and the law of large numbers give, entrywise,
\[
\frac1N\sum_{i=1}^N(1-\delta_i)V_F(X_i)\mathbf1\{\|V_F(X_i)\|\le M\}
\xrightarrow[]{p}\mathbb E[\{1-\pi(X)\}V_F(X)\mathbf1\{\|V_F(X)\|\le M\}].
\]
Since \(\mathbb E\|V_F(X)\|\le\mathbb E\|\phi^F(L)\|^2<\infty\), the empirical tail has expected norm at most
\[
\mathbb E[\|V_F(X)\|\mathbf1\{\|V_F(X)\|>M\}]\longrightarrow0,\qquad M\to\infty.
\]
The same bound applies to the limiting tail, so
\begin{equation}
\frac1N\sum_{i=1}^N(1-\delta_i)V_F(X_i)\xrightarrow[]{p}\Sigma_R.
\label{eq-lem-joint-full-observed-cond-var}
\end{equation}
For every \(\varepsilon>0\), the conditional Lindeberg sum is
\[
\frac1N\sum_{i=1}^N(1-\delta_i)\mathbb E\!\left[\|\phi^F(L_i)-m_F(X_i)\|^2
\mathbf1\{\|\phi^F(L_i)-m_F(X_i)\|>\varepsilon\sqrt N\}\mid X_i\right].
\]
Its expectation is bounded by \(\mathbb E[\|\phi^F-m_F\|^2\mathbf1\{\|\phi^F-m_F\|>\varepsilon\sqrt N\}]\to0\). Lemma~\ref{lem-conditional-lindeberg-feller} and Remark~\ref{rem-stable-convergence-use} therefore prove \eqref{eq-lem-joint-full-observed-stable-consequence}.

Write \(\phi_j=\mathcal L \phi^{*}_j\), where \(\Proj[\mathcal M \phi^{*}_j\mid\mathcal H^F]=\phi_j^F\). Lemma~\ref{lem-central-sequence-L2-limits} applied to the coordinates of \(\phi^{*}\) and \(\phi^F\) gives \((\Delta_N(\phi),\Delta_N(\mathcal L\phi^F))\xrightarrow[]{d}(G_{\mathrm{sp}},G_L)\). For \(1\le j,k\le p\),
\begin{align*}
\operatorname{Cov}\{(G_{\mathrm{sp}})_j,(G_L)_k\}
&=\langle\mathcal L \phi^{*}_j,\mathcal L\phi_k^F\rangle
=\langle\mathcal M \phi^{*}_j,\phi_k^F\rangle_F
=\langle\phi_j^F,\phi_k^F\rangle_F,\\
\operatorname{Var}(G_L)&=\mathbb E[m_F(X)m_F(X)^\top+\pi(X)V_F(X)]=\Sigma_F-\Sigma_R.
\end{align*}
Equation~\eqref{eq-lem-joint-full-observed-stable-consequence} supplies the independent \(G_R\). Adding \(G_R\) to \(G_L\) gives \(\mathbb G(\phi^F)\) and the remaining covariances.
\end{proof}

\begin{lemma}[Joint full-data LAN and local expansion of the finite-population target]
\label{lem-joint-full-data-lan}
Under Assumptions~\ref{ass-A1}--\ref{ass-A2} and~\ref{ass-B1}, for every \(g\in\dot{\mathcal P}_{P_0^F}\cap L^4(P_0^F)\) and fixed \(u\in\mathbb R\),
\begin{equation}
\log\frac{dQ^N_{u/\sqrt N,g}}{dQ_0^N}
=u\mathbb G_N(g)-\frac{u^2}{2}\|g\|_F^2+o_{Q_0^N}(1),\qquad
\mathbb G_N(g)\xrightarrow[]{d}N(0,\|g\|_F^2).
\label{eq-joint-full-data-lan}
\end{equation}
The laws \(Q^N_{u/\sqrt N,g}\) and \(Q_0^N\) are mutually contiguous. If Assumption~\ref{ass-A3} also holds, then, under \(Q^N_{u/\sqrt N,g}\),
\begin{equation}
\sqrt N\{\theta_N-\psi(P^F_{u/\sqrt N,g})\}
=\mathbb G_N(\phi^F)-u\dot\psi^F(g)+o_{Q^N_{u/\sqrt N,g}}(1).
\label{eq-local-finite-target-expansion}
\end{equation}
\end{lemma}

\begin{proof}
The conditional design density \(q_N(\boldsymbol\delta_N\mid\bm X_N)\) does not depend on the submodel parameter. Therefore it cancels from the joint likelihood ratio, and \(\frac{dQ^N_{u/\sqrt N,g}}{dQ_0^N} =\prod_{i=1}^N\frac{dP^F_{u/\sqrt N,g}}{dP_0^F}(L_i)\). The usual i.i.d.\ LAN expansion implied by full-data DQM gives \eqref{eq-joint-full-data-lan}. Le Cam's first lemma gives mutual contiguity. By Assumption~\ref{ass-A3}, \(\sqrt N(\theta_N-\theta_0)=\mathbb G_N(\phi^F)+o_{Q_0^N}(1)\). Contiguity transfers the remainder to \(Q^N_{u/\sqrt N,g}\). Assumption~\ref{ass-A2} gives \(\sqrt N\{\psi(P^F_{u/\sqrt N,g})-\theta_0\} =u\dot\psi^F(g)+o(1)\), and subtraction proves \eqref{eq-local-finite-target-expansion}.
\end{proof}

\begin{lemma}[Joint convolution with full-data scores]
\label{lem-joint-convolution-full-scores}
Under Assumptions~\ref{ass-A1}--\ref{ass-A2} and~\ref{ass-B1}--\ref{ass-B4}, let \(\hat\theta_N\) be locally regular for the superpopulation target and put \(A_N=\sqrt N(\hat\theta_N-\theta_0)\). For every integer \(q\ge0\) and fixed \(g_1,\ldots,g_q\in\mathcal H^F\), under \(Q_0^N\),
\begin{equation}
\bigl(A_N,\mathbb G_N(\phi^F),\mathbb G_N(g_1),\ldots,\mathbb G_N(g_q)\bigr)
\xrightarrow[]{d}\bigl(G_{\mathrm{sp}}+W,\mathbb G(\phi^F),\mathbb G(g_1),\ldots,\mathbb G(g_q)\bigr),
\label{eq-joint-convolution-full-scores}
\end{equation}
where \((G_{\mathrm{sp}},\mathbb G(\phi^F),\mathbb G(g_1),\ldots,\mathbb G(g_q))\) is centered Gaussian and independent of \(W\), and
\begin{equation}
\operatorname{Cov}\{G_{\mathrm{sp}},\mathbb G(g_\ell)\}
=\operatorname{Cov}\{\mathbb G(\phi^F),\mathbb G(g_\ell)\}
=\mathbb E[\phi^F g_\ell]=\dot\psi^F(g_\ell).
\label{eq-joint-convolution-full-score-covariances}
\end{equation}
For \(q=0\), all auxiliary score coordinates are omitted. The covariance of \((G_{\mathrm{sp}},\mathbb G(\phi^F))\) is given in Lemma~\ref{lem-joint-full-observed-eif}.
\end{lemma}

\begin{proof}
Apply Lemma~\ref{lem-joint-convolution-observed-central} to the coordinates of \(\mathcal L\phi^F\) and to \(\mathcal Lg_1,\ldots,\mathcal Lg_q\). The conditional variance, truncation, and Lindeberg bounds in Lemma~\ref{lem-joint-full-observed-eif} apply to the full-data vector \((\phi^F,g_1,\ldots,g_q)\), whose coordinates lie in \(L^2(P_0^F)\). Its residual after conditioning on \(\mathcal O_N\) therefore has a Gaussian limit independent of every observed-data measurable limit. Adding this residual to the observed-score vector proves \eqref{eq-joint-convolution-full-scores} and preserves independence of \(W\). Finally,
\[
\operatorname{Cov}\{G_{\mathrm{sp}},\mathbb G(g_\ell)\}
=\langle\phi,\mathcal Lg_\ell\rangle
=\dot\psi(\mathcal Lg_\ell)
=\dot\psi^F(g_\ell)
=\mathbb E[\phi^F g_\ell],
\]
whereas the full-data i.i.d.\ covariance gives \(\operatorname{Cov}\{\mathbb G(\phi^F),\mathbb G(g_\ell)\}=\mathbb E[\phi^F g_\ell]\).
\end{proof}

\begin{proof}[Proof of Theorem~\ref{thm-local-regularity-equivalence}]
Write \(A_N=\sqrt N(\hat\theta_N-\theta_0)\) and \(B_N=\sqrt N(\hat\theta_N-\theta_N)\). Assumption~\ref{ass-A3} gives \(B_N=A_N-\mathbb G_N(\phi^F)+o_{Q_0^N}(1)\).

\emph{From superpopulation to finite-population regularity.}
Suppose Definition~\ref{def-local-regularity}\textup{(a)} holds. For every admissible \(g\), Lemma~\ref{lem-joint-convolution-full-scores} and Assumption~\ref{ass-A3} give, under \(Q_0^N\),
\[
(B_N,\mathbb G_N(g))\xrightarrow[]{d}(B,\mathbb G(g)),\qquad
B=G_{\mathrm{sp}}-\mathbb G(\phi^F)+W.
\]
By \eqref{eq-joint-convolution-full-score-covariances}, \(\operatorname{Cov}\{G_{\mathrm{sp}}-\mathbb G(\phi^F),\mathbb G(g)\}=\dot\psi^F(g)-\dot\psi^F(g)=0\). Joint Gaussianity and \(W\perp(G_{\mathrm{sp}},\mathbb G(\phi^F),\mathbb G(g))\) imply \(B\perp\mathbb G(g)\). Apply Le Cam's third lemma to \eqref{eq-joint-full-data-lan} to obtain \(B_N\xrightarrow[]{d}B\) under every \(Q^N_{u/\sqrt N,g}\). Its law is independent of \(u\) and \(g\), proving Definition~\ref{def-local-regularity}\textup{(b)}.

\emph{From finite-population to superpopulation regularity.}
Suppose Definition~\ref{def-local-regularity}\textup{(b)} holds, and denote the common limit of \(B_N\) by \(B\). For admissible \(g_1,\ldots,g_q\), tightness permits a subsequential joint limit \((B,\mathbb G(g_1),\ldots,\mathbb G(g_q))\). For \(c\in\mathbb R^q\), Assumption~\ref{ass-A1} makes \(g_c=\sum_jc_jg_j\) admissible. Le Cam's third lemma, \eqref{eq-joint-full-data-lan}, and regularity for the finite-population target give, for every \(t\in\mathbb R^p\) and \(u\in\mathbb R\),
\begin{equation}
\mathbb E\exp\left\{it^\top B+u\sum_{j=1}^q c_j\mathbb G(g_j)-\frac{u^2}{2}\|g_c\|_F^2\right\}
=\mathbb E e^{it^\top B}.
\label{eq-fp-regularity-tilt-invariance}
\end{equation}
As in \eqref{eq-lem-joint-conv-real-tilt-hm}--\eqref{eq-lem-joint-conv-imag-tilt-hm}, analytic continuation to \(u=i\) gives \(B\perp(\mathbb G(g_1),\ldots,\mathbb G(g_q))\). This identifies every subsequential joint limit. For each coordinate of \(\phi^F\), Assumption~\ref{ass-A1} supplies admissible \(g_m\to\phi_j^F\), and
\[
\sup_N\mathbb E|\mathbb G_N(\phi_j^F)-\mathbb G_N(g_m)|^2
=\|\phi_j^F-g_m\|_F^2\longrightarrow0.
\]
The same characteristic-function approximation as in Lemma~\ref{lem-central-sequence-L2-limits} therefore gives joint convergence of \((B_N,\mathbb G_N(\phi^F),\mathbb G_N(g))\), with \(B\) independent of the Gaussian coordinates. Since \(\operatorname{Cov}\{\mathbb G(\phi^F),\mathbb G(g)\}=\dot\psi^F(g)\), Le Cam's third lemma gives, under \(Q^N_{u/\sqrt N,g}\),
\[
(B_N,\mathbb G_N(\phi^F)-u\dot\psi^F(g))\xrightarrow[]{d}(B,\mathbb G(\phi^F)).
\]
Lemma~\ref{lem-joint-full-data-lan} now yields
\[
\sqrt N\{\hat\theta_N-\psi(P^F_{u/\sqrt N,g})\}
=B_N+\mathbb G_N(\phi^F)-u\dot\psi^F(g)+o_{Q^N_{u/\sqrt N,g}}(1)
\xrightarrow[]{d}B+\mathbb G(\phi^F).
\]
The limit law is independent of \(u\) and \(g\). The statistic on the left is \(\mathcal O_N\)-measurable, so its law under \(Q^N_{u/\sqrt N,g}\) equals its law under the observed marginal \(P^N_{u/\sqrt N,\mathcal Lg}\). This proves Definition~\ref{def-local-regularity}\textup{(a)}.
\end{proof}

\begin{proof}[Proof of Theorem~\ref{thm-convolution}, finite-population part]
Under Assumption~\ref{ass-A3}, Theorem~\ref{thm-local-regularity-equivalence} makes either regularity condition sufficient. Put \(A_N=\sqrt N(\hat\theta_N-\theta_0)\) and \(B_N=\sqrt N(\hat\theta_N-\theta_N)\). Lemma~\ref{lem-joint-convolution-full-scores}, with no auxiliary score, gives
\[
(A_N,\mathbb G_N(\phi^F))\xrightarrow[]{d}(G_{\mathrm{sp}}+W,\mathbb G(\phi^F)),\qquad
W\perp(G_{\mathrm{sp}},\mathbb G(\phi^F)).
\]
Since \(B_N=A_N-\mathbb G_N(\phi^F)+o_{Q_0^N}(1)\), setting \(G_{\mathrm{fp}}=G_{\mathrm{sp}}-\mathbb G(\phi^F)\) gives
\[
(A_N,B_N)\xrightarrow[]{d}(G_{\mathrm{sp}}+W,G_{\mathrm{fp}}+W),\qquad
W\perp(G_{\mathrm{sp}},G_{\mathrm{fp}}).
\]
The covariances in Lemma~\ref{lem-joint-full-observed-eif} imply \(\operatorname{Var}(G_{\mathrm{fp}})=\Sigma_{\theta_0}+\Sigma_F-2\Sigma_F=\Sigma_{\theta_N}\) and \(\operatorname{Cov}(G_{\mathrm{sp}},G_{\mathrm{fp}})=\Sigma_{\theta_N}\). Thus \(\Sigma_{\theta_N}\) is positive semidefinite, and the same residual \(W\) determines efficiency for the superpopulation and finite-population targets.
\end{proof}

\begin{proof}[Proof of Theorem~\ref{thm-efficiency}]
Put \(A_N=\sqrt N(\hat\theta_N-\theta_0)\). Suppose first that \(A_N=\Delta_N(\phi)+o_{P_0^N}(1)\). Write \(\phi=\mathcal L \phi^{*}\) coordinatewise. Lemma~\ref{lem-central-sequence-L2-limits} gives joint Gaussian convergence of \(\Delta_N(\phi)\) and \(\Delta_N(\mathcal Lg)\), with cross-covariance \(\mathbb E_{\mathbb P^\pi}[\phi\mathcal Lg]=\dot\psi(\mathcal Lg)\). LAN transfers the remainder to \(P^N_{u/\sqrt N,\mathcal Lg}\) by contiguity. Le Cam's third lemma and pathwise differentiability give, under that local law,
\[
\sqrt N\{\hat\theta_N-\psi(P^N_{u/\sqrt N,\mathcal Lg})\}
=\Delta_N(\phi)-u\dot\psi(\mathcal Lg)+o_{P^N_{u/\sqrt N,\mathcal Lg}}(1)
\xrightarrow[]{d}G_{\mathrm{sp}}.
\]
This proves local regularity and efficiency for the superpopulation target. Conversely, for a locally regular estimator, Lemma~\ref{lem-joint-convolution-observed-central} gives \((A_N,\Delta_N(\phi))\xrightarrow[]{d}(G_{\mathrm{sp}}+W,G_{\mathrm{sp}})\). Efficiency makes \(W=0\) almost surely, so \(A_N-\Delta_N(\phi)=o_{P_0^N}(1)\), proving \eqref{eq-best-regular-sp}.

Under Assumption~\ref{ass-A3}, Theorems~\ref{thm-local-regularity-equivalence} and~\ref{thm-convolution} identify the same regular and efficient estimators for the superpopulation and finite-population targets. Moreover,
\begin{align*}
&\sqrt N(\hat\theta_N-\theta_N)-\{\Delta_N(\phi)-\mathbb G_N(\phi^F)\}\\
&\quad=A_N-\Delta_N(\phi)-\{\sqrt N(\theta_N-\theta_0)-\mathbb G_N(\phi^F)\}
=A_N-\Delta_N(\phi)+o_{Q_0^N}(1).
\end{align*}
Since \(A_N-\Delta_N(\phi)\) is \(\mathcal O_N\)-measurable, it is \(o_{Q_0^N}(1)\) if and only if it is \(o_{P_0^N}(1)\). This proves the equivalence of \eqref{eq-best-regular-sp} and~\eqref{eq-best-regular-fp}.
\end{proof}

\subsubsection{Optimal sampling design}
\label{sec-optimal-design}

For a scalar target, we optimize the limiting inclusion probability. For each candidate \(\pi\), let \(\phi^{*}_\pi\) solve \eqref{eq-obs-eif-equation} with \(\mathcal L\) and \(\mathcal M\) constructed from \(\pi\). Let \(\Sigma_{\theta_0}(\pi)=\mathbb E_{\mathbb P^\pi}[(\mathcal L \phi^{*}_\pi)(O)^2]\).
Under Assumption~\ref{ass-A3}, the corresponding finite-population bound is \(\Sigma_{\theta_N}(\pi)=\Sigma_{\theta_0}(\pi)-\Sigma_F\). For a given design and budget, let \(\mathfrak D_\rho\) be the feasible class of limiting inclusion probabilities. Typical constraints are \(\mathbb E[\pi(X)]\le\rho\) or \(\mathbb E[c(X)\pi(X)]\le C\),
where \(\rho>0\) and \(C>0\) are budgets and \(c(X)\ge0\) is the measurement cost of the study variable. If \(c\in L^2(P_0^F)\), each constraint defines a closed half-space in \(L^2(P_0^F)\). Section~\ref{app-sampling-examples} specifies feasible classes for each sampling design in Table~\ref{tab-sampling-examples} under some budget constraints. At a fixed sampling fraction, SRSWOR and cluster sampling admit only \(\pi(x)\equiv\rho\), whereas Poisson, PPS, rejective, stratified, and self-weighted \(\pi\)PS/SRS designs allow nonconstant limits.

The following theorem characterizes optimal limiting inclusion probabilities for scalar pathwise-differentiable targets in general semiparametric full-data models. Implementing an optimizer additionally requires finite-population inclusion probabilities \(\pi_{N,i}\) satisfying Assumption~\ref{ass-B2}.

\begin{theorem}[Optimal limiting inclusion probability]
\label{thm-optimal-design}
Suppose that Assumptions~\ref{ass-A1}--\ref{ass-A2} hold and that \(\theta_0\in\mathbb R\). Let \(\mathfrak D_\rho\) be a nonempty, closed, convex subset of \(L^2(P_0^F)\) consisting of \(\sigma(X)\)-measurable functions satisfying \(\pi_-\le\pi(X)\le1\) a.s. for a common constant \(\pi_->0\). Assume further that every \(\pi\in\mathfrak D_\rho\) is attainable by a sequence of sampling designs satisfying Assumptions~\ref{ass-B1}--\ref{ass-B4} as the limit in Assumption~\ref{ass-B2}.
Then \(\pi\mapsto\Sigma_{\theta_0}(\pi)\) is convex and weakly lower semicontinuous on \(\mathfrak D_\rho\) and attains its minimum. A function \(\pi^*\in\mathfrak D_\rho\) is a minimizer of \(\Sigma_{\theta_0}(\pi)\) if and only if
\begin{equation}
\mathbb E\!\left[
\Var\{\phi^{*}_{\pi^*}(L)\mid X\}
\{\pi(X)-\pi^*(X)\}
\right]
\le0,
\qquad
\pi\in\mathfrak D_\rho.
\label{eq-optimal-design-variational-inequality}
\end{equation}
Under Assumption~\ref{ass-A3}, the sets of minimizers of \(\Sigma_{\theta_0}\) and \(\Sigma_{\theta_N}\) over \(\mathfrak D_\rho\) coincide.
\end{theorem}
The equality of the two sets of minimizers follows from the variance decomposition, since \(\Sigma_F\) does not depend on the design. An optimal limiting inclusion probability need not be unique, but strict convexity of $\Sigma_{\theta_0}$ on $\mathfrak D_\rho$ ensures uniqueness. For $\phi^{*}_{\pi^*}$ in \eqref{eq-obs-eif-equation}, the residual $\phi^{*}_{\pi^*}(L)-\mathbb E[\phi^{*}_{\pi^*}(L)\mid X]$ is the component unrecoverable from covariates and revealed only when $Y$ is observed. Its conditional variance measures the marginal reduction in the efficiency bound as inclusion probability increases, so \eqref{eq-optimal-design-variational-inequality} says that no feasible reallocation can improve precision.

For the mean targets in Example~\ref{ex-eif-mean}, \(\operatorname{Var}\{\phi^{*}_\pi(L)\mid X\}=\operatorname{Var}(Y\mid X)/\pi(X)^2\).
With unit costs and \(X\) indicating membership in fixed strata, an interior optimum satisfies \(p_j\propto\{\operatorname{Var}(Y\mid X\in\mathcal X_j)\}^{1/2}\), which is the Neyman allocation \citep{Neyman1934}.
For the scalar target \(\beta_{0,j}\) in Example~\ref{ex-eif-regression}, \(\operatorname{Var}\{\phi^{*}_\pi(L)\mid X\}=\{(I_\pi^{-1}\dot\mu_0(X))_j\}^2/\sigma^2(X)\). Thus the allocation criterion in \eqref{eq-optimal-design-variational-inequality} depends on both the coefficient of interest and the conditional variance. 

\citet[Section~5.1]{Armstrong2022} derives allocation conditions from an explicit expression for the EIF, with resource constraints given by finitely many linear expectation inequalities. Our criterion instead incorporates full-data model restrictions through the projection equation \eqref{eq-obs-eif-equation} and allows feasible classes of survey inclusion probabilities that need not be defined by such constraints.

\begin{proof}[Proof of Theorem~\ref{thm-optimal-design}]
To express the scalar bound variationally, put \(b(g)=\mathbb E[\phi^F g]\) and
\[
a_\pi(g,h)=\mathbb E[\pi(X)\Cov(g,h\mid X)+\mathbb E[g\mid X]\mathbb E[h\mid X]]
=\langle\mathcal Lg,\mathcal Lh\rangle,
\]
where the inner product uses \(\pi\). The normal equation \eqref{eq-obs-eif-equation} becomes \(a_\pi(\phi^{*}_\pi,g)=b(g)\) for every \(g\in\mathcal H^F\), and
\begin{equation}
\pi_-\|g\|_F^2\le a_\pi(g,g)\le\|g\|_F^2.
\label{eq-optimal-design-coercivity}
\end{equation}
Completing the square gives \(2b(g)-a_\pi(g,g)=a_\pi(\phi^{*}_\pi,\phi^{*}_\pi)-a_\pi(g-\phi^{*}_\pi,g-\phi^{*}_\pi)\), hence
\begin{equation}
\Sigma_{\theta_0}(\pi)=\sup_{g\in\mathcal H^F}\{2b(g)-a_\pi(g,g)\},
\label{eq-optimal-design-variational-representation}
\end{equation}
with unique maximizer \(\phi^{*}_\pi\). Each objective is affine in \(\pi\), so the bound is convex. By Assumption~\ref{ass-A1} and continuity in \(g\), the supremum is unchanged when restricted to \(\dot{\mathcal P}_{P_0^F}\cap L^4(P_0^F)\). For these \(g\), \(\mathbb E[\Var(g\mid X)^2]\le\mathbb E[g^4]<\infty\), making the objective weakly continuous in \(\pi\). The supremum is therefore weakly lower semicontinuous. The bounded, closed, convex set \(\mathfrak D_\rho\) is weakly compact in \(L^2(P_0^F)\), so the minimum is attained.

To identify the minimizers, fix \(\pi,\widetilde\pi\in\mathfrak D_\rho\), and put \(h=\widetilde\pi-\pi\), \(\pi_t=\pi+th\), and \(\phi^{*}_t=\phi^{*}_{\pi_t}\) for \(0\le t\le1\). The pointwise bounds imply \(h\in L^\infty(P_0^F)\). Subtracting the normal equations gives \(a_\pi(\phi^{*}_t-\phi^{*}_\pi,g)=-t\mathbb E[h(X)\Cov(\phi^{*}_t,g\mid X)]\). Take \(g=\phi^{*}_t-\phi^{*}_\pi\) and use Cauchy--Schwarz and \eqref{eq-optimal-design-coercivity} to obtain
\[
\|\phi^{*}_t-\phi^{*}_\pi\|_F\le\frac{t\|h\|_\infty}{\pi_-}\|\phi^{*}_t\|_F,\qquad
\sup_{0\le t\le1}\|\phi^{*}_t\|_F\le\frac{\|\phi^F\|_F}{\pi_-}.
\]
Thus \(\phi^{*}_t\to \phi^{*}_\pi\) in \(L^2(P_0^F)\) and \(\Var(\phi^{*}_t\mid X)\to\Var(\phi^{*}_\pi\mid X)\) in \(L^1(P_0^F)\). Evaluating the variational objectives at \(\phi^{*}_\pi\) and \(\phi^{*}_t\) yields
\[
-\mathbb E[h\Var(\phi^{*}_\pi\mid X)]
\le\frac{\Sigma_{\theta_0}(\pi_t)-\Sigma_{\theta_0}(\pi)}{t}
\le-\mathbb E[h\Var(\phi^{*}_t\mid X)].
\]
Consequently, the right derivative at \(t=0\) equals \(-\mathbb E[\Var(\phi^{*}_\pi\mid X)(\widetilde\pi-\pi)]\). Convexity implies that \(\pi^*\) minimizes the bound exactly when these derivatives are nonnegative for every feasible \(\widetilde\pi\), which is \eqref{eq-optimal-design-variational-inequality}. Under Assumption~\ref{ass-A3}, Theorem~\ref{thm-convolution} gives \(\Sigma_{\theta_N}(\pi)=\Sigma_{\theta_0}(\pi)-\Sigma_F\). Since \(\Sigma_F\) does not depend on \(\pi\), both bounds have the same minimizers.
\end{proof}

\subsection{Verification for the sampling designs}
\label{app-sampling-examples}

This subsection verifies Assumptions~\ref{ass-B1}--\ref{ass-B4} for the designs in Table~\ref{tab-sampling-examples}. For each design, the sampling setting and the regularity conditions are stated before the proposition. Every design below is determined by \(\bm X_N\) and auxiliary randomization independent of \(\bm Y_N\), so Assumption~\ref{ass-B1} is immediate. The proofs therefore focus on Assumptions~\ref{ass-B2}--\ref{ass-B4}.

All finite-population definitions and sampling operations within a stratum apply only when it is nonempty. For empty strata, counts, ratios, and extrema are set to zero.

For use with Theorem~\ref{thm-optimal-design}, we record a closed and convex feasible class of limiting inclusion probabilities for each design under an asymptotic budget and, when needed, a finite-population construction showing attainability. For stratified two-stage sampling, the class is specified for the self-weighted \(\pi\)PS/SRS specialization. Except for PPS sampling, \(\rho\) denotes an inclusion or sampling fraction. For PPS sampling, it denotes the number of draws per population unit. Each class may also be intersected with
\(\{\pi:\mathbb E[c(X)\pi(X)]\le C\}\), where \(c\in L^2(P_0^F)\) and \(c(X)\ge0\), provided that the intersection is nonempty.

For Assumption~\ref{ass-B3}, it is enough to consider nonnegative \(h\). Indeed, writing \(h=h^+-h^-\) and using
\[
\Var(A+B\mid\bm X_N)
\le
2\Var(A\mid\bm X_N)+2\Var(B\mid\bm X_N)
\]
reduces the general case to \(h^+\) and \(h^-\).

\subsubsection{A variance bound}

Consider sampling indicators with first-order inclusion probabilities \(\pi_{N,i}\). Impose the conditional covariance condition
\begin{equation}
\Cov(\delta_i,\delta_j\mid\bm X_N)\le0,
\qquad i\ne j.
\label{eq-negative-dependence-condition}
\end{equation}

\begin{lemma}[Variance bound under non-positive pairwise covariances]
\label{lem-negative-dependence-bound}
Under \eqref{eq-negative-dependence-condition}, Assumption~\ref{ass-B3} holds.
\end{lemma}

\begin{proof}
For nonnegative \(h\), write \(h_i=h(X_i)\). Then
\begin{align*}
0
&\le
\Var\left(
\frac1N\sum_{i=1}^N(\delta_i-\pi_{N,i})h_i
\,\middle|\,\bm X_N
\right)
\le
\frac1{N^2}\sum_{i=1}^Nh_i^2
=
\frac1N\left\{\frac1N\sum_{i=1}^Nh(X_i)^2\right\}
=o_{P_0^N}(1).
\end{align*}
\end{proof}

\subsubsection{Poisson sampling}

Conditionally on \(\bm X_N\), let \(\delta_1,\ldots,\delta_N\) be independent Bernoulli variables with
\[
P_0^N(\delta_i=1\mid\bm X_N)=\pi(X_i).
\]
We impose the following assumption.
\begin{enumerate}[label=\textup{(\alph*)}]
\item The function \(\pi\) satisfies
\begin{equation}
0<\pi_-\le\pi(X)\le1
\quad\text{a.s.}
\label{eq-poisson-sampling-positivity}
\end{equation}
\end{enumerate}

\begin{proposition}[Poisson sampling]
\label{prop-poisson-sampling}
Under the preceding assumption, Assumptions~\ref{ass-B1}--\ref{ass-B4} hold.
\end{proposition}

\begin{proof}
Assumption~\ref{ass-B2} is immediate because \(\pi_{N,i}=\pi(X_i)\). Conditional independence gives \eqref{eq-negative-dependence-condition}, so Lemma~\ref{lem-negative-dependence-bound} gives Assumption~\ref{ass-B3}. Equation~\eqref{eq-poisson-sampling-positivity} gives Assumption~\ref{ass-B4}.
\end{proof}

Under an upper bound \(\rho\) on the expected sampling fraction, the following feasible class of limiting inclusion probabilities is closed and convex.
\[
\mathfrak D_\rho^{\mathrm{Poi}}
=
\left\{
\pi:
\pi_-\le\pi(X)\le1\ \text{a.s.},
\quad
\mathbb E[\pi(X)]\le\rho
\right\}.
\]

\subsubsection{SRSWOR}

A sample of fixed size \(n_N\) is drawn uniformly without replacement from the \(N\) units. We impose the following assumption.
\begin{enumerate}[label=\textup{(\alph*)}]
\item The sampling fraction satisfies
\begin{equation}
\frac{n_N}{N}\to\rho
\quad\text{for some }\rho\in(0,1).
\label{eq-srs-sampling-fraction}
\end{equation}
\end{enumerate}

\begin{proposition}[SRSWOR]
\label{prop-srswor}
Under the preceding setting and assumption, Assumptions~\ref{ass-B1}--\ref{ass-B4} hold with \(\pi(x)\equiv\rho\).
\end{proposition}

\begin{proof}
For every unit, \(\pi_{N,i}=n_N/N\), so Assumption~\ref{ass-B2} follows from \eqref{eq-srs-sampling-fraction}. The indicators are negatively correlated, and Lemma~\ref{lem-negative-dependence-bound} gives Assumption~\ref{ass-B3}. Equation~\eqref{eq-srs-sampling-fraction} gives Assumption~\ref{ass-B4}.
\end{proof}

At the fixed sampling fraction considered here, the feasible class of limiting inclusion probabilities is
\[
\mathfrak D_\rho^{\mathrm{SRS}}
=
\{\pi:\pi(x)\equiv\rho\}.
\]
Thus there is no design choice at the level of the limiting inclusion probability.

\subsubsection{Cluster sampling}

Let
\[
\{1,\ldots,N\}=\bigsqcup_{a=1}^{M_N}G_{N,a}
\]
be an \(\bm X_N\)-measurable cluster partition. Draw \(m_N\) clusters by SRSWOR from the \(M_N\) clusters and observe every unit in each selected cluster. Let \(A_{N,a}\) denote the indicator that cluster \(a\) is selected. Such a partition may group units with similar outcome distributions through their covariates. Assumption~\ref{ass-B1} excludes dependence between \(\boldsymbol\delta_N\) and \(\bm Y_N\) that remains after conditioning on \(\bm X_N\). The following conditions are imposed.
\begin{enumerate}[label=\textup{(\alph*)}]
\item The cluster sampling fraction satisfies
\[
\frac{m_N}{M_N}\to\rho\in(0,1).
\]

\item The largest cluster is asymptotically negligible,
\[
\max_{1\le a\le M_N}\frac{|G_{N,a}|}{N}
=o_{P_0^N}(1).
\]
\end{enumerate}

\begin{proposition}[Cluster sampling]
\label{prop-cluster-sampling}
Under the preceding cluster sampling setting and conditions, Assumptions~\ref{ass-B1}--\ref{ass-B4} hold with \(\pi(x)\equiv\rho\).
\end{proposition}

\begin{proof}
For \(i\in G_{N,a}\), \(\delta_i=A_{N,a}\) and \(\pi_{N,i}=m_N/M_N\), which gives Assumptions~\ref{ass-B2} and~\ref{ass-B4}. For nonnegative \(h\), put \(H_{N,a}=\sum_{i\in G_{N,a}}h(X_i)\). The non-positive pairwise covariances of the indicators for cluster selection give
\begin{align*}
&\Var\left(
\frac1N\sum_{i=1}^N(\delta_i-\pi_{N,i})h(X_i)
\,\middle|\,\bm X_N
\right)\\
&\quad=
\frac1{N^2}\Var\left(
\sum_{a=1}^{M_N}\left(A_{N,a}-\frac{m_N}{M_N}\right)H_{N,a}
\,\middle|\,\bm X_N
\right)\\
&\quad\le
\frac1{N^2}\sum_{a=1}^{M_N}H_{N,a}^2
\le
\left(\max_{1\le a\le M_N}\frac{|G_{N,a}|}{N}\right)
\left\{\frac1N\sum_{i=1}^Nh(X_i)^2\right\}
=o_{P_0^N}(1).
\end{align*}
Thus Assumption~\ref{ass-B3} holds.
\end{proof}

At the fixed cluster sampling fraction considered here, the feasible class of limiting inclusion probabilities is again
\[
\mathfrak D_\rho^{\mathrm{Cl}}
=
\{\pi:\pi(x)\equiv\rho\}.
\]

\subsubsection{PPS sampling with replacement}

Conditionally on \(\bm X_N\), draw \(J_{N,1},\ldots,J_{N,n_N}\) independently with
\[
P_0^N(J_{N,\ell}=i\mid\bm X_N)
=w_{N,i}
=
\frac{r(X_i)}{\sum_{j=1}^Nr(X_j)}.
\]
Let
\[
C_{N,i}=\sum_{\ell=1}^{n_N}\mathbf 1\{J_{N,\ell}=i\},
\qquad
\delta_i=\mathbf 1\{C_{N,i}\ge1\}.
\]
Each distinct selected unit contributes one observed value of \(Y_i\). Retaining the multiplicities \(C_{N,i}\) as design information would not change the likelihood ratio or the efficiency bound, because their conditional law given \(\bm X_N\) is known and repeated draws reveal no additional value of \(Y_i\). The following conditions are imposed.
\begin{enumerate}[label=\textup{(\alph*)}]
\item The size measure satisfies
\[
0<c_r\le r(X)\le C_r<\infty
\quad\text{a.s.}
\]

\item The number of draws satisfies
\[
\frac{n_N}{N}\to\rho\in(0,\infty).
\]
\end{enumerate}
Put \(\mu_r=\E[r(X)]\).

\begin{proposition}[PPS sampling with replacement]
\label{prop-pps-with-replacement}
Under the preceding setting and conditions for PPS sampling with replacement, Assumptions~\ref{ass-B1}--\ref{ass-B4} hold with
\[
\pi(x)=1-\exp\left\{-\rho\frac{r(x)}{\mu_r}\right\}.
\]
\end{proposition}

\begin{proof}
The first-order inclusion probability is
\[
\pi_{N,i}=1-(1-w_{N,i})^{n_N}.
\]
The law of large numbers and boundedness of \(r\) give
\[
\sup_i\left|n_Nw_{N,i}-\rho\frac{r(X_i)}{\mu_r}\right|
=o_{P_0^N}(1),
\qquad
\max_iw_{N,i}=O_{P_0^N}(N^{-1}).
\]
Hence
\[
\sup_i\left|\pi_{N,i}-1+\exp\left\{-\rho\frac{r(X_i)}{\mu_r}\right\}\right|
=o_{P_0^N}(1),
\]
which proves Assumption~\ref{ass-B2}. For \(i\ne j\),
\begin{align*}
P_0^N(\delta_i=1,\delta_j=1\mid\bm X_N)
&=1-(1-w_{N,i})^{n_N}-(1-w_{N,j})^{n_N}\\
&\quad+(1-w_{N,i}-w_{N,j})^{n_N}.
\end{align*}
Since \(1-w_{N,i}-w_{N,j}\le(1-w_{N,i})(1-w_{N,j})\), the indicators that a unit is selected at least once have non-positive conditional covariances. Lemma~\ref{lem-negative-dependence-bound} gives Assumption~\ref{ass-B3}. The lower bound on \(r\) gives Assumption~\ref{ass-B4}.
\end{proof}

Fix \(0<a<\rho<b<\infty\). Under a maximum draw rate \(\rho\), the following feasible class of limiting inclusion probabilities is closed and convex.
\[
\mathfrak D_\rho^{\mathrm{PPS}}
=
\left\{
\pi:
a\le-\log\{1-\pi(X)\}\le b\ \text{a.s.},
\quad
\mathbb E[-\log\{1-\pi(X)\}]\le\rho
\right\}.
\]
For any \(\pi\) in this class, take
\[
r(X)=-\log\{1-\pi(X)\},
\qquad
\frac{n_N}{N}\to\mathbb E[r(X)].
\]
The preceding PPS construction then has limiting inclusion probability \(\pi\). The pointwise bounds verify the required positivity and boundedness and define a closed interval for \(\pi(X)\). The aggregate constraint defines a convex and \(L^2(P_0^F)\)-closed set because \(p\mapsto-\log(1-p)\) is convex and Lipschitz on this interval.

\subsubsection{Rejective sampling}

Let the sample be drawn by fixed-size rejective sampling with sample size \(n_N\). Its prescribed first-order inclusion probabilities are
\[
\pi_{N,i}
=
\frac{n_Nr(X_i)}{\sum_{j=1}^Nr(X_j)}.
\]
The law of rejective sampling is obtained by conditioning independent Bernoulli variables with success probabilities \(p_{N,1},\ldots,p_{N,N}\) on their sum being \(n_N\). We choose the representation satisfying \(\sum_i p_{N,i}=n_N\) and put
\[
d_N=\sum_{i=1}^Np_{N,i}(1-p_{N,i}).
\]
The following conditions are imposed.
\begin{enumerate}[label=\textup{(\alph*)}]
\item The size measure satisfies
\[
0<c_r\le r(X)\le C_r<\infty
\quad\text{a.s.}
\]

\item The sample size satisfies
\[
\frac{n_N}{N}\to\rho\in(0,1).
\]

\item The prescribed probabilities satisfy
\begin{equation}
0<\pi_{N,i}<1,
\qquad
1\le i\le N,
\qquad
\text{a.s. for every }N.
\label{eq-rejective-finite-N-feasibility}
\end{equation}

\item There are constants \(\varepsilon,c>0\) such that
\begin{equation}
P_0^N\left(
\varepsilon\le\min_{1\le i\le N}\pi_{N,i}
\le\max_{1\le i\le N}\pi_{N,i}\le1-\varepsilon,
\quad
d_N\ge cN
\right)\to1.
\label{eq-rejective-nondegeneracy}
\end{equation}
\end{enumerate}
Put \(\mu_r=\E[r(X)]\).

\begin{proposition}[Rejective sampling]
\label{prop-rejective-sampling}
Under the preceding rejective sampling setting and conditions, Assumptions~\ref{ass-B1}--\ref{ass-B4} hold with
\[
\pi(x)=\rho\frac{r(x)}{\mu_r}.
\]
\end{proposition}

\begin{proof}
Equation~\eqref{eq-rejective-finite-N-feasibility} makes the prescribed probabilities feasible for every \(N\). The law of large numbers and boundedness of \(r\) give
\[
\sup_i\left|\pi_{N,i}-\rho\frac{r(X_i)}{\mu_r}\right|
=o_{P_0^N}(1),
\]
which proves Assumption~\ref{ass-B2}. Equation~\eqref{eq-rejective-nondegeneracy} and Assumption~\ref{ass-B2} give Assumption~\ref{ass-B4}. The conditional-Bernoulli representation is given in \citet{Hajek1964}. The result of \citet{ChenDempsterLiu1994} establishes \eqref{eq-negative-dependence-condition}. Thus, Lemma~\ref{lem-negative-dependence-bound} gives Assumption~\ref{ass-B3}.
\end{proof}

At a fixed sampling fraction \(\rho\), the following feasible class of limiting inclusion probabilities is closed and convex.
\[
\mathfrak D_\rho^{\mathrm{Rej}}
=
\left\{
\pi:
\pi_-\le\pi(X)\le1-\varepsilon\ \text{a.s.},
\quad
\mathbb E[\pi(X)]=\rho
\right\}.
\]
For any \(\pi\) in this class, prescribe
\[
\pi_{N,i}
=
\operatorname{expit}
\{\operatorname{logit}\pi(X_i)+\lambda_N\},
\qquad
\sum_{i=1}^N\pi_{N,i}=n_N,
\]
where the second equation uniquely determines \(\lambda_N\). The common pointwise margins, the law of large numbers, and \(n_N/N\to\rho\) imply
\[
\lambda_N=o_{P_0^N}(1),
\qquad
\sup_{1\le i\le N}
|\pi_{N,i}-\pi(X_i)|
=o_{P_0^N}(1).
\]
Because \(0<\pi_{N,i}<1\) and \(\sum_{i=1}^N\pi_{N,i}=n_N\), \citet{ChenDempsterLiu1994} implies that there is a rejective law with these prescribed first-order inclusion probabilities. The displayed convergence gives Assumption~\ref{ass-B2}. Non-informativeness, negative dependence, and the pointwise lower bound give Assumptions~\ref{ass-B1}, \ref{ass-B3}, and~\ref{ass-B4}.

\subsubsection{Stratified sampling}

Let
\[
\mathcal X=\bigsqcup_{j=1}^J\mathcal X_j
\]
be a measurable partition. Given \(\bm X_N\), define
\[
S_{N,j}=\{i\le N:X_i\in\mathcal X_j\},
\qquad
N_{N,j}=|S_{N,j}|.
\]
Independently across strata, draw \(n_{N,j}\) units by SRSWOR from \(S_{N,j}\). We impose the following assumptions.
\begin{enumerate}[label=\textup{(\alph*)}]
\item The number of strata is fixed and
\[
P_0^F(X\in\mathcal X_j)>0,
\qquad j=1,\ldots,J.
\]

\item The within-stratum sampling fractions satisfy
\begin{equation}
\frac{n_{N,j}}{N_{N,j}}\to p_j\in(0,1],
\qquad j=1,\ldots,J.
\label{eq-stratified-sampling-fractions}
\end{equation}
\end{enumerate}

\begin{proposition}[Stratified sampling]
\label{prop-stratified-sampling}
Under the preceding setting and assumptions, Assumptions~\ref{ass-B1}--\ref{ass-B4} hold with
\[
\pi(x)=\sum_{j=1}^Jp_j\mathbf 1\{x\in\mathcal X_j\}.
\]
\end{proposition}

\begin{proof}
For \(i\in S_{N,j}\), \(\pi_{N,i}=n_{N,j}/N_{N,j}\). Since \(J\) is fixed, \eqref{eq-stratified-sampling-fractions} gives Assumptions~\ref{ass-B2} and~\ref{ass-B4}. Within each stratum the indicators are negatively correlated, and the designs are independent across strata. Lemma~\ref{lem-negative-dependence-bound} gives Assumption~\ref{ass-B3}.
\end{proof}

For the fixed partition above, an upper bound \(\rho\) on the limiting overall sampling fraction gives the following closed and convex feasible class.
\[
\mathfrak D_\rho^{\mathrm{Str}}
=
\left\{
\pi(x)=\sum_{j=1}^Jp_j\mathbf 1\{x\in\mathcal X_j\}:
\pi_-\le p_j\le1,
\quad
\sum_{j=1}^JP_0^F(X\in\mathcal X_j)p_j\le\rho
\right\}.
\]
Each member is attained by choosing integers \(n_{N,j}\) such that \(n_{N,j}/N_{N,j}\to p_j\).

\subsubsection{Stratified two-stage sampling}
\label{app-stratified-two-stage}

Let \(\mathcal X=\bigsqcup_{h=1}^H\mathcal X_h\) be a fixed measurable partition with \(P_0^F(X\in\mathcal X_h)>0\). Given \(\bm X_N\), write
\[
S_{N,h}=\{i\le N:X_i\in\mathcal X_h\},
\qquad
N_{N,h}=|S_{N,h}|,
\]
and let
\[
S_{N,h}=\bigsqcup_{a=1}^{M_{N,h}}G_{N,h,a}
\]
be an \(\bm X_N\)-measurable partition into PSUs. Put \(M_N=\sum_{h=1}^H M_{N,h}\) and \(N_{N,h,a}=|G_{N,h,a}|\). At the first stage, let \(\delta^{(1)}_{N,h,a}\) indicate selection of PSU \((h,a)\), with
\[
\pi^{(1)}_{N,h,a}
=
P_0^N\{\delta^{(1)}_{N,h,a}=1\mid\bm X_N\}.
\]
The first-stage designs are independent across strata, and distinct first-stage indicators within each stratum have non-positive conditional covariance given \(\bm X_N\).

For every PSU, let \(\delta^{(2)}_{N,h,a,i}\) denote the potential second-stage selection indicator of \(i\in G_{N,h,a}\), with
\[
\pi^{(2)}_{N,h,a,i}
=
P_0^N\{\delta^{(2)}_{N,h,a,i}=1\mid\bm X_N\}.
\]
Conditional on \(\bm X_N\), the second-stage randomizations are independent across PSUs and independent of the first-stage randomization. Within each PSU,
\[
\Cov\{\delta^{(2)}_{N,h,a,i},\delta^{(2)}_{N,h,a,j}\mid\bm X_N\}\le0,
\qquad i\ne j.
\]
All stage-specific randomizations are independent of \(\bm Y_N\) conditional on \(\bm X_N\). For \(i\in G_{N,h,a}\), define
\begin{equation}
\delta_i
=
\delta^{(1)}_{N,h,a}\delta^{(2)}_{N,h,a,i},
\qquad
\pi_{N,i}
=
\pi^{(1)}_{N,h,a}\pi^{(2)}_{N,h,a,i}.
\label{eq-two-stage-ultimate-inclusion}
\end{equation}
Assume
\begin{enumerate}[label=\textup{(\alph*)},beginpenalty=10000]
\item There are measurable functions \(\pi^{(1)},\pi^{(2)}:\mathcal X\to(0,1]\), determined by the two stage-specific designs, such that
\begin{align}
\sup_{h,a}\sup_{i\in G_{N,h,a}}
\left|
\pi^{(1)}_{N,h,a}-\pi^{(1)}(X_i)
\right|
&=o_{P_0^N}(1),
\notag\\
\sup_{h,a}\sup_{i\in G_{N,h,a}}
\left|
\pi^{(2)}_{N,h,a,i}-\pi^{(2)}(X_i)
\right|
&=o_{P_0^N}(1),
\label{eq-two-stage-stagewise-stability}
\end{align}
and \(\pi^{(1)}(X)\pi^{(2)}(X)\ge\pi_-\) a.s. for some \(\pi_->0\).
\item
\begin{equation}
\max_{1\le h\le H}\max_{1\le a\le M_{N,h}}
\frac{|G_{N,h,a}|}{N}
=o_{P_0^N}(1).
\label{eq-two-stage-negligible-psu}
\end{equation}
\end{enumerate}

\begin{proposition}[Stratified two-stage sampling]
\label{prop-stratified-two-stage}
Under the preceding setting and assumptions, Assumptions~\ref{ass-B1}--\ref{ass-B4} hold.
\end{proposition}

\begin{proof}
The measurability and conditional independence assumptions give Assumption~\ref{ass-B1}. Since all stage-specific probabilities lie in \((0,1]\), \eqref{eq-two-stage-ultimate-inclusion} and~\eqref{eq-two-stage-stagewise-stability} give
\begin{align*}
\sup_{h,a}\sup_{i\in G_{N,h,a}}
&\left|
\pi_{N,i}-\pi^{(1)}(X_i)\pi^{(2)}(X_i)
\right| \\
\le 
&\sup_{h,a}\sup_{i\in G_{N,h,a}}
\left|\pi^{(1)}_{N,h,a}-\pi^{(1)}(X_i)\right| \\
&\quad+
\sup_{h,a}\sup_{i\in G_{N,h,a}}
\left|\pi^{(2)}_{N,h,a,i}-\pi^{(2)}(X_i)\right|
=o_{P_0^N}(1).
\end{align*}
Thus Assumptions~\ref{ass-B2} and~\ref{ass-B4} hold with \(\pi(x)=\pi^{(1)}(x)\pi^{(2)}(x)\).

It remains to verify Assumption~\ref{ass-B3}. It is enough to take a nonnegative function \(g\) with \(\E\{g(X)^2\}<\infty\) and write \(g_i=g(X_i)\). Since
\[
\delta^{(1)}_{N,h,a}\delta^{(2)}_{N,h,a,i}
-
\pi^{(1)}_{N,h,a}\pi^{(2)}_{N,h,a,i}
=
\pi^{(2)}_{N,h,a,i}\{\delta^{(1)}_{N,h,a}-\pi^{(1)}_{N,h,a}\}
+
\delta^{(1)}_{N,h,a}\{\delta^{(2)}_{N,h,a,i}-\pi^{(2)}_{N,h,a,i}\},
\]
we have
\[
\frac1N\sum_{i=1}^N(\delta_i-\pi_{N,i})g_i=T_{1N}+T_{2N},
\]
where
\[
T_{1N}
=
\frac1N\sum_{h,a}
\{\delta^{(1)}_{N,h,a}-\pi^{(1)}_{N,h,a}\}
\sum_{i\in G_{N,h,a}}\pi^{(2)}_{N,h,a,i}g_i
\]
and
\[
T_{2N}
=
\frac1N\sum_{h,a}\delta^{(1)}_{N,h,a}
\sum_{i\in G_{N,h,a}}
\{\delta^{(2)}_{N,h,a,i}-\pi^{(2)}_{N,h,a,i}\}g_i.
\]
Because \(\E(T_{2N}\mid\bm X_N,\{\delta^{(1)}_{N,h,a}\}_{h,a})=0\), the two terms are conditionally uncorrelated. The first-stage covariance condition and Cauchy--Schwarz give
\[
\Var(T_{1N}\mid\bm X_N)
\le
\left\{
\max_{h,a}\frac{|G_{N,h,a}|}{N}
\right\}
\left\{\frac1N\sum_{i=1}^Ng_i^2\right\}
=o_{P_0^N}(1).
\]
Conditional on \((\bm X_N,\{\delta^{(1)}_{N,h,a}\}_{h,a})\), the second-stage covariance condition and independence across PSUs give
\[
\Var(T_{2N}\mid\bm X_N,\{\delta^{(1)}_{N,h,a}\}_{h,a})
\le
\frac1N\left\{\frac1N\sum_{i=1}^Ng_i^2\right\}
=o_{P_0^N}(1).
\]
Thus Assumption~\ref{ass-B3} holds.
\end{proof}

We next specialize the preceding two-stage design to self-weighted \(\pi\)PS/SRS sampling. This is the standard self-weighting two-stage PPS construction, with PSU selection proportional to size followed by a fixed-size simple random sample within each selected PSU \citep{SinghHarter2015,KaltonEtAl2021}. Here $\pi$PS denotes prescribed first-order inclusion probabilities proportional to PSU size, not PPS draws with replacement. Retain the preceding stratified two-stage setting and \eqref{eq-two-stage-negligible-psu}, but do not impose \eqref{eq-two-stage-stagewise-stability}. For each stratum \(h\), let \(m^{(1)}_{N,h}>0\) and \(m^{(2)}_{N,h}\in\mathbb N\) be \(\sigma(\bm X_N)\)-measurable and satisfy
\[
0<\frac{m^{(1)}_{N,h}N_{N,h,a}}{N_{N,h}}\le1,\qquad 1\le m^{(2)}_{N,h}\le\min_{1\le a\le M_{N,h}}N_{N,h,a},
\]
for every \(a\), with \(m^{(1)}_{N,h}\in\mathbb N\) under rejective sampling. At the first stage, sample PSUs by Poisson or rejective sampling with prescribed first-order inclusion probabilities
\begin{equation}
\pi^{(1)}_{N,h,a}=\frac{m^{(1)}_{N,h}N_{N,h,a}}{N_{N,h}}.
\label{eq-self-weighted-first-stage}
\end{equation}
At the second stage, draw \(m^{(2)}_{N,h}\) units by SRSWOR from each selected PSU. Hence, for \(i\in G_{N,h,a}\),
\begin{equation}
\pi^{(2)}_{N,h,a,i}=\frac{m^{(2)}_{N,h}}{N_{N,h,a}},\qquad \pi_{N,i}=\pi^{(1)}_{N,h,a}\pi^{(2)}_{N,h,a,i}=\rho_{N,h},\qquad \rho_{N,h}=\frac{m^{(1)}_{N,h}m^{(2)}_{N,h}}{N_{N,h}}.
\label{eq-self-weighted-ultimate-probability}
\end{equation}

Here \(h\) indexes strata and \(a\) indexes PSUs within stratum \(h\). The quantities \(N_{N,h}\), \(M_{N,h}\), and \(N_{N,h,a}\) are the stratum population size, the number of PSUs in the stratum, and the size of PSU \(G_{N,h,a}\), respectively. The quantity \(m^{(1)}_{N,h}\) is the expected number of selected PSUs under Poisson sampling and the fixed number under rejective sampling, while \(m^{(2)}_{N,h}\) is the number of units sampled from each selected PSU. The quantities \(\pi^{(1)}_{N,h,a}\), \(\pi^{(2)}_{N,h,a,i}\), and \(\rho_{N,h}\) are the first-stage, second-stage, and ultimate inclusion probabilities, respectively.

\begin{corollary}[Self-weighted \(\pi\)PS/SRS sampling]
\label{cor-self-weighted-two-stage}
Suppose that
\[
\rho_{N,h}\xrightarrow[]{p}\rho_h\in(0,1],\qquad h=1,\ldots,H.
\]
Then Assumptions~\ref{ass-B1}--\ref{ass-B4} hold with
\[
\pi(x)=\sum_{h=1}^H\rho_h\mathbf1\{x\in\mathcal X_h\}.
\]
\end{corollary}

\begin{proof}
Equation~\eqref{eq-self-weighted-ultimate-probability} gives
\[
\sup_{1\le i\le N}\left|\pi_{N,i}-\sum_{h=1}^H\rho_h\mathbf1\{X_i\in\mathcal X_h\}\right|\le\max_{1\le h\le H}|\rho_{N,h}-\rho_h|=o_{P_0^N}(1),
\]
which proves Assumption~\ref{ass-B2}, while \(\min_h\rho_h>0\) gives Assumption~\ref{ass-B4}. The arguments for Assumptions~\ref{ass-B1} and~\ref{ass-B3} in the proof of Proposition~\ref{prop-stratified-two-stage} apply unchanged because they use only the common sampling structure and \eqref{eq-two-stage-negligible-psu}.
\end{proof}

Thus exact self-weighting replaces separate stagewise convergence by convergence of \(\rho_{N,h}\), allowing the PSU sizes \(N_{N,h,a}\) to vary with \(N\) without requiring either stage-specific inclusion probability to converge. Self-weighting holds within strata. Unequal $\rho_{N,h}$ still give unequal base weights across strata. The cancellation in \eqref{eq-self-weighted-ultimate-probability} uses the same actual PSU size in both stages. Outdated or estimated size measures need not give exact self-weighting \citep{KaltonEtAl2021}. Other first-stage implementations, such as systematic PPS, require their own design-condition checks.

For optimal design, fix the second-stage sample-size sequences in the preceding construction, put
\[
N_{N,h}^{\max}=\max_{1\le a\le M_{N,h}}N_{N,h,a},\qquad \gamma_{N,h}=\frac{m^{(2)}_{N,h}}{N_{N,h}^{\max}},
\]
and suppose that
\begin{equation}
\gamma_{N,h}\xrightarrow[]{p}\gamma_h\in(0,1],\qquad h=1,\ldots,H.
\label{eq-two-stage-second-stage-capacity}
\end{equation}
Under an upper bound \(\rho\) on the limiting expected ultimate sampling fraction, suppose that
\begin{equation}
\mathfrak D_\rho^{\pi\mathrm{PS/SRS}}=\left\{\pi(x)=\sum_{h=1}^Hp_h\mathbf1\{x\in\mathcal X_h\}\,\middle|\,\pi_-\le p_h\le\gamma_h,\quad \sum_{h=1}^HP_0^F(X\in\mathcal X_h)p_h\le\rho\right\}
\label{eq-two-stage-optimal-feasible-class}
\end{equation}
is nonempty. This class is closed and convex in \(L^2(P_0^F)\), and every member is attainable with either Poisson or rejective sampling at the first stage.

Indeed, fix a member with stratum-specific values \(p_h\), and choose
\[
m^{(1)}_{N,h}=
\begin{cases}
\displaystyle \frac{N_{N,h}}{m^{(2)}_{N,h}}\min\{p_h,\gamma_{N,h}\}, & \text{under Poisson sampling},\\[6pt]
\displaystyle \max\left\{1,\left\lfloor\frac{N_{N,h}}{m^{(2)}_{N,h}}\min\{p_h,\gamma_{N,h}\}\right\rfloor\right\}, & \text{under rejective sampling}.
\end{cases}
\]
Then
\[
0<m^{(1)}_{N,h}\le\frac{N_{N,h}}{N_{N,h}^{\max}},
\]
so the first-stage inclusion probabilities in \eqref{eq-self-weighted-first-stage} are feasible. Under Poisson sampling, \(\rho_{N,h}=\min\{p_h,\gamma_{N,h}\}\). Under rejective sampling,
\[
\left|\rho_{N,h}-\min\{p_h,\gamma_{N,h}\}\right|\le\frac{m^{(2)}_{N,h}}{N_{N,h}}\le\frac{N_{N,h}^{\max}/N}{N_{N,h}/N}=o_{P_0^N}(1).
\]
Hence \(\rho_{N,h}\xrightarrow[]{p}p_h\) under either design. Under rejective sampling, the prescribed first-stage inclusion probabilities sum to the integer \(m^{(1)}_{N,h}\), and, after separating any certainty PSUs, \citet{ChenDempsterLiu1994} gives a rejective law with these first-order inclusion probabilities. Corollary~\ref{cor-self-weighted-two-stage} therefore verifies Assumptions~\ref{ass-B1}--\ref{ass-B4}. Finally,
\[
\frac1N\sum_{i=1}^N\pi_{N,i}=\sum_{h=1}^H\frac{N_{N,h}}N\rho_{N,h}\xrightarrow[]{p}\sum_{h=1}^HP_0^F(X\in\mathcal X_h)p_h\le\rho.
\]

\subsection{Details for efficient estimation}
\label{app-efficient-estimation}

This subsection proves the results in Section~\ref{sec-efficient-estimation} and gives their design-specific sufficient conditions. The folds and training sigma-fields are those of Section~\ref{sec-efficient-estimation}, including the independent randomization in Remark~\ref{rem-randomized-folds}. Additional fold and nuisance conditions are imposed only where stated. All stochastic orders are under \(P_0^N\) unless \(Q_0^N\) is displayed. Assumptions~\ref{ass-B2} and~\ref{ass-B4} imply
\begin{equation}
P_0^N\left(
\min_{1\le i\le N}\pi_{N,i}\ge\frac{\pi_-}{2}
\right)\to1.
\label{eq-app-c-uniform-positivity}
\end{equation}
On the event in \eqref{eq-app-c-uniform-positivity},
\begin{equation}
\max_{1\le i\le N}
\left|\pi_{N,i}^{-1}-\pi(X_i)^{-1}\right|
\le
\frac{2}{\pi_-^2}
\sup_{1\le i\le N}|\pi_{N,i}-\pi(X_i)|
=
o_{P_0^N}(1).
\label{eq-app-c-inverse-probability-convergence}
\end{equation}

\subsubsection{Proof of the general result}
\label{app-efficient-general-proof}

\begin{proof}[Proof of Theorem~\ref{thm-efficient-estimation}\textup{(a)}]
Condition~\ref{cond-general-score} gives
\begin{equation}
\sqrt N\hat\Psi_N(\theta_0)
=
\frac1{\sqrt N}\sum_{i=1}^N
\phi_i
+o_p(1).
\label{eq-general-score-at-truth-proof}
\end{equation}
The central sequence on the right side of \eqref{eq-general-score-at-truth-proof} is \(O_p(1)\), and hence
\begin{equation}
\hat\Psi_N(\theta_0)=O_p(N^{-1/2}).
\label{eq-general-score-at-truth-rate}
\end{equation}

We first prove \(\hat\theta_N\xrightarrow[]{p}\theta_0\). Fix \(\varepsilon>0\), and put
\[
c_\varepsilon
=
\frac12
\inf_{\theta\in U,\ \|\theta-\theta_0\|\ge\varepsilon}
\|\Psi(\theta)\|.
\]
Condition~\ref{cond-general-root} gives \(c_\varepsilon>0\). Moreover,
\begin{align}
P_0^N(\|\hat\theta_N-\theta_0\|\ge\varepsilon)
&\le
P_0^N(\hat\theta_N\notin U)
+
P_0^N\{\|\hat\Psi_N(\hat\theta_N)\|\ge c_\varepsilon\}
\notag\\
&\quad+
P_0^N\left\{
\sup_{\theta\in U}\|\hat\Psi_N(\theta)-\Psi(\theta)\|
\ge c_\varepsilon
\right\}.
\label{eq-general-consistency-probability-bound}
\end{align}
Each probability on the right side of \eqref{eq-general-consistency-probability-bound} tends to zero by Condition~\ref{cond-general-root}. Thus \(\hat\theta_N\xrightarrow[]{p}\theta_0\).

Choose a deterministic sequence \(r_N\downarrow0\) such that
\begin{equation}
P_0^N(\|\hat\theta_N-\theta_0\|\le r_N)\to1.
\label{eq-general-local-neighborhood-probability}
\end{equation}
On the intersection of \(\{\hat\theta_N\in U\}\) and the event in \eqref{eq-general-local-neighborhood-probability}, the integral form of the Taylor expansion gives
\begin{align}
\hat\Psi_N(\hat\theta_N)-\hat\Psi_N(\theta_0)
&=
\left\{
\int_0^1
\frac1N\sum_{i=1}^N
\frac{\partial}{\partial\theta^\top}
\hat\phi_{i}
\{\theta_0+t(\hat\theta_N-\theta_0)\}
\,dt
\right\}
(\hat\theta_N-\theta_0)
\notag\\
&=
-(\hat\theta_N-\theta_0)
+o_p(\|\hat\theta_N-\theta_0\|),
\label{eq-general-local-linearization-proof}
\end{align}
where the second equality follows from Condition~\ref{cond-general-jacobian}. Condition~\ref{cond-general-root} and equations~\eqref{eq-general-score-at-truth-rate} and~\eqref{eq-general-local-linearization-proof} imply
\begin{equation}
\hat\theta_N-\theta_0=O_p(N^{-1/2}).
\label{eq-general-root-n-rate}
\end{equation}
Substituting \eqref{eq-general-root-n-rate} into \eqref{eq-general-local-linearization-proof} yields
\begin{equation}
\sqrt N(\hat\theta_N-\theta_0)
=
\sqrt N\hat\Psi_N(\theta_0)+o_p(1).
\label{eq-general-estimator-linearization-proof}
\end{equation}
Equations~\eqref{eq-general-score-at-truth-proof} and~\eqref{eq-general-estimator-linearization-proof} prove the superpopulation-centered expansion \eqref{eq-best-regular-sp}. Its remainder is observed-data measurable, so the same expansion holds under \(Q_0^N\). Under Assumption~\ref{ass-A3}, subtracting \eqref{eq-A3-finite-pop-target-expansion} yields the finite-population-centered expansion \eqref{eq-best-regular-fp}. Theorems~\ref{thm-efficiency}, \ref{thm-local-regularity-equivalence}, and~\ref{thm-convolution} give the asserted local regularity and efficiency.
\end{proof}

\subsubsection{Proof of the general variance result}
\label{app-efficient-variance-proof}

The notation, variance estimators, and Condition~\ref{cond-general-variance} are given in Section~\ref{sec-efficient-general-result}.

\begin{lemma}[Laws of large numbers for observed-data scores]
\label{lem-observed-score-lln}
Assume Assumptions~\ref{ass-B1}--\ref{ass-B4}. Let
\(d\in\{L^2_0(P_0^F)\}^q\) be fixed. Then
\begin{equation}
\frac1N\sum_{i=1}^N(\mathcal Ld)(O_i)
\xrightarrow[]{p}0,
\label{eq-observed-score-lln}
\end{equation}
and
\begin{equation}
\frac1N\sum_{i=1}^N
(\mathcal Ld)(O_i)(\mathcal Ld)(O_i)^\top
\xrightarrow[]{p}
\mathbb E_{\mathbb P^\pi}[(\mathcal Ld)(O)(\mathcal Ld)(O)^\top].
\label{eq-observed-score-quadratic-lln}
\end{equation}
\end{lemma}

\begin{proof}
Since
\[
(\mathcal Ld)(O_i)
=
\E[d(L_i)\mid X_i]
+
\delta_i\{d(L_i)-\E[d(L_i)\mid X_i]\},
\]
Assumption~\ref{ass-B1} gives
\begin{align*}
&\E\left[
\left\|
\frac1N\sum_{i=1}^N
\delta_i\{d(L_i)-\E[d(L_i)\mid X_i]\}
\right\|^2
\,\middle|\,
\bm X_N,\boldsymbol\delta_N
\right]\\
&\quad=
\frac1{N^2}\sum_{i=1}^N
\delta_i
\E\left[
\|d(L_i)-\E[d(L_i)\mid X_i]\|^2
\,\middle|\,
X_i
\right].
\end{align*}
The expectation of the right side is at most
\[
\frac1N\E\|d(L)-\E[d(L)\mid X]\|^2,
\]
which converges to zero. The ordinary law of large numbers gives
\[
\frac1N\sum_{i=1}^N\E[d(L_i)\mid X_i]
\xrightarrow[]{p}\E[d(L)]=0.
\]
This proves \eqref{eq-observed-score-lln}.

Suppose first that \(d\) is bounded. Since \(\delta_i^2=\delta_i\),
\begin{align}
&(\mathcal Ld)(O_i)(\mathcal Ld)(O_i)^\top
\notag\\
&\quad=
\delta_i d(L_i)d(L_i)^\top
+
(1-\delta_i)
\E[d(L_i)\mid X_i]\E[d(L_i)\mid X_i]^\top.
\label{eq-observed-score-square-expansion}
\end{align}
Consequently,
\begin{align}
&\frac1N\sum_{i=1}^N
(\mathcal Ld)(O_i)(\mathcal Ld)(O_i)^\top
\notag\\
&\quad=
\frac1N\sum_{i=1}^N
\delta_i
\left[
 d(L_i)d(L_i)^\top
-
\E\{d(L_i)d(L_i)^\top\mid X_i\}
\right]
\notag\\
&\qquad+
\frac1N\sum_{i=1}^N
\left[
\pi_{N,i}\E\{d(L_i)d(L_i)^\top\mid X_i\}
+
(1-\pi_{N,i})
\E[d(L_i)\mid X_i]\E[d(L_i)\mid X_i]^\top
\right]
\notag\\
&\qquad+
\frac1N\sum_{i=1}^N
(\delta_i-\pi_{N,i})
\left[
\E\{d(L_i)d(L_i)^\top\mid X_i\}
-
\E[d(L_i)\mid X_i]\E[d(L_i)\mid X_i]^\top
\right].
\label{eq-observed-score-quadratic-decomposition}
\end{align}
Conditionally on \((\bm X_N,\boldsymbol\delta_N)\), the first sum on the right side of \eqref{eq-observed-score-quadratic-decomposition} is centered, and each of its entries has second moment of order \(N^{-1}\). Assumption~\ref{ass-B3}, applied entrywise, makes the third sum \(o_p(1)\). Assumption~\ref{ass-B2} and the ordinary law of large numbers give
\begin{align*}
&\frac1N\sum_{i=1}^N
\left[
\pi_{N,i}\E\{d(L_i)d(L_i)^\top\mid X_i\}
+
(1-\pi_{N,i})
\E[d(L_i)\mid X_i]\E[d(L_i)\mid X_i]^\top
\right]
\\
&\quad\xrightarrow[]{p}
\E\left[
\pi(X)d(L)d(L)^\top
+
\{1-\pi(X)\}\E[d(L)\mid X]\E[d(L)\mid X]^\top
\right],
\end{align*}
which is the limit in \eqref{eq-observed-score-quadratic-lln}.

For general \(d\in\{L^2_0(P_0^F)\}^q\), let
\[
d_M(L)=d(L)\mathbf 1\{\|d(L)\|\le M\}-\E[d(L)\mathbf 1\{\|d(L)\|\le M\}].
\]
The conditional expectation is an \(L^2\)-contraction, and hence
\begin{align*}
&\sup_N
\E\left[
\frac1N\sum_{i=1}^N
\|(\mathcal Ld)(O_i)-(\mathcal Ld_M)(O_i)\|^2
\right]
\le
2\E\|d(L)-d_M(L)\|^2
\to0.
\end{align*}
Together with
\[
\|aa^\top-bb^\top\|
\le
\|a-b\|(\|a\|+\|b\|),
\]
Cauchy--Schwarz shows that the empirical matrices and their limits for \(d_M\) converge to those for \(d\) as \(M\to\infty\). Applying the bounded quadratic argument to \(d_M\) proves \eqref{eq-observed-score-quadratic-lln}.
\end{proof}

\begin{lemma}[HT law of large numbers]
\label{lem-horvitz-thompson-lln}
Assume Assumptions~\ref{ass-B1}--\ref{ass-B4}. If \(h(L)\) is a fixed finite-dimensional vector- or matrix-valued function satisfying \(\E\|h(L)\|<\infty\), then
\begin{equation}
\frac1N\sum_{i=1}^N
\frac{\delta_i}{\pi_{N,i}}h(L_i)
\xrightarrow[]{p}
\E[h(L)].
\label{eq-horvitz-thompson-lln}
\end{equation}
\end{lemma}

\begin{proof}
Suppose first that \(h\) is bounded. Then
\begin{align}
\frac1N\sum_{i=1}^N\frac{\delta_i}{\pi_{N,i}}h(L_i)
&=
\frac1N\sum_{i=1}^N
\frac{\delta_i}{\pi_{N,i}}
\{h(L_i)-\E[h(L_i)\mid X_i]\}
\notag\\
&\quad+
\frac1N\sum_{i=1}^N\E[h(L_i)\mid X_i]
\notag\\
&\quad+
\frac1N\sum_{i=1}^N
(\delta_i-\pi_{N,i})
\frac{\E[h(L_i)\mid X_i]}{\pi_{N,i}}.
\label{eq-horvitz-thompson-lln-decomposition}
\end{align}
Conditionally on \((\bm X_N,\boldsymbol\delta_N)\), the first sum on the right side of \eqref{eq-horvitz-thompson-lln-decomposition} is centered, and its second moment is \(O_p(N^{-1})\) by \eqref{eq-app-c-uniform-positivity}. For the last sum,
\begin{align*}
&\frac1N\sum_{i=1}^N
(\delta_i-\pi_{N,i})
\frac{\E[h(L_i)\mid X_i]}{\pi_{N,i}}
=
\frac1N\sum_{i=1}^N
(\delta_i-\pi_{N,i})
\frac{\E[h(L_i)\mid X_i]}{\pi(X_i)}
+o_p(1)
=o_p(1)
\end{align*}
by \eqref{eq-app-c-inverse-probability-convergence} and Assumption~\ref{ass-B3}. The middle sum in \eqref{eq-horvitz-thompson-lln-decomposition} converges to \(\E[h(L)]\).

For a general integrable \(h\), apply the result to
\(h(L)\mathbf 1\{\|h(L)\|\le M\}\). Assumption~\ref{ass-B1} and the definition of \(\pi_{N,i}\) give
\begin{align*}
&\E\left[
\frac1N\sum_{i=1}^N
\frac{\delta_i}{\pi_{N,i}}
\|h(L_i)\|\mathbf 1\{\|h(L_i)\|>M\}
\right]
=
\E[\|h(L)\|\mathbf 1\{\|h(L)\|>M\}],
\end{align*}
which converges to zero as \(M\to\infty\). This proves \eqref{eq-horvitz-thompson-lln}.
\end{proof}

\begin{proof}[Proof of Theorem~\ref{thm-efficient-estimation}\textup{(b)}]
Write \(\phi=\mathcal L \phi^{*}\). Proposition~\ref{prop-obs-eif} gives
\(\phi^{*}\in(\mathcal H^F)^p\), so Lemma~\ref{lem-observed-score-lln} yields
\[
\frac1N\sum_{i=1}^N
\phi_i\{\phi_i\}^\top
\xrightarrow[]{p}
\Sigma_{\theta_0}.
\]
For vectors \(a\) and \(b\),
\[
\|aa^\top-bb^\top\|
\le
\|a-b\|(\|a\|+\|b\|).
\]
Condition~\ref{cond-general-variance}, Cauchy--Schwarz, and the preceding law of large numbers imply
\[
\hat\Sigma_{\theta_0,N}
-
\frac1N\sum_{i=1}^N
\phi_i\{\phi_i\}^\top
=o_p(1).
\]
This proves consistency of the superpopulation variance estimator in Theorem~\ref{thm-efficient-estimation}\textup{(b)}.

Since \(\phi^F\in(\mathcal H^F)^p\), Lemma~\ref{lem-horvitz-thompson-lln}, applied to
\(\phi^F(L)\phi^F(L)^\top\), gives
\[
\frac1N\sum_{i=1}^N
\frac{\delta_i}{\pi_{N,i}}
\phi_i^F\{\phi_i^F\}^\top
\xrightarrow[]{p}
\Sigma_F.
\]
The weighted form of the same matrix inequality, Condition~\ref{cond-general-variance}, and Cauchy--Schwarz yield
\[
\hat\Sigma_{F,N}
-
\frac1N\sum_{i=1}^N
\frac{\delta_i}{\pi_{N,i}}
\phi_i^F\{\phi_i^F\}^\top
=o_p(1).
\]
Hence \(\hat\Sigma_{F,N}\xrightarrow[]{p}\Sigma_F\). Under Assumption~\ref{ass-A3},
\[
\widetilde{\Sigma}_{\theta_N,N}
=
\hat\Sigma_{\theta_0,N}-\hat\Sigma_{F,N}
\xrightarrow[]{p}
\Sigma_{\theta_0}-\Sigma_F
=
\Sigma_{\theta_N}.
\]
\end{proof}

\subsubsection{Reduction for mean targets}
\label{app-mean-target-conditions}

We use the mean targets in Example~\ref{ex-eif-mean} and Conditions~\ref{cond-mean-regression}--\ref{cond-mean-sampling-term} in Section~\ref{sec-mean-target-estimation}.

\begin{proposition}[Verification of the general conditions for the mean]
\label{thm-mean-sufficient-C-conditions}
Assume \(\mathcal H^F=L^2_0(P_0^F)\), Assumption~\ref{ass-A2}, Assumptions~\ref{ass-B1}--\ref{ass-B4}, and Conditions~\ref{cond-mean-regression}--\ref{cond-mean-sampling-term}. Then Conditions~\ref{cond-general-root}--\ref{cond-general-score} hold for \(\hat\mu_N\).
\end{proposition}

\begin{proof}
\[
Y-\mu_0=\phi^F\in\mathcal H^F\subset L_0^2(P_0^F)
\quad\Longrightarrow\quad
\E[Y^2]<\infty.
\]
For the mean target,
\begin{equation}
\phi_i
=
m_0(X_i)-\mu_0
+
\frac{\delta_i}{\pi(X_i)}\{Y_i-m_0(X_i)\}.
\label{eq-mean-oracle-observed-eif}
\end{equation}
The function
\[
m_0(X)-\mu_0+\frac{Y-m_0(X)}{\pi(X)}
\]
belongs to \(L^2_0(P_0^F)\), and its observed-data score is \eqref{eq-mean-oracle-observed-eif}. Lemma~\ref{lem-observed-score-lln} therefore gives
\begin{equation}
\frac1N\sum_{i=1}^N\phi_i=o_p(1).
\label{eq-mean-oracle-score-average}
\end{equation}
Moreover,
\begin{align}
&\Var\left(
\frac1{\sqrt N}\sum_{i=1}^N
\delta_i
\left\{
\frac1{\pi_{N,i}}-\frac1{\pi(X_i)}
\right\}
\{Y_i-m_0(X_i)\}
\,\middle|\,
\bm X_N,\boldsymbol\delta_N
\right)
\notag\\
&\quad=
\frac1N\sum_{i=1}^N
\delta_i
\left\{
\frac1{\pi_{N,i}}-\frac1{\pi(X_i)}
\right\}^2
\Var(Y_i\mid X_i)
\notag\\
&\quad\le
\max_{1\le i\le N}
\left|
\frac1{\pi_{N,i}}-\frac1{\pi(X_i)}
\right|^2
\frac1N\sum_{i=1}^N\delta_i\Var(Y_i\mid X_i)
=o_p(1),
\label{eq-mean-inclusion-probability-remainder-variance}
\end{align}
where
\[
\E\left[
\frac1N\sum_{i=1}^N\delta_i\Var(Y_i\mid X_i)
\right]
\le
\E[\{Y-m_0(X)\}^2]
<\infty.
\]
Thus conditional Chebyshev's inequality gives
\begin{equation}
\frac1{\sqrt N}\sum_{i=1}^N
\delta_i
\left\{
\frac1{\pi_{N,i}}-\frac1{\pi(X_i)}
\right\}
\{Y_i-m_0(X_i)\}
=o_p(1).
\label{eq-mean-inclusion-probability-remainder}
\end{equation}

The estimator satisfies
\begin{align}
\hat\mu_N-\mu_0
&=
\frac1N\sum_{i=1}^N\phi_i
\notag\\
&\quad+
\frac1N\sum_{i=1}^N
\delta_i
\left\{
\frac1{\pi_{N,i}}-\frac1{\pi(X_i)}
\right\}
\{Y_i-m_0(X_i)\}
\notag\\
&\quad-
\frac1N\sum_{i=1}^N
(\delta_i-\pi_{N,i})
\frac{\hat m_{-k(i)}(X_i)-m_0(X_i)}{\pi_{N,i}}.
\label{eq-mean-first-order-reduction}
\end{align}
Equations~\eqref{eq-mean-oracle-score-average}--\eqref{eq-mean-inclusion-probability-remainder} and Condition~\ref{cond-mean-sampling-term} imply \(\hat\mu_N\xrightarrow[]{p}\mu_0\). Since
\[
\hat\Psi_N(\mu)=\hat\mu_N-\mu,
\]
Condition~\ref{cond-general-root} follows with \(\Psi(\mu)=\mu_0-\mu\), and Condition~\ref{cond-general-jacobian} follows because the derivative is exactly \(-1\).

At \(\mu_0\),
\begin{align}
\hat\phi_{i}(\mu_0)-\phi_i
&=
-\frac{\delta_i-\pi_{N,i}}{\pi_{N,i}}
\{\hat m_{-k(i)}(X_i)-m_0(X_i)\}
\notag\\
&\quad+
\delta_i
\left\{
\frac1{\pi_{N,i}}-\frac1{\pi(X_i)}
\right\}
\{Y_i-m_0(X_i)\}.
\label{eq-mean-score-decomposition}
\end{align}
Condition~\ref{cond-mean-sampling-term} and
\eqref{eq-mean-inclusion-probability-remainder}
imply that summing
\eqref{eq-mean-score-decomposition} over \(i\) and dividing by
\(\sqrt N\) yields \(o_p(1)\). This proves Condition~\ref{cond-general-score}.
\end{proof}

\begin{proposition}[Variance-estimation conditions for the mean]
\label{prop-mean-variance-conditions}
Assume \(\mathcal H^F=L^2_0(P_0^F)\), Assumption~\ref{ass-A2}, Assumptions~\ref{ass-B1}--\ref{ass-B4}, and Condition~\ref{cond-mean-regression}. Then Condition~\ref{cond-general-variance} holds with \(\hat\phi_i^F=Y_i-\hat\mu_N\). The residual-form estimator in \eqref{eq-fp-var-estimator-mean-positive} is consistent for \(\Sigma_{\theta_N}\).
\end{proposition}

\begin{proof}
The last term in \eqref{eq-mean-first-order-reduction} satisfies
\begin{align*}
&\left|
\frac1N\sum_{i=1}^N
(\delta_i-\pi_{N,i})
\frac{\hat m_{-k(i)}(X_i)-m_0(X_i)}{\pi_{N,i}}
\right|
\\
&\quad\le
\frac2{\pi_-}
\left[
\frac1N\sum_{i=1}^N
\{\hat m_{-k(i)}(X_i)-m_0(X_i)\}^2
\right]^{1/2}
=o_p(1).
\end{align*}
Equations~\eqref{eq-mean-oracle-score-average} and~\eqref{eq-mean-inclusion-probability-remainder} therefore show that Condition~\ref{cond-mean-regression} alone gives \(\hat\mu_N\xrightarrow[]{p}\mu_0\).

Equations~\eqref{eq-mean-score-decomposition} and~\eqref{eq-app-c-uniform-positivity} give
\begin{align*}
&\frac1N\sum_{i=1}^N
|\hat\phi_{i}(\hat\mu_N)-\phi_i|^2
\\
&\quad\le
\frac{12}{\pi_-^2}
\frac1N\sum_{i=1}^N
\{\hat m_{-k(i)}(X_i)-m_0(X_i)\}^2
\\
&\qquad+
3\max_{1\le i\le N}
|\pi_{N,i}^{-1}-\pi(X_i)^{-1}|^2
\frac1N\sum_{i=1}^N
\delta_i\{Y_i-m_0(X_i)\}^2
+3(\hat\mu_N-\mu_0)^2.
\end{align*}
The expectation of
\(N^{-1}\sum_i\delta_i\{Y_i-m_0(X_i)\}^2\)
 is bounded by \(\E[\{Y-m_0(X)\}^2]\). Hence the preceding display is \(o_p(1)\) by Condition~\ref{cond-mean-regression} and \eqref{eq-app-c-inverse-probability-convergence}.

Also,
\[
\frac1N\sum_{i=1}^N
\frac{\delta_i}{\pi_{N,i}}
|\hat\phi_i^F-\phi_i^F|^2
=
(\hat\mu_N-\mu_0)^2
\frac1N\sum_{i=1}^N\frac{\delta_i}{\pi_{N,i}}
=o_p(1),
\]
because
\[
\E\left[
\frac1N\sum_{i=1}^N\frac{\delta_i}{\pi_{N,i}}
\right]=1.
\]
This proves Condition~\ref{cond-general-variance}.

On the event in \eqref{eq-app-c-uniform-positivity},
\[
0\le
\frac{\delta_i}{\pi_{N,i}}
\left(\frac1{\pi_{N,i}}-1\right)
\le
\frac4{\pi_-^2}\delta_i.
\]
Condition~\ref{cond-mean-regression} and Cauchy--Schwarz therefore give
\begin{align}
\hat\Sigma_{\theta_N,N}
&=
\frac1N\sum_{i=1}^N
\frac{\delta_i}{\pi_{N,i}}
\left(\frac1{\pi_{N,i}}-1\right)
\{Y_i-m_0(X_i)\}^2
+o_p(1).
\label{eq-mean-residual-variance-oracle-reduction}
\end{align}
Moreover,
\begin{align*}
&\left|
\frac1N\sum_{i=1}^N
\frac{\delta_i}{\pi_{N,i}}
\left\{
\frac1{\pi_{N,i}}-\frac1{\pi(X_i)}
\right\}
\{Y_i-m_0(X_i)\}^2
\right|
\\
&\quad\le
\max_{1\le i\le N}
|\pi_{N,i}^{-1}-\pi(X_i)^{-1}|
\frac1N\sum_{i=1}^N
\frac{\delta_i}{\pi_{N,i}}
\{Y_i-m_0(X_i)\}^2
=o_p(1)
\end{align*}
by \eqref{eq-app-c-inverse-probability-convergence} and Lemma~\ref{lem-horvitz-thompson-lln}. A second application of Lemma~\ref{lem-horvitz-thompson-lln}, now to
\[
\left\{\frac1{\pi(X)}-1\right\}\{Y-m_0(X)\}^2,
\]
gives
\begin{align*}
&\frac1N\sum_{i=1}^N
\frac{\delta_i}{\pi_{N,i}}
\left\{\frac1{\pi(X_i)}-1\right\}
\{Y_i-m_0(X_i)\}^2
\xrightarrow[]{p}
\E\left[
\left\{\frac1{\pi(X)}-1\right\}
\Var(Y\mid X)
\right]
=
\Sigma_{\theta_N}.
\end{align*}
Together with \eqref{eq-mean-residual-variance-oracle-reduction}, this proves the consistency of \eqref{eq-fp-var-estimator-mean-positive}.
\end{proof}

\begin{proof}[Proof of Theorem~\ref{thm-mean-efficient-estimation}]
Proposition~\ref{thm-mean-sufficient-C-conditions} verifies Conditions~\ref{cond-general-root}--\ref{cond-general-score}. With the EIFs in Example~\ref{ex-eif-mean}, Theorem~\ref{thm-efficient-estimation}\textup{(a)} proves part~\textup{(a)}. Proposition~\ref{prop-mean-variance-conditions} verifies Condition~\ref{cond-general-variance} and consistency of the residual finite-population variance estimator. Together with Theorem~\ref{thm-efficient-estimation}\textup{(b)}, this proves part~\textup{(b)}.
\end{proof}

\subsubsection{Mean estimation under the sampling designs}
\label{app-mean-design-results}
\label{app-mean-sampling-term}

The sampling settings and their notation are specified in Section~\ref{app-sampling-examples}. For the mean estimator in Section~\ref{sec-mean-target-estimation}, we state the fold conditions, prove Proposition~\ref{prop-cluster-two-stage-regression-rates}, and establish the auxiliary bounds used in the proof of Theorem~\ref{thm-mean-sampling-designs}.

\medskip\noindent\textbf{Fold conditions.}
The fold conditions supplement the standing \(\sigma(\bm X_N)\)-measurability of the partition.

\begin{enumerate}[label=(F\arabic*),ref=F\arabic*,labelsep=0.5em]
\item\label{cond-fold-balance}
\textbf{Limiting fold proportions.}
The number of folds is fixed. There are constants \(\lambda_1,\ldots,\lambda_K\in(0,1)\), with \(\sum_{k=1}^K\lambda_k=1\), such that
\begin{equation}
\frac{|I_k|}{N}\xrightarrow[]{p}\lambda_k,
\qquad
k=1,\ldots,K.
\label{eq-fold-size-stability}
\end{equation}

\item\label{cond-fold-sampling-structure}
\textbf{Construction of folds under stratified sampling, cluster sampling, or stratified two-stage sampling.}
\begin{enumerate}[label=\textup{(\roman*)},ref=\theenumi(\roman*),labelsep=0.5em,labelwidth=0pt,labelindent=!]

\item For stratified sampling, there are constants \(\lambda_{kj}\in(0,1)\), with \(\sum_{k=1}^K\lambda_{kj}=1\), such that
\begin{equation}
\frac{|I_k\cap S_{N,j}|}{N_{N,j}}
\xrightarrow[]{p}\lambda_{kj},
\qquad
k=1,\ldots,K,
\quad
j=1,\ldots,J.
\label{eq-stratified-fold-balance}
\end{equation}

\item For cluster sampling, every cluster is contained in one fold, and there are constants \(\lambda_1^G,\ldots,\lambda_K^G\in(0,1)\), with \(\sum_{k=1}^K\lambda_k^G=1\), such that
\begin{equation}
\frac1{M_N}
\sum_{a=1}^{M_N}\mathbf 1\{G_{N,a}\subset I_k\}
\xrightarrow[]{p}\lambda_k^G,
\qquad
k=1,\ldots,K.
\label{eq-cluster-fold-size-stability}
\end{equation}

\item\label{cond-fold-two-stage-psu}
For stratified two-stage sampling, every PSU is contained in one fold. For each stratum \(h=1,\ldots,H\), there are constants \(\lambda^G_{kh}\in(0,1)\), with \(\sum_{k=1}^K\lambda^G_{kh}=1\), such that
\begin{equation}
\frac1{M_{N,h}}
\sum_{a=1}^{M_{N,h}}
\mathbf 1\{G_{N,h,a}\subset I_k\}
\xrightarrow[]{p}\lambda^G_{kh},
\qquad k=1,\ldots,K.
\label{eq-two-stage-fold-balance}
\end{equation}

\end{enumerate}
\end{enumerate}

Condition~\ref{cond-mean-group-regression} uses \(G=G_{N,a}\), \(a=1,\ldots,M_N\), under cluster sampling and \(G=G_{N,h,a}\), \(h=1,\ldots,H\), \(a=1,\ldots,M_{N,h}\), under stratified two-stage sampling. It controls sums of regression errors within groups that share a first-stage selection indicator. It neither implies nor is implied by Condition~\ref{cond-mean-regression} in general. Under cluster sampling, group membership is \(\bm X_N\)-measurable, so the i.i.d.\ full-data assumption and \eqref{eq-non-informative} give \(\mathbb E[Y_i\mid X_i,\{G_{N,a}\}_{a=1}^{M_N},\boldsymbol\delta_N]=\mathbb E[Y_i\mid X_i]=m_0(X_i)\). The same conditional mean therefore applies to sampled and nonsampled clusters, while Conditions~\ref{cond-mean-regression} and~\ref{cond-mean-group-regression} control its estimation.

\medskip\noindent\textbf{Proof of the sufficient-rate result.}

\begin{proof}[Proof of Proposition~\ref{prop-cluster-two-stage-regression-rates}]
Since the groups are nonempty, \(\max_G|G|\ge1\), so \eqref{eq-group-size-regression-rate} implies Condition~\ref{cond-mean-regression}. Cauchy--Schwarz gives
\[
\begin{aligned}
&\frac1N\sum_G\left[\sum_{i\in G}\{\hat m_{-k(i)}(X_i)-m_0(X_i)\}\right]^2\\
&\quad\le\max_G|G|\,\frac1N\sum_{i=1}^N\{\hat m_{-k(i)}(X_i)-m_0(X_i)\}^2=o_p(1),
\end{aligned}
\]
which proves Condition~\ref{cond-mean-group-regression}.
\end{proof}

\medskip\noindent\textbf{Auxiliary bounds.}
Since \(K\) is fixed, Condition~\ref{cond-mean-sampling-term} follows if, for every \(k=1,\ldots,K\),
\begin{equation}
\frac1{\sqrt N}\sum_{i\in I_k}
(\delta_i-\pi_{N,i})
\frac{\hat m_{-k}(X_i)-m_0(X_i)}{\pi_{N,i}}
=o_p(1).
\label{eq-condition-M2-by-fold}
\end{equation}
For \(z\in\mathbb R\), write \(z_+=\max(z,0)\) and \(z_-=\max(-z,0)\).

\begin{lemma}[Signed sums under non-positive pairwise covariances]
\label{lem-signed-negative-covariance}
Let \(\mathcal G\) be a \(\sigma\)-field. Suppose that, conditionally on \(\mathcal G\),
\begin{equation}
\Cov(\delta_i,\delta_j\mid\mathcal G)\le0,
\qquad i\ne j.
\label{eq-signed-negative-covariance-assumption}
\end{equation}
Then, for any \(\mathcal G\)-measurable real numbers \(a_i\),
\begin{equation}
\Var\left(
\sum_i\{\delta_i-\E(\delta_i\mid\mathcal G)\}a_i
\,\middle|\,
\mathcal G
\right)
\le2\sum_i a_i^2.
\label{eq-signed-negative-covariance-bound}
\end{equation}
\end{lemma}

\begin{proof}
By \(a_i=(a_i)_+-(a_i)_-\) and \(\Var(U-V)\le2\Var(U)+2\Var(V)\),
\begin{align}
&\Var\left(
\sum_i\{\delta_i-\E(\delta_i\mid\mathcal G)\}a_i
\,\middle|\,
\mathcal G
\right)
\notag\\
&\quad\le
2\Var\left(
\sum_i\{\delta_i-\E(\delta_i\mid\mathcal G)\}(a_i)_+
\,\middle|\,
\mathcal G
\right)
+
2\Var\left(
\sum_i\{\delta_i-\E(\delta_i\mid\mathcal G)\}(a_i)_-
\,\middle|\,
\mathcal G
\right).
\label{eq-signed-negative-covariance-split}
\end{align}
For either sign, \eqref{eq-signed-negative-covariance-assumption} gives
\begin{align}
&\Var\left(
\sum_i\{\delta_i-\E(\delta_i\mid\mathcal G)\}(a_i)_{\pm}
\,\middle|\,
\mathcal G
\right)
\notag\\
&\quad=
\sum_i(a_i)_{\pm}^2\Var(\delta_i\mid\mathcal G)
+2\sum_{i<j}(a_i)_{\pm}(a_j)_{\pm}
\Cov(\delta_i,\delta_j\mid\mathcal G)
\le\sum_i(a_i)_{\pm}^2.
\label{eq-signed-negative-covariance-one-sign}
\end{align}
Equations~\eqref{eq-signed-negative-covariance-split} and~\eqref{eq-signed-negative-covariance-one-sign} give \eqref{eq-signed-negative-covariance-bound}.
\end{proof}

\begin{lemma}[A criterion for sampling individual units within each fold]
\label{lem-unit-foldwise-criterion}
Assume Assumptions~\ref{ass-B2} and~\ref{ass-B4}, and Condition~\ref{cond-mean-regression}. Fix \(k\). Suppose that there is a \(\sigma\)-field \(\mathcal G_{N,k}\supset\mathcal T_{N,k}\) such that
\begin{equation}
\Cov(\delta_i,\delta_j\mid\mathcal G_{N,k})\le0,
\qquad i,j\in I_k,
\quad i\ne j,
\label{eq-unit-foldwise-negative-covariance}
\end{equation}
and
\begin{equation}
\max_{i\in I_k}
\left|
P_0^N(\delta_i=1\mid\mathcal G_{N,k})-\pi_{N,i}
\right|
=O_p(N^{-1/2}).
\label{eq-unit-foldwise-inclusion-stability}
\end{equation}
Then \eqref{eq-condition-M2-by-fold} holds.
\end{lemma}

\begin{proof}
The left side of \eqref{eq-condition-M2-by-fold} equals
\begin{align}
&\frac1{\sqrt N}\sum_{i\in I_k}
\left\{
\delta_i-P_0^N(\delta_i=1\mid\mathcal G_{N,k})
\right\}
\frac{\hat m_{-k}(X_i)-m_0(X_i)}{\pi_{N,i}}
\notag\\
&\quad+
\frac1{\sqrt N}\sum_{i\in I_k}
\left\{
P_0^N(\delta_i=1\mid\mathcal G_{N,k})-\pi_{N,i}
\right\}
\frac{\hat m_{-k}(X_i)-m_0(X_i)}{\pi_{N,i}}.
\label{eq-unit-foldwise-decomposition}
\end{align}
On the event in \eqref{eq-app-c-uniform-positivity}, Lemma~\ref{lem-signed-negative-covariance} gives
\begin{align}
&\Var\left(
\frac1{\sqrt N}\sum_{i\in I_k}
\left\{
\delta_i-P_0^N(\delta_i=1\mid\mathcal G_{N,k})
\right\}
\frac{\hat m_{-k}(X_i)-m_0(X_i)}{\pi_{N,i}}
\,\middle|\,
\mathcal G_{N,k}
\right)
\notag\\
&\quad\le
\frac8{\pi_-^2N}
\sum_{i\in I_k}
\{\hat m_{-k}(X_i)-m_0(X_i)\}^2
=o_p(1).
\label{eq-unit-foldwise-centered-variance}
\end{align}
Conditional Chebyshev's inequality and \eqref{eq-unit-foldwise-centered-variance} make the first term in \eqref{eq-unit-foldwise-decomposition} \(o_p(1)\). For the second term,
\begin{align}
&\left|
\frac1{\sqrt N}\sum_{i\in I_k}
\left\{
P_0^N(\delta_i=1\mid\mathcal G_{N,k})-\pi_{N,i}
\right\}
\frac{\hat m_{-k}(X_i)-m_0(X_i)}{\pi_{N,i}}
\right|
\notag\\
&\quad\le
\frac2{\pi_-}\sqrt N
\max_{i\in I_k}
\left|
P_0^N(\delta_i=1\mid\mathcal G_{N,k})-\pi_{N,i}
\right|
\left[
\frac1N\sum_{i\in I_k}
\{\hat m_{-k}(X_i)-m_0(X_i)\}^2
\right]^{1/2}
\notag\\
&\quad=o_p(1)
\label{eq-unit-foldwise-mean-bound}
\end{align}
by \eqref{eq-unit-foldwise-inclusion-stability} and Condition~\ref{cond-mean-regression}.
\end{proof}

For each design that samples individual units, it remains to choose \(\mathcal G_{N,k}\) and verify \eqref{eq-unit-foldwise-negative-covariance}--\eqref{eq-unit-foldwise-inclusion-stability}.

\begin{lemma}[A criterion for sampling clusters within each fold]
\label{lem-cluster-foldwise-criterion}
Assume Conditions~\ref{cond-fold-sampling-structure}\textup{(ii)} and~\ref{cond-mean-group-regression}. Fix \(k\), and assume that \(\hat m_{-k}(X_i)\) is \(\mathcal T_{N,k}\)-measurable for every \(i\in I_k\). Suppose that
\[
\min_{1\le a\le M_N}
P_0^N(A_{N,a}=1\mid\bm X_N)
\]
is bounded away from zero in probability. If there is a \(\sigma\)-field \(\mathcal G_{N,k}\supset\mathcal T_{N,k}\) such that
\begin{equation}
\Cov(A_{N,a},A_{N,b}\mid\mathcal G_{N,k})\le0
\label{eq-cluster-foldwise-negative-covariance}
\end{equation}
for distinct clusters contained in \(I_k\), and
\begin{equation}
\max_{a:G_{N,a}\subset I_k}
\left|
P_0^N(A_{N,a}=1\mid\mathcal G_{N,k})
-
P_0^N(A_{N,a}=1\mid\bm X_N)
\right|
=O_p(M_N^{-1/2}),
\label{eq-cluster-foldwise-inclusion-stability}
\end{equation}
then
\begin{align}
&\frac1{\sqrt N}
\sum_{a:G_{N,a}\subset I_k}
\frac{
A_{N,a}-P_0^N(A_{N,a}=1\mid\bm X_N)
}
{P_0^N(A_{N,a}=1\mid\bm X_N)}
\sum_{i\in G_{N,a}}
\{\hat m_{-k}(X_i)-m_0(X_i)\}
=o_p(1).
\label{eq-cluster-foldwise-conclusion}
\end{align}
\end{lemma}

\begin{proof}
Centering \(A_{N,a}\) at its conditional mean gives two terms as in \eqref{eq-unit-foldwise-decomposition}. Lemma~\ref{lem-signed-negative-covariance} and Condition~\ref{cond-mean-group-regression} give
\begin{align}
&\Var\left(
\begin{aligned}
&\frac1{\sqrt N}
\sum_{a:G_{N,a}\subset I_k}
\frac{
A_{N,a}-P_0^N(A_{N,a}=1\mid\mathcal G_{N,k})
}
{P_0^N(A_{N,a}=1\mid\bm X_N)} \\
&\qquad{}\times\sum_{i\in G_{N,a}}
\{\hat m_{-k}(X_i)-m_0(X_i)\}
\end{aligned}
\,\middle|\,
\mathcal G_{N,k}
\right)
\notag\\
&\quad\le
\frac{O_p(1)}N
\sum_{a=1}^{M_N}
\left[
\sum_{i\in G_{N,a}}
\{\hat m_{-k(i)}(X_i)-m_0(X_i)\}
\right]^2
=o_p(1).
\label{eq-cluster-foldwise-centered-variance}
\end{align}
The remaining term is bounded by
\begin{align}
&\frac{O_p(1)}{\sqrt N}
\max_{a:G_{N,a}\subset I_k}
\left|
P_0^N(A_{N,a}=1\mid\mathcal G_{N,k})
-
P_0^N(A_{N,a}=1\mid\bm X_N)
\right|
\notag\\
&\qquad\times
\sum_{a:G_{N,a}\subset I_k}
\left|
\sum_{i\in G_{N,a}}
\{\hat m_{-k}(X_i)-m_0(X_i)\}
\right|
\notag\\
&\quad\le
\frac{O_p(1)}{\sqrt{M_NN}}
\left\{
\sum_{a=1}^{M_N}\mathbf 1\{G_{N,a}\subset I_k\}
\right\}^{1/2}
\notag\\
&\qquad\times
\left[
\sum_{a=1}^{M_N}
\left\{
\sum_{i\in G_{N,a}}
\{\hat m_{-k(i)}(X_i)-m_0(X_i)\}
\right\}^2
\right]^{1/2}
=o_p(1)
\label{eq-cluster-foldwise-mean-bound}
\end{align}
by \eqref{eq-cluster-foldwise-inclusion-stability}, \eqref{eq-cluster-fold-size-stability}, and Condition~\ref{cond-mean-group-regression}.
\end{proof}

For cluster sampling, it remains to verify \eqref{eq-cluster-foldwise-negative-covariance}--\eqref{eq-cluster-foldwise-inclusion-stability}. For stratified two-stage sampling, analogous bounds are established at the PSU level within each stratum.

The next lemma bounds the change in the success probabilities when a sum of independent Bernoulli variables is conditioned to equal its mean.

\begin{lemma}[Bernoulli conditioning bound]
\label{lem-bernoulli-conditioning-bound}
Let \(B_1,\ldots,B_m\) be independent Bernoulli variables with
\[
P(B_j=1)=q_j\in(0,1),
\qquad
S=\sum_{j=1}^m B_j,
\qquad
\mu=\sum_{j=1}^m q_j,
\qquad
D=\sum_{j=1}^m q_j(1-q_j).
\]
If \(\mu\in\mathbb Z\) and \(D\ge1\), then
\begin{equation}
\max_{1\le j\le m}
\left|P(B_j=1\mid S=\mu)-q_j\right|
\le
\frac{5\pi}{9\sqrt D}.
\label{eq-bernoulli-conditioning-bound}
\end{equation}
\end{lemma}

\begin{proof}
For a sum \(T\) of independent Bernoulli variables, write
\[
p_T(\ell)=P(T=\ell),
\qquad
V=\Var(T),
\qquad
\varphi_T(t)=\E(e^{itT}).
\]
If \(V>0\), Fourier inversion gives
\begin{align}
|p_T(\ell)-p_T(\ell-1)|
&\le
\frac1{2\pi}
\int_{-\pi}^{\pi}
|1-e^{it}|\,|\varphi_T(t)|\,dt
\notag\\
&\le
\frac1{2\pi}
\int_{-\pi}^{\pi}
|t|\exp\left(-\frac{2Vt^2}{\pi^2}\right)dt
\le
\frac{\pi}{4V}.
\label{eq-bernoulli-adjacent-mass-bound}
\end{align}
Indeed,
\[
|1-q+qe^{it}|^2
=1-4q(1-q)\sin^2(t/2)
\le
\exp\{-4q(1-q)\sin^2(t/2)\},
\]
and \(\sin(|t|/2)\ge |t|/\pi\) for \(|t|\le\pi\), which gives the second line of \eqref{eq-bernoulli-adjacent-mass-bound}.

Since \(\mu\) is an integer, it is a mode of the distribution of \(S\) by \citet{Darroch1964}. Chebyshev's inequality therefore gives
\begin{align}
P(S=\mu)
&\ge
\frac{P(|S-\mu|<2\sqrt D)}
{\#\{\ell\in\mathbb Z:|\ell-\mu|<2\sqrt D\}}
\ge
\frac{3}{4(4\sqrt D+1)}
\ge
\frac{3}{20\sqrt D}.
\label{eq-bernoulli-mass-at-mean-bound}
\end{align}
Fix \(j\), put \(T_j=S-B_j\), and define
\[
a_j=P(T_j=\mu-1),
\qquad
b_j=P(T_j=\mu).
\]
Then
\begin{equation}
P(B_j=1\mid S=\mu)-q_j
=
\frac{q_j(1-q_j)(a_j-b_j)}{P(S=\mu)}.
\label{eq-bernoulli-conditioning-identity}
\end{equation}
The variance \(V_j\) of \(T_j\) satisfies
\[
V_j=D-q_j(1-q_j)\ge D-\frac14\ge\frac{3D}{4}.
\]
Combining \eqref{eq-bernoulli-adjacent-mass-bound}--\eqref{eq-bernoulli-conditioning-identity} yields
\[
\left|P(B_j=1\mid S=\mu)-q_j\right|
\le
\frac14\frac{\pi}{4V_j}\frac{20\sqrt D}{3}
\le
\frac{5\pi}{9\sqrt D}.
\]
\end{proof}

\medskip\noindent\textbf{Design-specific conditions and proof of Theorem~\ref{thm-mean-sampling-designs}.}
For the rejective first-stage case below, let \(p^{(1)}_{N,h,a}\) be the canonical Bernoulli probabilities and put
\[
d^{(1)}_{N,h}
=
\sum_{a=1}^{M_{N,h}}p^{(1)}_{N,h,a}\{1-p^{(1)}_{N,h,a}\}.
\]
The additional condition for this case is that there exist constants \(\varepsilon,c>0\) such that
\begin{equation}
P_0^N\left(
\begin{gathered}
\min_{1\le h\le H}\frac{d^{(1)}_{N,h}}{M_{N,h}}\ge c,\\
\varepsilon\le\min_{h,a}\pi^{(1)}_{N,h,a}
\le\max_{h,a}\pi^{(1)}_{N,h,a}\le1-\varepsilon,\\
\min_{1\le h\le H}\frac{M_{N,h}}{M_N}\ge c
\end{gathered}
\right)\to1.
\label{eq-two-stage-rejective-first-stage-condition}
\end{equation}

The sampling and fold conditions used in Theorem~\ref{thm-mean-sampling-designs} are listed below. Each item uses the corresponding sampling setting and assumptions in Section~\ref{app-sampling-examples}. For stratified two-stage sampling, either Proposition~\ref{prop-stratified-two-stage} or Corollary~\ref{cor-self-weighted-two-stage} may be used. Items~\textup{(a)--(e)} correspond to Theorem~\ref{thm-mean-sampling-designs}\textup{(i)}, and items~\textup{(f)--(h)} correspond to Theorem~\ref{thm-mean-sampling-designs}\textup{(ii)}.
\begin{enumerate}[label=\textup{(\alph*)},ref=\alph*,labelsep=0.5em]
\item\label{case-mean-poisson} Poisson sampling, with Condition~\ref{cond-mean-regression}.
\item\label{case-mean-srswor} SRSWOR, with Conditions~\ref{cond-fold-balance} and~\ref{cond-mean-regression}.
\item\label{case-mean-stratified} Stratified sampling, with Conditions~\ref{cond-fold-sampling-structure}\textup{(i)} and~\ref{cond-mean-regression}.
\item\label{case-mean-pps} PPS sampling with replacement, with Conditions~\ref{cond-fold-balance} and~\ref{cond-mean-regression}.
\item\label{case-mean-rejective} Rejective sampling, with Conditions~\ref{cond-fold-balance} and~\ref{cond-mean-regression}.
\item\label{case-mean-cluster} Cluster sampling, with Conditions~\ref{cond-fold-sampling-structure}\textup{(ii)}, \ref{cond-mean-regression}, and~\ref{cond-mean-group-regression}.
\item\label{case-mean-two-stage-poisson} Stratified two-stage sampling with Poisson sampling of PSUs at the first stage, with Conditions~\ref{cond-fold-two-stage-psu}, \ref{cond-mean-regression}, and~\ref{cond-mean-group-regression}.
\item\label{case-mean-two-stage-rejective} Stratified two-stage sampling with rejective sampling of PSUs at the first stage, with Conditions~\ref{cond-fold-two-stage-psu}, \ref{cond-mean-regression}, and~\ref{cond-mean-group-regression}, together with \eqref{eq-two-stage-rejective-first-stage-condition}.
\end{enumerate}

\begin{proof}[Proof of Theorem~\ref{thm-mean-sampling-designs}]
The results in Section~\ref{app-sampling-examples}, including Corollary~\ref{cor-self-weighted-two-stage} for self-weighted two-stage sampling, verify Assumptions~\ref{ass-B1}--\ref{ass-B4}. By Theorem~\ref{thm-mean-efficient-estimation}, it remains to prove Condition~\ref{cond-mean-sampling-term}. Fix \(k\). We establish \eqref{eq-condition-M2-by-fold} in each case.

\medskip\noindent\textbf{\textup{(\ref{case-mean-poisson})} Poisson sampling.}
Take \(\mathcal G_{N,k}=\mathcal T_{N,k}\). Then
\[
P_0^N(\delta_i=1\mid\mathcal T_{N,k})=\pi_{N,i},
\qquad
\Cov(\delta_i,\delta_j\mid\mathcal T_{N,k})=0,
\qquad i\ne j.
\]
Lemma~\ref{lem-unit-foldwise-criterion} proves \eqref{eq-condition-M2-by-fold}.

\medskip\noindent\textbf{\textup{(\ref{case-mean-srswor})} SRSWOR.}
Take \(\mathcal G_{N,k}=\mathcal T_{N,k}\). For \(i,j\in I_k\), \(i\ne j\),
\begin{equation}
P_0^N(\delta_i=1\mid\mathcal T_{N,k})
=
\frac{n_N-\sum_{\ell\in I_k^c}\delta_\ell}{|I_k|},
\label{eq-srs-conditional-inclusion}
\end{equation}
\begin{align}
&\Cov(\delta_i,\delta_j\mid\mathcal T_{N,k})
\notag\\*
&\quad=
\frac{\left(n_N-\sum_{\ell\in I_k^c}\delta_\ell\right)
\left(n_N-\sum_{\ell\in I_k^c}\delta_\ell-1\right)}
{|I_k|(|I_k|-1)}
-
\left(
\frac{n_N-\sum_{\ell\in I_k^c}\delta_\ell}{|I_k|}
\right)^2
\le0.
\label{eq-srs-conditional-covariance}
\end{align}
Moreover,
\begin{align}
&\E\left[
\left\{
\sum_{\ell\in I_k^c}
\left(\delta_\ell-\frac{n_N}{N}\right)
\right\}^2
\,\middle|\,
\mathcal I_N,\bm X_N
\right]
\le
\sum_{\ell\in I_k^c}
\Var(\delta_\ell\mid\mathcal I_N,\bm X_N)
\le N,
\label{eq-srs-outside-count-bound}
\end{align}
where the first inequality follows from the non-positive pairwise covariances of SRSWOR. Hence
\begin{align}
&\max_{i\in I_k}
\left|
P_0^N(\delta_i=1\mid\mathcal T_{N,k})-\frac{n_N}{N}
\right|
=
\frac1{|I_k|}
\left|
\sum_{\ell\in I_k^c}
\left(\delta_\ell-\frac{n_N}{N}\right)
\right|
=O_p(N^{-1/2})
\label{eq-srs-fold-inclusion-stability}
\end{align}
by \eqref{eq-fold-size-stability} and \eqref{eq-srs-outside-count-bound}. Lemma~\ref{lem-unit-foldwise-criterion} proves \eqref{eq-condition-M2-by-fold}.

\medskip\noindent\textbf{\textup{(\ref{case-mean-stratified})} Stratified sampling.}
Take \(\mathcal G_{N,k}=\mathcal T_{N,k}\). Within stratum \(j\), equations~\eqref{eq-srs-conditional-inclusion}--\eqref{eq-srs-fold-inclusion-stability} apply with
\[
I_k,
\quad
n_N,
\quad
N
\]
replaced by
\[
I_k\cap S_{N,j},
\quad
n_{N,j},
\quad
N_{N,j},
\]
respectively. Indicators from different strata are conditionally independent. Therefore
\begin{equation}
\Cov(\delta_i,\delta_\ell\mid\mathcal T_{N,k})\le0,
\qquad i,\ell\in I_k,
\quad i\ne\ell,
\label{eq-stratified-fold-negative-covariance}
\end{equation}
and, for every fixed \(j\),
\begin{align}
&\max_{i\in I_k\cap S_{N,j}}
\left|
P_0^N(\delta_i=1\mid\mathcal T_{N,k})
-
\frac{n_{N,j}}{N_{N,j}}
\right|
=O_p(N_{N,j}^{-1/2})
=O_p(N^{-1/2})
\label{eq-stratified-fold-inclusion-stability}
\end{align}
by \eqref{eq-stratified-fold-balance} and \(N_{N,j}/N\xrightarrow[]{p}P_0^F(X\in\mathcal X_j)>0\). Since \(J\) is fixed, Lemma~\ref{lem-unit-foldwise-criterion} proves \eqref{eq-condition-M2-by-fold}.

\medskip\noindent\textbf{\textup{(\ref{case-mean-pps})} PPS sampling with replacement.}
\smallskip\noindent\emph{Conditional inclusion probabilities and covariances.}
Put
\begin{equation}
W_{N,k}=\sum_{i\in I_k}w_{N,i},
\qquad
R_{N,k}=n_N-\sum_{j\in I_k^c}C_{N,j},
\label{eq-pps-fold-mass-and-count}
\end{equation}
and take
\begin{equation}
\mathcal G_{N,k}
=
\sigma\left(
\mathcal T_{N,k},
\{C_{N,j}\mid j\in I_k^c\}
\right).
\label{eq-pps-fold-conditioning-field}
\end{equation}
The bounds on \(r\) and \eqref{eq-fold-size-stability} give, for some constant \(c>0\),
\begin{equation}
P_0^N\{c\le W_{N,k}\le1-c\}\to1,
\qquad
\max_iw_{N,i}=O_p(N^{-1}).
\label{eq-pps-fold-weight-bounds}
\end{equation}
Since each draw falls in \(I_k\) with conditional probability \(W_{N,k}\),
\begin{equation}
\begin{aligned}
\E(R_{N,k}\mid\bm X_N,\mathcal I_N)&=n_NW_{N,k},\\
\Var(R_{N,k}\mid\bm X_N,\mathcal I_N)&=n_NW_{N,k}(1-W_{N,k})\le n_N,
\end{aligned}
\label{eq-pps-fold-count-moments}
\end{equation}
and hence
\begin{equation}
R_{N,k}-n_NW_{N,k}=O_p(\sqrt N).
\label{eq-pps-fold-count-rate}
\end{equation}
Conditionally on \(\mathcal G_{N,k}\), the remaining \(R_{N,k}\) draws are independent with probabilities \(w_{N,i}/W_{N,k}\), \(i\in I_k\). Thus
\begin{equation}
P_0^N(\delta_i=1\mid\mathcal G_{N,k})
=
1-
\left(1-\frac{w_{N,i}}{W_{N,k}}\right)^{R_{N,k}},
\label{eq-pps-fold-conditional-inclusion}
\end{equation}
and, for distinct \(i,j\in I_k\),
\begin{align}
&\Cov(\delta_i,\delta_j\mid\mathcal G_{N,k})
\notag\\
&\quad=
\left(1-\frac{w_{N,i}+w_{N,j}}{W_{N,k}}\right)^{R_{N,k}}
-
\left(1-\frac{w_{N,i}}{W_{N,k}}\right)^{R_{N,k}}
\left(1-\frac{w_{N,j}}{W_{N,k}}\right)^{R_{N,k}}
\le0.
\label{eq-pps-fold-negative-covariance}
\end{align}
\smallskip\noindent\emph{Stability of the conditional inclusion probabilities.}
For \(u,v\le0\) and \(0\le x<1\),
\begin{equation}
|e^u-e^v|
=
\left|\int_v^u e^t\,dt\right|
\le|u-v|,
\qquad
0\le-\log(1-x)-x
=
\int_0^x\frac{t}{1-t}\,dt
\le\frac{x^2}{2(1-x)}.
\label{eq-pps-elementary-log-bounds}
\end{equation}
Equations~\eqref{eq-pps-fold-weight-bounds} and~\eqref{eq-pps-elementary-log-bounds} give, uniformly over \(i\in I_k\),
\begin{align}
&\left|\log\left(1-\frac{w_{N,i}}{W_{N,k}}\right)\right|
\le
\frac{w_{N,i}/W_{N,k}}{1-w_{N,i}/W_{N,k}}
\le O_p(1)w_{N,i},
\label{eq-pps-log-first-bound}
\\
&\left|
W_{N,k}\log\left(1-\frac{w_{N,i}}{W_{N,k}}\right)
-
\log(1-w_{N,i})
\right|
\notag\\
&\quad\le
W_{N,k}
\left|
-\log\left(1-\frac{w_{N,i}}{W_{N,k}}\right)
-
\frac{w_{N,i}}{W_{N,k}}
\right|
+
\left|-\log(1-w_{N,i})-w_{N,i}\right|
\notag\\
&\quad\le
\frac{w_{N,i}^2}
{2W_{N,k}\{1-w_{N,i}/W_{N,k}\}}
+
\frac{w_{N,i}^2}{2(1-w_{N,i})}
\le O_p(1)w_{N,i}^2.
\label{eq-pps-log-second-bound}
\end{align}
Since \(\pi_{N,i}=1-(1-w_{N,i})^{n_N}\), equations~\eqref{eq-pps-fold-conditional-inclusion} and~\eqref{eq-pps-elementary-log-bounds} give
\begin{align}
&\left|
P_0^N(\delta_i=1\mid\mathcal G_{N,k})-\pi_{N,i}
\right|
\notag\\
&\quad=
\left|
\exp\left\{
R_{N,k}\log\left(1-\frac{w_{N,i}}{W_{N,k}}\right)
\right\}
-
\exp\{n_N\log(1-w_{N,i})\}
\right|
\notag\\
&\quad\le
\left|
R_{N,k}\log\left(1-\frac{w_{N,i}}{W_{N,k}}\right)
-
n_N\log(1-w_{N,i})
\right|
\notag\\
&\quad\le
|R_{N,k}-n_NW_{N,k}|
\left|\log\left(1-\frac{w_{N,i}}{W_{N,k}}\right)\right|
\notag\\
&\qquad+
n_N
\left|
W_{N,k}\log\left(1-\frac{w_{N,i}}{W_{N,k}}\right)
-
\log(1-w_{N,i})
\right|
\notag\\
&\quad\le
O_p(1)
\left[
|R_{N,k}-n_NW_{N,k}|w_{N,i}
+n_Nw_{N,i}^2
\right].
\label{eq-pps-inclusion-difference-bound}
\end{align}
Hence
\begin{align}
&\max_{i\in I_k}
\left|
P_0^N(\delta_i=1\mid\mathcal G_{N,k})-\pi_{N,i}
\right|
=O_p(N^{-1/2})
\label{eq-pps-fold-inclusion-stability}
\end{align}
by \eqref{eq-pps-fold-count-rate}, \eqref{eq-pps-fold-weight-bounds}, and~\eqref{eq-pps-inclusion-difference-bound}. Lemma~\ref{lem-unit-foldwise-criterion} proves \eqref{eq-condition-M2-by-fold}.

\medskip\noindent\textbf{\textup{(\ref{case-mean-rejective})} Rejective sampling.}
Take \(\mathcal G_{N,k}=\mathcal T_{N,k}\). Under the rejective sampling setting and Condition~\ref{cond-fold-balance}, we establish the two conditions of Lemma~\ref{lem-unit-foldwise-criterion}:
\begin{align}
\Cov(\delta_i,\delta_j\mid\mathcal T_{N,k})&\le0,
\qquad i,j\in I_k,
\quad i\ne j,
\label{eq-rejective-fold-negative-covariance}
\\
\max_{i\in I_k}
\left|
P_0^N(\delta_i=1\mid\mathcal T_{N,k})-\pi_{N,i}
\right|
&=O_p(N^{-1/2}).
\label{eq-rejective-foldwise-stability}
\end{align}

\smallskip\noindent\emph{Canonical probabilities and the fold count.}
Let \(p_{N,1},\ldots,p_{N,N}\) be the canonical Bernoulli probabilities in the rejective representation, chosen so that
\[
\sum_{i=1}^N p_{N,i}=n_N,
\qquad
d_N=\sum_{i=1}^N p_{N,i}(1-p_{N,i}).
\]
Conditionally on \(\bm X_N\), Lemma~\ref{lem-bernoulli-conditioning-bound} applied to the full Bernoulli array gives
\begin{equation}
\max_{1\le i\le N}|\pi_{N,i}-p_{N,i}|
=O_p(d_N^{-1/2})
=O_p(N^{-1/2}).
\label{eq-rejective-canonical-first-order-approximation}
\end{equation}
Together with \eqref{eq-rejective-nondegeneracy}, this implies that, for some \(c_0\in(0,1/2)\),
\begin{equation}
P_0^N\left(
 c_0\le\min_i p_{N,i}
 \le\max_i p_{N,i}\le1-c_0
\right)\to1.
\label{eq-rejective-canonical-uniform-bounds}
\end{equation}

Put
\[
S_{N,k}=\sum_{i\in I_k}\delta_i
=n_N-\sum_{i\in I_k^c}\delta_i,
\qquad
b_{N,k}=\sum_{i\in I_k}p_{N,i}.
\]
By \(\sigma(\bm X_N)\)-measurability of the folds and \eqref{eq-negative-dependence-condition},
\begin{equation}
\E\left[
\left\{
S_{N,k}-\sum_{i\in I_k}\pi_{N,i}
\right\}^2
\,\middle|\,
\mathcal I_N,\bm X_N
\right]
\le |I_k|.
\label{eq-rejective-fold-count-variance}
\end{equation}
Moreover, \eqref{eq-rejective-canonical-first-order-approximation} and Condition~\ref{cond-fold-balance} give
\begin{equation}
\left|b_{N,k}-\sum_{i\in I_k}\pi_{N,i}\right|
\le
|I_k|\max_i|p_{N,i}-\pi_{N,i}|
=O_p(\sqrt N).
\label{eq-rejective-fold-count-centering}
\end{equation}
Hence
\begin{equation}
S_{N,k}-b_{N,k}=O_p(\sqrt N).
\label{eq-rejective-fold-count-rate}
\end{equation}

\smallskip\noindent\emph{Tilting to the conditional sample size.}
For \(t\in\mathbb R\), define
\begin{equation}
p_{N,i}(t)
=
\frac{e^tp_{N,i}}{1-p_{N,i}+e^tp_{N,i}},
\qquad
F_{N,k}(t)=\sum_{i\in I_k}p_{N,i}(t).
\label{eq-rejective-odds-tilt}
\end{equation}
On the event in \eqref{eq-rejective-canonical-uniform-bounds}, there is a constant \(c_1>0\) such that
\begin{equation}
\inf_{|t|\le1}F_{N,k}'(t)
=
\inf_{|t|\le1}
\sum_{i\in I_k}p_{N,i}(t)\{1-p_{N,i}(t)\}
\ge c_1|I_k|.
\label{eq-rejective-tilt-derivative-bound}
\end{equation}
Consequently,
\begin{equation}
F_{N,k}(1)-F_{N,k}(0)\ge c_1|I_k|,
\qquad
F_{N,k}(0)-F_{N,k}(-1)\ge c_1|I_k|.
\label{eq-rejective-tilt-bracketing}
\end{equation}
Equations~\eqref{eq-rejective-fold-count-rate} and~\eqref{eq-rejective-tilt-bracketing}, together with Condition~\ref{cond-fold-balance}, imply
\begin{equation}
P_0^N\{F_{N,k}(-1)<S_{N,k}<F_{N,k}(1)\}\to1.
\label{eq-rejective-tilt-existence-event}
\end{equation}
Since \(F_{N,k}\) is strictly increasing, on the event in \eqref{eq-rejective-tilt-existence-event} there is a unique \(t_{N,k}\in(-1,1)\) satisfying
\[
F_{N,k}(t_{N,k})=S_{N,k}.
\]
Set \(t_{N,k}=0\) off this event. The mean value theorem and \eqref{eq-rejective-tilt-derivative-bound} give
\begin{equation}
|t_{N,k}|
\le
\frac{|S_{N,k}-F_{N,k}(0)|}{c_1|I_k|}
=O_p(N^{-1/2}).
\label{eq-rejective-tilt-rate}
\end{equation}
Consequently,
\begin{equation}
d_{N,k}
:=
\sum_{i\in I_k}p_{N,i}(t_{N,k})
\{1-p_{N,i}(t_{N,k})\},
\qquad
d_{N,k}^{-1}=O_p(N^{-1}).
\label{eq-rejective-fold-variance-order}
\end{equation}
The variables \(S_{N,k}\), \(t_{N,k}\), and
\((p_{N,i}(t_{N,k}))_{i\in I_k}\) are \(\mathcal T_{N,k}\)-measurable.

\smallskip\noindent\emph{Conditional inclusion probabilities.}
For \(d=(d_i)_{i\in I_k}\in\{0,1\}^{I_k}\), Assumption~\ref{ass-B1} and \(\sigma(\bm X_N)\)-measurability of the folds give
\begin{align}
P_0^N\{(\delta_i)_{i\in I_k}=d\mid\mathcal T_{N,k}\}
&\propto
\prod_{i\in I_k}
\left(\frac{p_{N,i}}{1-p_{N,i}}\right)^{d_i}
\mathbf 1\left\{\sum_{i\in I_k}d_i=S_{N,k}\right\}
\notag\\
&\propto
\prod_{i\in I_k}
\left(\frac{p_{N,i}(t_{N,k})}{1-p_{N,i}(t_{N,k})}\right)^{d_i}
\mathbf 1\left\{\sum_{i\in I_k}d_i=S_{N,k}\right\}.
\label{eq-rejective-fold-conditional-law}
\end{align}
If \(S_{N,k}\in\{0,|I_k|\}\), \eqref{eq-rejective-fold-negative-covariance} is immediate. Otherwise, \citet{ChenDempsterLiu1994} gives \eqref{eq-rejective-fold-negative-covariance}.

On the event in \eqref{eq-rejective-tilt-existence-event}, the second proportionality in \eqref{eq-rejective-fold-conditional-law} identifies the conditional law as rejective sampling on \(I_k\) with canonical probabilities \(p_{N,i}(t_{N,k})\). On this event, conditionally on \(\mathcal T_{N,k}\), Lemma~\ref{lem-bernoulli-conditioning-bound} and \eqref{eq-rejective-fold-variance-order} give
\begin{equation}
\max_{i\in I_k}
\left|
P_0^N(\delta_i=1\mid\mathcal T_{N,k})
-p_{N,i}(t_{N,k})
\right|
=O_p(N^{-1/2}).
\label{eq-rejective-conditional-first-order-approximation}
\end{equation}
Since
\[
\partial_t p_{N,i}(t)
=p_{N,i}(t)\{1-p_{N,i}(t)\}
\le\frac14,
\]
\eqref{eq-rejective-tilt-rate} gives
\begin{equation}
\max_{i\in I_k}|p_{N,i}(t_{N,k})-p_{N,i}|
\le\frac{|t_{N,k}|}{4}
=O_p(N^{-1/2}).
\label{eq-rejective-tilted-canonical-difference}
\end{equation}
The triangle inequality applied to
\eqref{eq-rejective-canonical-first-order-approximation},
\eqref{eq-rejective-conditional-first-order-approximation}, and
\eqref{eq-rejective-tilted-canonical-difference} proves
\eqref{eq-rejective-foldwise-stability}.
Equations~\eqref{eq-rejective-fold-negative-covariance} and~\eqref{eq-rejective-foldwise-stability} therefore allow us to apply Lemma~\ref{lem-unit-foldwise-criterion}, which proves \eqref{eq-condition-M2-by-fold}.

\medskip\noindent\textbf{\textup{(\ref{case-mean-cluster})} Cluster sampling.}
Take \(\mathcal G_{N,k}=\mathcal T_{N,k}\). For every cluster \(G_{N,a}\subset I_k\),
\begin{equation}
P_0^N(A_{N,a}=1\mid\mathcal T_{N,k})
=
\frac{m_N-\sum_{b:G_{N,b}\subset I_k^c}A_{N,b}}
{\sum_{b=1}^{M_N}\mathbf 1\{G_{N,b}\subset I_k\}}.
\label{eq-cluster-fold-conditional-inclusion}
\end{equation}
Equation~\eqref{eq-srs-conditional-covariance}, with units replaced by clusters, gives
\begin{equation}
\Cov(A_{N,a},A_{N,b}\mid\mathcal T_{N,k})\le0
\label{eq-cluster-fold-negative-covariance}
\end{equation}
for distinct clusters contained in \(I_k\). Moreover,
\begin{align}
&\E\left[
\left\{
\sum_{b:G_{N,b}\subset I_k^c}
\left(A_{N,b}-\frac{m_N}{M_N}\right)
\right\}^2
\,\middle|\,
\mathcal I_N,\bm X_N
\right]
\le M_N,
\label{eq-cluster-outside-count-bound}
\end{align}
by the non-positive pairwise covariances of cluster SRSWOR. Hence
\begin{align}
&\max_{a:G_{N,a}\subset I_k}
\left|
P_0^N(A_{N,a}=1\mid\mathcal T_{N,k})-\frac{m_N}{M_N}
\right|
\notag\\
&\quad=
\frac{
\left|
\sum_{b:G_{N,b}\subset I_k^c}
\left(A_{N,b}-m_N/M_N\right)
\right|
}
{\sum_{b=1}^{M_N}\mathbf 1\{G_{N,b}\subset I_k\}}
=O_p(M_N^{-1/2})
\label{eq-cluster-fold-inclusion-stability}
\end{align}
by \eqref{eq-cluster-fold-size-stability} and \eqref{eq-cluster-outside-count-bound}. Lemma~\ref{lem-cluster-foldwise-criterion} proves \eqref{eq-condition-M2-by-fold}.

\medskip\noindent\textbf{\textup{(\ref{case-mean-two-stage-poisson})--(\ref{case-mean-two-stage-rejective})} Stratified two-stage sampling.}
We use the setting and conditions of either Proposition~\ref{prop-stratified-two-stage} or Corollary~\ref{cor-self-weighted-two-stage}, together with Conditions~\ref{cond-mean-regression}, \ref{cond-fold-two-stage-psu}, and~\ref{cond-mean-group-regression}. The first-stage design is Poisson sampling in case~\textup{(\ref{case-mean-two-stage-poisson})} and rejective sampling satisfying \eqref{eq-two-stage-rejective-first-stage-condition} in case~\textup{(\ref{case-mean-two-stage-rejective})}, independently across strata.

Put \(e_i=\hat m_{-k(i)}(X_i)-m_0(X_i)\), and let \(J_{N,kh}=\{a:G_{N,h,a}\subset I_k\}\). To prove \eqref{eq-condition-M2-by-fold}, decompose its left side into the first-stage and second-stage remainders
\begin{align}
\frac1{\sqrt N}\sum_{i\in I_k}\frac{\delta_i-\pi_{N,i}}{\pi_{N,i}}e_i
&=R^{(1)}_{N,k}+R^{(2)}_{N,k},
\label{eq-two-stage-foldwise-remainder-decomposition}\\
R^{(1)}_{N,k}
&=\frac1{\sqrt N}\sum_{h=1}^H\sum_{a\in J_{N,kh}}\frac{\delta^{(1)}_{N,h,a}-\pi^{(1)}_{N,h,a}}{\pi^{(1)}_{N,h,a}}\sum_{i\in G_{N,h,a}}e_i,
\notag\\
R^{(2)}_{N,k}
&=\frac1{\sqrt N}\sum_{h=1}^H\sum_{a\in J_{N,kh}}\frac{\delta^{(1)}_{N,h,a}}{\pi^{(1)}_{N,h,a}}\sum_{i\in G_{N,h,a}}\frac{\delta^{(2)}_{N,h,a,i}-\pi^{(2)}_{N,h,a,i}}{\pi^{(2)}_{N,h,a,i}}e_i.
\notag
\end{align}
We show that both remainders are \(o_p(1)\). The identity follows from \eqref{eq-two-stage-ultimate-inclusion}. In particular, the second-stage inclusion probabilities cancel from \(R^{(1)}_{N,k}\), whose coefficients involve the unweighted PSU sums \(\sum_{i\in G_{N,h,a}}e_i\).

Both Proposition~\ref{prop-stratified-two-stage} and Corollary~\ref{cor-self-weighted-two-stage} imply Assumptions~\ref{ass-B2} and~\ref{ass-B4}. Hence, for some constant \(c_0>0\), with probability tending to one,
\[
\min\{\pi^{(1)}_{N,h,a},\pi^{(2)}_{N,h,a,i}\}\ge\pi^{(1)}_{N,h,a}\pi^{(2)}_{N,h,a,i}=\pi_{N,i}\ge c_0
\]
for all \(h,a\) and \(i\in G_{N,h,a}\). All bounds below are on this event.

\smallskip\noindent\emph{First-stage conditional inclusion probabilities.}
For either design, take
\[
\mathcal G_{N,k}=\sigma\!\left(\mathcal T_{N,k},\{\delta^{(1)}_{N,h,a}:G_{N,h,a}\subset I_k^c\}_{h,a}\right).
\]
The errors \(e_i\), \(i\in I_k\), are \(\mathcal G_{N,k}\)-measurable. Given \(\bm X_N\) and the first-stage indicators outside \(I_k\), the training data contain only additional information from outcomes and second-stage randomizations outside \(I_k\). By the sampling assumptions and \(\sigma(\bm X_N)\)-measurability of the folds, this information does not change the conditional law of the first-stage indicators inside \(I_k\). The first-stage vectors in different strata therefore remain conditionally independent given \(\mathcal G_{N,k}\).

For Poisson first-stage sampling, the first-stage indicators inside \(I_k\) are conditionally independent, with conditional inclusion probabilities \(\pi^{(1)}_{N,h,a}\).

For rejective first-stage sampling, the fixed first-stage sample size in each stratum gives
\[
S^{(1)}_{N,kh}=\sum_{a\in J_{N,kh}}\delta^{(1)}_{N,h,a}=\sum_{a=1}^{M_{N,h}}p^{(1)}_{N,h,a}-\sum_{a\notin J_{N,kh}}\delta^{(1)}_{N,h,a},
\]
so \(S^{(1)}_{N,kh}\) is \(\mathcal G_{N,k}\)-measurable. Conditional on \(\mathcal G_{N,k}\), the indicators in \(J_{N,kh}\) have the rejective law in \eqref{eq-rejective-fold-conditional-law}, with units replaced by PSUs and the count fixed at \(S^{(1)}_{N,kh}\). Their pairwise conditional covariances are therefore non-positive.

To bound the change in their inclusion probabilities, define \(p^{(1)}_{N,h,a}(t)\) as in \eqref{eq-rejective-odds-tilt}. On the PSU-level analogue of the event in \eqref{eq-rejective-tilt-existence-event}, choose \(t_{N,kh}\) so that
\[
\sum_{a\in J_{N,kh}}p^{(1)}_{N,h,a}(t_{N,kh})=S^{(1)}_{N,kh},
\]
and set \(t_{N,kh}=0\) otherwise. Under \eqref{eq-two-stage-rejective-first-stage-condition} and \eqref{eq-two-stage-fold-balance}, the PSU-level versions of \eqref{eq-rejective-canonical-first-order-approximation}, \eqref{eq-rejective-tilt-rate}, \eqref{eq-rejective-conditional-first-order-approximation}, and \eqref{eq-rejective-tilted-canonical-difference} give
\begin{align}
\max_a|\pi^{(1)}_{N,h,a}-p^{(1)}_{N,h,a}|&=O_p(M_{N,h}^{-1/2}),\notag\\
|t_{N,kh}|&=O_p(M_{N,h}^{-1/2}),\notag\\
\max_{a\in J_{N,kh}}\left|P_0^N\{\delta^{(1)}_{N,h,a}=1\mid\mathcal G_{N,k}\}-p^{(1)}_{N,h,a}(t_{N,kh})\right|&=O_p(M_{N,h}^{-1/2}),\notag\\
\max_{a\in J_{N,kh}}|p^{(1)}_{N,h,a}(t_{N,kh})-p^{(1)}_{N,h,a}|&=O_p(M_{N,h}^{-1/2}).
\label{eq-two-stage-rejective-ref-chain}
\end{align}
Since \(H\) is fixed and \(\min_hM_{N,h}/M_N\ge c\) with probability tending to one, the triangle inequality yields
\begin{equation}
\max_{1\le h\le H}\max_{a\in J_{N,kh}}\left|P_0^N\{\delta^{(1)}_{N,h,a}=1\mid\mathcal G_{N,k}\}-\pi^{(1)}_{N,h,a}\right|=O_p(M_N^{-1/2}).
\label{eq-two-stage-foldwise-first-stage-stability}
\end{equation}
Thus, under either design,
\begin{equation}
\operatorname{Cov}\{\delta^{(1)}_{N,h,a},\delta^{(1)}_{N,h,b}\mid\mathcal G_{N,k}\}\le0,\qquad a,b\in J_{N,kh},\quad a\ne b,
\label{eq-two-stage-foldwise-first-stage-covariance}
\end{equation}
and the corresponding cross-stratum conditional covariances are zero. For Poisson first-stage sampling, the left side of \eqref{eq-two-stage-foldwise-first-stage-stability} is zero.

\smallskip\noindent\emph{First-stage remainder.}
By \eqref{eq-two-stage-foldwise-first-stage-covariance} and Lemma~\ref{lem-signed-negative-covariance},
\begin{align}
\operatorname{Var}(R^{(1)}_{N,k}\mid\mathcal G_{N,k})
&\le\frac2N\sum_{h=1}^H\sum_{a\in J_{N,kh}}\frac{(\sum_{i\in G_{N,h,a}}e_i)^2}{(\pi^{(1)}_{N,h,a})^2}\notag\\
&\le\frac2{c_0^2N}\sum_{h=1}^H\sum_{a\in J_{N,kh}}\left(\sum_{i\in G_{N,h,a}}e_i\right)^2=o_p(1)
\label{eq-two-stage-foldwise-first-stage-variance}
\end{align}
by Condition~\ref{cond-mean-group-regression}. To control its conditional mean, use \eqref{eq-two-stage-foldwise-first-stage-stability}, Cauchy--Schwarz, and \(\sum_h|J_{N,kh}|\le M_N\) to obtain
\begin{align}
&\left|\mathbb E[R^{(1)}_{N,k}\mid\mathcal G_{N,k}]\right|\notag\\
&\quad=\left|\frac1{\sqrt N}\sum_{h=1}^H\sum_{a\in J_{N,kh}}\frac{P_0^N\{\delta^{(1)}_{N,h,a}=1\mid\mathcal G_{N,k}\}-\pi^{(1)}_{N,h,a}}{\pi^{(1)}_{N,h,a}}\sum_{i\in G_{N,h,a}}e_i\right|\notag\\
&\quad\le\frac{O_p(M_N^{-1/2})}{c_0\sqrt N}\sum_{h=1}^H\sum_{a\in J_{N,kh}}\left|\sum_{i\in G_{N,h,a}}e_i\right|\notag\\
&\quad\le\frac{O_p(M_N^{-1/2})\sqrt{M_N}}{c_0}\left\{\frac1N\sum_{h=1}^H\sum_{a\in J_{N,kh}}\left(\sum_{i\in G_{N,h,a}}e_i\right)^2\right\}^{1/2}=o_p(1).
\label{eq-two-stage-foldwise-first-stage-mean}
\end{align}
Equations~\eqref{eq-two-stage-foldwise-first-stage-variance}--\eqref{eq-two-stage-foldwise-first-stage-mean} and conditional Chebyshev's inequality give \(R^{(1)}_{N,k}=o_p(1)\).

\smallskip\noindent\emph{Second-stage remainder.}
To hold the first-stage indicators fixed, condition on
\[
\mathcal H_{N,k}=\sigma\!\left(\mathcal G_{N,k},\{\delta^{(1)}_{N,h,a}\}_{h,a}\right).
\]
Because folds contain whole PSUs, the second-stage randomizations inside \(I_k\) retain their original conditional laws given \(\bm X_N\) and remain independent across PSUs. Their conditional means are \(\pi^{(2)}_{N,h,a,i}\), and their pairwise conditional covariances within each PSU are non-positive. Hence \(\mathbb E[R^{(2)}_{N,k}\mid\mathcal H_{N,k}]=0\), and Lemma~\ref{lem-signed-negative-covariance} gives
\begin{align*}
&\operatorname{Var}(R^{(2)}_{N,k}\mid\mathcal H_{N,k})\\
&\quad=\frac1N\sum_{h=1}^H\sum_{a\in J_{N,kh}}\frac{\delta^{(1)}_{N,h,a}}{(\pi^{(1)}_{N,h,a})^2}\operatorname{Var}\!\left(\sum_{i\in G_{N,h,a}}\frac{\delta^{(2)}_{N,h,a,i}-\pi^{(2)}_{N,h,a,i}}{\pi^{(2)}_{N,h,a,i}}e_i\,\middle|\,\mathcal H_{N,k}\right)\\
&\quad\le\frac2N\sum_{h=1}^H\sum_{a\in J_{N,kh}}\frac{\delta^{(1)}_{N,h,a}}{(\pi^{(1)}_{N,h,a})^2}\sum_{i\in G_{N,h,a}}\frac{e_i^2}{(\pi^{(2)}_{N,h,a,i})^2}\\
&\quad\le\frac2{c_0^2N}\sum_{i\in I_k}e_i^2=o_p(1),
\end{align*}
where the last inequality uses \(\pi^{(1)}_{N,h,a}\pi^{(2)}_{N,h,a,i}=\pi_{N,i}\ge c_0\), and the last equality follows from Condition~\ref{cond-mean-regression}. Conditional Chebyshev's inequality gives \(R^{(2)}_{N,k}=o_p(1)\). Equation~\eqref{eq-two-stage-foldwise-remainder-decomposition} now proves \eqref{eq-condition-M2-by-fold}.

Since \(K\) is fixed, the conclusions for the individual folds give Condition~\ref{cond-mean-sampling-term} in every case. The remaining conclusions follow from Theorem~\ref{thm-mean-efficient-estimation}.
\end{proof}

\subsubsection{Reduction for regression targets}
\label{app-regression-target-conditions}

We use the regression targets in Example~\ref{ex-eif-regression}, the estimating function \(\Psi_R\) in Condition~\ref{cond-regression-root}, and Conditions~\ref{cond-regression-regularity}--\ref{cond-regression-root} in Section~\ref{sec-regression-target-estimation}.

Conditional on \(\mathcal T_{N,k}\) and \(\boldsymbol\delta_N\), validation-fold residuals remain independent and centered. Consequently, Conditions~\ref{cond-regression-regularity}--\ref{cond-regression-variance-estimation} control the nuisance remainder by Lemma~\ref{lem-regression-crossfit-remainders}, without Conditions~\ref{cond-mean-sampling-term} or~\ref{cond-mean-group-regression}. No design-specific remainder verification analogous to the proof of Theorem~\ref{thm-mean-sampling-designs} is needed. Under Conditions~\ref{cond-regression-regularity}--\ref{cond-regression-root} and the standing fold measurability, the design-specific results therefore require only Assumptions~\ref{ass-B1}--\ref{ass-B4}, verified in Section~\ref{app-sampling-examples}.

\begin{proof}[Verification of Example~\ref{ex-eif-regression}]
For \(g\in\mathcal H^F\), put \(\dot\beta_g=\dot\psi^F(g)\). The tangent-space assumption in Example~\ref{ex-eif-regression} gives
\begin{equation}
\mathbb E[\varepsilon g(L)\mid X]
=
\dot\mu_0(X)^\top\dot\beta_g.
\label{eq-regression-tangent-identity}
\end{equation}
Moreover,
\[
\mathbb E\!\left[
\varepsilon
I_F^{-1}\frac{\dot\mu_0(X)}{\sigma^2(X)}\varepsilon
\,\middle|\,X
\right]
=
I_F^{-1}\dot\mu_0(X),
\]
so the full-data candidate in \eqref{eq-regression-eifs} belongs to \((\mathcal H^F)^p\). Equation~\eqref{eq-regression-tangent-identity} yields
\[
\mathbb E\!\left[
I_F^{-1}\frac{\dot\mu_0(X)}{\sigma^2(X)}\varepsilon g(L)
\right]
=
\dot\beta_g,
\]
which proves the first identity in \eqref{eq-regression-eifs}. The same tangent-space argument applies to
\[
I_\pi^{-1}\frac{\dot\mu_0(X)}{\sigma^2(X)}\varepsilon.
\]
Its conditional expectation given \(X\) is zero. Applying \(\mathcal L\) therefore gives the observed-data candidate in \eqref{eq-regression-eifs}. Finally,
\begin{align*}
&\mathbb E_{\mathbb P^\pi}\!\left[
I_\pi^{-1}\delta\frac{\dot\mu_0(X)}{\sigma^2(X)}\varepsilon
(\mathcal Lg)(O)
\right]
=
I_\pi^{-1}\mathbb E\!\left[
\pi(X)\frac{\dot\mu_0(X)}{\sigma^2(X)}
\mathbb E[\varepsilon g(L)\mid X]
\right]
=
\dot\beta_g
\end{align*}
by \eqref{eq-regression-tangent-identity}. This proves the second identity in \eqref{eq-regression-eifs}. Taking covariance matrices and applying the covariance identity in Theorem~\ref{thm-convolution} gives the bounds in Example~\ref{ex-eif-regression}.
\end{proof}

\begin{lemma}[Expansion of the regression finite-population target]
\label{lem-regression-finite-target-expansion}
Under Condition~\ref{cond-regression-regularity}, the target in \eqref{eq-regression-finite-population-target} satisfies Assumption~\ref{ass-A3} with the full-data EIF in \eqref{eq-regression-eifs}.
\end{lemma}

\begin{proof}
All convergences in this proof are under \(Q_0^N\).
Put
\[
S_N(\beta)
=
\frac1N\sum_{i=1}^N
\frac{\dot\mu_\beta(X_i)}{\sigma^2(X_i)}
\{Y_i-\mu(X_i,\beta)\}.
\]
At \(\beta_0\),
\[
\sqrt N S_N(\beta_0)
=
\frac1{\sqrt N}\sum_{i=1}^N
\frac{\dot\mu_0(X_i)}{\sigma^2(X_i)}\varepsilon_i
=O_{Q_0^N}(1).
\]
Condition~\ref{cond-regression-regularity} and the law of large numbers give, uniformly on a shrinking neighborhood of \(\beta_0\),
\begin{equation}
-\frac{\partial S_N(\beta)}{\partial\beta^\top}
\xrightarrow[]{p}
I_F.
\label{eq-regression-finite-target-jacobian}
\end{equation}
Indeed, at \(\beta_0\), the summand on the left is
\[
\frac{
\dot\mu_0(X_i)\dot\mu_0(X_i)^\top
-\ddot\mu_0(X_i)\varepsilon_i}
{\sigma^2(X_i)},
\]
and the second term has mean zero. Taylor expansion of \(S_N(\beta_N)=0\), consistency of \(\beta_N\), and \eqref{eq-regression-finite-target-jacobian} yield
\[
\sqrt N(\beta_N-\beta_0)
=
I_F^{-1}
\frac1{\sqrt N}\sum_{i=1}^N
\frac{\dot\mu_0(X_i)}{\sigma^2(X_i)}\varepsilon_i
+o_{Q_0^N}(1),
\]
which is Assumption~\ref{ass-A3} by \eqref{eq-regression-eifs}.
\end{proof}

\begin{lemma}[Uniform convergence of the regression estimating equation]
\label{lem-regression-score-ulln}
Suppose Assumptions~\ref{ass-B1}--\ref{ass-B4} and Conditions~\ref{cond-regression-regularity}--\ref{cond-regression-variance-estimation} hold. Then
\begin{equation}
\sup_{\beta\in U}\|\hat U_N(\beta)-\Psi_R(\beta)\|
=o_{P_0^N}(1).
\label{eq-regression-score-ulln}
\end{equation}
\end{lemma}

\begin{proof}
For the oracle conditional-mean term, put
\begin{equation}
q_\beta(X)
=
\frac{\dot\mu_\beta(X)}{\sigma^2(X)}
\{\mu(X,\beta_0)-\mu(X,\beta)\}.
\label{eq-regression-oracle-conditional-mean-term}
\end{equation}
Condition~\ref{cond-regression-regularity} and the mean value theorem give, for some constant \(C_0<\infty\),
\begin{equation}
\sup_{\beta\in U}\|q_\beta(X)\|
\le C_0b(X)^2,
\qquad
\|q_\beta(X)-q_{\beta'}(X)\|
\le C_0b(X)^2\|\beta-\beta'\|.
\label{eq-regression-oracle-term-envelope}
\end{equation}
By \(Y_i-\mu(X_i,\beta)=\varepsilon_i+\mu(X_i,\beta_0)-\mu(X_i,\beta)\),
\begin{align}
\hat U_N(\beta)-\Psi_R(\beta)
&=
\frac1N\sum_{i=1}^N
\delta_i
\frac{\dot\mu_\beta(X_i)}{\hat\sigma^2_{-k(i)}(X_i)}
\varepsilon_i
\notag\\
&\quad+
\frac1N\sum_{i=1}^N
\delta_i\dot\mu_\beta(X_i)
\{\mu(X_i,\beta_0)-\mu(X_i,\beta)\}
\left\{
\frac1{\hat\sigma^2_{-k(i)}(X_i)}-
\frac1{\sigma^2(X_i)}
\right\}
\notag\\
&\quad+
\frac1N\sum_{i=1}^N\delta_iq_\beta(X_i)
-\mathbb E\{\pi(X)q_\beta(X)\}.
\label{eq-regression-ulln-decomposition}
\end{align}

Fix \(k\) and \(\beta\in U\). On the event in Condition~\ref{cond-regression-variance-estimation}, conditional on \(\mathcal T_{N,k}\) and \(\boldsymbol\delta_N\),
\begin{align}
&\mathbb E\!\left[
\left\|
\frac1N\sum_{i\in I_k}
\delta_i
\frac{\dot\mu_\beta(X_i)}{\hat\sigma^2_{-k}(X_i)}
\varepsilon_i
\right\|^2
\,\middle|\,
\mathcal T_{N,k},\boldsymbol\delta_N
\right]
\le
\frac{\sigma_+^2}{c^2N^2}\sum_{i\in I_k}b(X_i)^2
=O_p(N^{-1}).
\label{eq-regression-ulln-residual-fixed}
\end{align}
For every \(\eta>0\),
\begin{align}
&\sup_{\beta,\beta'\in U,\ \|\beta-\beta'\|\le\eta}
\left\|
\frac1N\sum_{i=1}^N
\delta_i
\frac{\dot\mu_\beta(X_i)-\dot\mu_{\beta'}(X_i)}
{\hat\sigma^2_{-k(i)}(X_i)}
\varepsilon_i
\right\|
\le
\frac{\eta}{cN}\sum_{i=1}^Nb(X_i)|\varepsilon_i|
=O_p(\eta).
\label{eq-regression-ulln-residual-modulus}
\end{align}
Since \(U\) is compact and \(K\) is fixed, \eqref{eq-regression-ulln-residual-fixed}--\eqref{eq-regression-ulln-residual-modulus} and a finite \(\eta\)-net yield
\begin{equation}
\sup_{\beta\in U}
\left\|
\frac1N\sum_{i=1}^N
\delta_i
\frac{\dot\mu_\beta(X_i)}{\hat\sigma^2_{-k(i)}(X_i)}
\varepsilon_i
\right\|
=o_p(1).
\label{eq-regression-ulln-residual-uniform}
\end{equation}

Moreover, on the same event,
\begin{align}
&\sup_{\beta\in U}
\left\|
\frac1N\sum_{i=1}^N
\delta_i\dot\mu_\beta(X_i)
\{\mu(X_i,\beta_0)-\mu(X_i,\beta)\}
\left\{
\frac1{\hat\sigma^2_{-k(i)}(X_i)}-
\frac1{\sigma^2(X_i)}
\right\}
\right\|
\notag\\
&\quad\le
\frac{2}{c\sigma_-^2}
\left[
\frac1N\sum_{i=1}^N
b(X_i)^2
\{\hat\sigma^2_{-k(i)}(X_i)-\sigma^2(X_i)\}^2
\right]^{1/2}
\left\{\frac1N\sum_{i=1}^Nb(X_i)^2\right\}^{1/2}
=o_p(1).
\label{eq-regression-ulln-weight-uniform}
\end{align}

For every fixed \(\beta\in U\), Assumption~\ref{ass-B3}, applied coordinatewise using \eqref{eq-regression-oracle-term-envelope}, and the ordinary law of large numbers give
\begin{equation}
\frac1N\sum_{i=1}^N(\delta_i-\pi_{N,i})q_\beta(X_i)=o_p(1),
\qquad
\frac1N\sum_{i=1}^N\pi(X_i)q_\beta(X_i)
-\mathbb E\{\pi(X)q_\beta(X)\}=o_p(1).
\label{eq-regression-ulln-oracle-fixed}
\end{equation}
Equation~\eqref{eq-regression-oracle-term-envelope} also gives
\begin{align}
&\sup_{\beta,\beta'\in U,\ \|\beta-\beta'\|\le\eta}
\left\|
\frac1N\sum_{i=1}^N
(\delta_i-\pi_{N,i})
\{q_\beta(X_i)-q_{\beta'}(X_i)\}
\right\|
\le C_0\eta\frac1N\sum_{i=1}^Nb(X_i)^2
=O_p(\eta),
\label{eq-regression-ulln-design-modulus}
\\
&\sup_{\beta,\beta'\in U,\ \|\beta-\beta'\|\le\eta}
\left\|
\frac1N\sum_{i=1}^N\pi(X_i)
\{q_\beta(X_i)-q_{\beta'}(X_i)\}
-\mathbb E[\pi(X)\{q_\beta(X)-q_{\beta'}(X)\}]
\right\|
\notag\\
&\quad\le
C_0\eta
\left\{
\frac1N\sum_{i=1}^Nb(X_i)^2+\mathbb E[b(X)^2]
\right\}
=O_p(\eta),
\label{eq-regression-ulln-iid-modulus}
\end{align}
and Assumption~\ref{ass-B2} implies
\begin{equation}
\sup_{\beta\in U}
\left\|
\frac1N\sum_{i=1}^N
\{\pi_{N,i}-\pi(X_i)\}q_\beta(X_i)
\right\|
\le
C_0\sup_{1\le i\le N}|\pi_{N,i}-\pi(X_i)|
\frac1N\sum_{i=1}^Nb(X_i)^2
=o_p(1).
\label{eq-regression-ulln-inclusion-probability}
\end{equation}
A finite \(\eta\)-net, \eqref{eq-regression-ulln-oracle-fixed}--\eqref{eq-regression-ulln-inclusion-probability}, and \(\eta\downarrow0\) give
\begin{equation}
\sup_{\beta\in U}
\left\|
\frac1N\sum_{i=1}^N\delta_iq_\beta(X_i)
-\mathbb E\{\pi(X)q_\beta(X)\}
\right\|
=o_p(1).
\label{eq-regression-ulln-oracle-uniform}
\end{equation}
Equations~\eqref{eq-regression-ulln-decomposition}, \eqref{eq-regression-ulln-residual-uniform}, \eqref{eq-regression-ulln-weight-uniform}, and \eqref{eq-regression-ulln-oracle-uniform} prove \eqref{eq-regression-score-ulln}.
\end{proof}

\begin{lemma}[Cross-fitted conditional-mean remainders]
\label{lem-regression-crossfit-remainders}
Assume Assumption~\ref{ass-B1} and Conditions~\ref{cond-regression-regularity}--\ref{cond-regression-variance-estimation}. Then
\begin{equation}
\frac1{\sqrt N}\sum_{i=1}^N
\delta_i
\left\{
\frac1{\hat\sigma^2_{-k(i)}(X_i)}
-
\frac1{\sigma^2(X_i)}
\right\}
\dot\mu_0(X_i)\varepsilon_i
=o_{P_0^N}(1),
\label{eq-regression-crossfit-weight-remainder}
\end{equation}
and
\begin{equation}
\frac1N\sum_{i=1}^N
\delta_i
\frac{\ddot\mu_0(X_i)}{\hat\sigma^2_{-k(i)}(X_i)}
\varepsilon_i
=o_{P_0^N}(1).
\label{eq-regression-crossfit-hessian-remainder}
\end{equation}
\end{lemma}

\begin{proof}
On the event in Condition~\ref{cond-regression-variance-estimation}, Condition~\ref{cond-regression-regularity} gives
\begin{equation}
\sigma^2(X_i)
\left\{
\frac1{\hat\sigma^2_{-k(i)}(X_i)}
-
\frac1{\sigma^2(X_i)}
\right\}^2
\le
\frac{
\{\hat\sigma^2_{-k(i)}(X_i)-\sigma^2(X_i)\}^2}
{c^2\sigma_-^2}.
\label{eq-regression-inverse-variance-bound}
\end{equation}
Since \(K\) is fixed, consider one fold \(I_k\). Conditional on \(\mathcal T_{N,k}\) and \(\boldsymbol\delta_N\), the variables \(\{\varepsilon_i\mid i\in I_k\}\) are independent and centered. Hence
\begin{align}
&\mathbb E\!\left[
\left\|
\frac1{\sqrt N}\sum_{i\in I_k}
\delta_i
\left\{
\frac1{\hat\sigma^2_{-k}(X_i)}
-
\frac1{\sigma^2(X_i)}
\right\}
\dot\mu_0(X_i)\varepsilon_i
\right\|^2
\,\middle|\,
\mathcal T_{N,k},\boldsymbol\delta_N
\right]
\notag\\
&\quad\le
\frac1{c^2\sigma_-^2N}\sum_{i\in I_k}
 b(X_i)^2
\{\hat\sigma^2_{-k}(X_i)-\sigma^2(X_i)\}^2
=o_p(1)
\label{eq-regression-crossfit-conditional-second-moment}
\end{align}
by \eqref{eq-regression-variance-L2}. Conditional Chebyshev's inequality and summation over the folds prove \eqref{eq-regression-crossfit-weight-remainder}.

On the same event, the conditional second moment of the left side of \eqref{eq-regression-crossfit-hessian-remainder}, restricted to \(I_k\), is at most
\[
\frac{\sigma_+^2}{c^2N^2}\sum_{i\in I_k}\delta_i b(X_i)^2.
\]
Its expectation is \(O(N^{-1})\), so conditional Chebyshev's inequality proves \eqref{eq-regression-crossfit-hessian-remainder}.
\end{proof}

\begin{proposition}[Verification of the general conditions for regression]
\label{prop-regression-sufficient-C-conditions}
Assume Assumptions~\ref{ass-B1}--\ref{ass-B4} and Conditions~\ref{cond-regression-regularity}--\ref{cond-regression-root}. Then Conditions~\ref{cond-general-root}--\ref{cond-general-score} hold for the contributions in \eqref{eq-regression-estimated-eif}.
\end{proposition}

\begin{proof}
For \(\epsilon>0\), put
\[
c_\epsilon
=
\frac12
\inf_{\beta\in U,\ \|\beta-\beta_0\|\ge\epsilon}
\|\Psi_R(\beta)\|.
\]
Condition~\ref{cond-regression-root} gives \(c_\epsilon>0\), and
\begin{align}
P_0^N(\|\hat\beta_N-\beta_0\|\ge\epsilon)
&\le
P_0^N(\hat\beta_N\notin U)
+
P_0^N\{\|\hat U_N(\hat\beta_N)\|\ge c_\epsilon\}
\notag\\
&\quad+
P_0^N\!\left\{
\sup_{\beta\in U}
\|\hat U_N(\beta)-\Psi_R(\beta)\|
\ge c_\epsilon
\right\}
\to0
\label{eq-regression-consistency-bound}
\end{align}
by Condition~\ref{cond-regression-root}, and Lemma~\ref{lem-regression-score-ulln}. Hence
\begin{equation}
\hat\beta_N\xrightarrow[]{p}\beta_0.
\label{eq-regression-estimator-consistency}
\end{equation}

Define the oracle empirical information matrix
\[
I_{\pi,N}^0
=
\frac1N\sum_{i=1}^N
\delta_i
\frac{\dot\mu_0(X_i)\dot\mu_0(X_i)^\top}{\sigma^2(X_i)}.
\]
Assumptions~\ref{ass-B2}--\ref{ass-B3} give
\begin{equation}
I_{\pi,N}^0\xrightarrow[]{p}I_\pi.
\label{eq-regression-oracle-information-consistency}
\end{equation}
On the event in Condition~\ref{cond-regression-variance-estimation},
\begin{align}
\|\hat I_{\pi,N}-I_{\pi,N}^0\|
&\le
\frac1N\sum_{i=1}^N
b(X_i)^2
\left|
\frac1{\hat\sigma^2_{-k(i)}(X_i)}
-
\frac1{\sigma^2(X_i)}
\right|
\notag\\
&\quad+
\frac1{cN}\sum_{i=1}^N
\left\|
\dot\mu_{\hat\beta_N}(X_i)
\dot\mu_{\hat\beta_N}(X_i)^\top
-
\dot\mu_0(X_i)\dot\mu_0(X_i)^\top
\right\|.
\label{eq-regression-information-difference-bound}
\end{align}
The first term is \(o_p(1)\) by Cauchy--Schwarz, \eqref{eq-regression-variance-L2}, the lower bounds in Conditions~\ref{cond-regression-regularity}--\ref{cond-regression-variance-estimation}, and the law of large numbers. The second is \(o_p(1)\) by \eqref{eq-regression-estimator-consistency}, continuity, and the envelope in Condition~\ref{cond-regression-regularity}. Thus \eqref{eq-regression-oracle-information-consistency}--\eqref{eq-regression-information-difference-bound} give
\begin{equation}
\hat I_{\pi,N}\xrightarrow[]{p}I_\pi,
\qquad
\hat I_{\pi,N}^{-1}\xrightarrow[]{p}I_\pi^{-1}.
\label{eq-regression-observed-information-consistency-proof}
\end{equation}
Since
\[
\hat\Psi_N(\beta)=\hat I_{\pi,N}^{-1}\hat U_N(\beta),
\qquad
\Psi(\beta)=I_\pi^{-1}\Psi_R(\beta),
\]
Condition~\ref{cond-regression-root} and \eqref{eq-regression-score-ulln}, together with \eqref{eq-regression-observed-information-consistency-proof}, imply Condition~\ref{cond-general-root}.

For \(\beta\in U\), differentiation gives
\begin{align}
\frac1N\sum_{i=1}^N
\frac{\partial}{\partial\beta^\top}
\hat\phi_{i}(\beta)
&=
\hat I_{\pi,N}^{-1}
\frac1N\sum_{i=1}^N
\delta_i
\frac{\ddot\mu_\beta(X_i)}{\hat\sigma^2_{-k(i)}(X_i)}
\{Y_i-\mu(X_i,\beta)\}
\notag\\
&\quad-
\hat I_{\pi,N}^{-1}
\frac1N\sum_{i=1}^N
\delta_i
\frac{\dot\mu_\beta(X_i)\dot\mu_\beta(X_i)^\top}
{\hat\sigma^2_{-k(i)}(X_i)}.
\label{eq-regression-estimated-eif-gradient}
\end{align}
For every \(r_N\downarrow0\), \eqref{eq-regression-crossfit-hessian-remainder}, \eqref{eq-regression-observed-information-consistency-proof}, continuity, and the envelope in Condition~\ref{cond-regression-regularity} give
\begin{align}
&\sup_{\beta\in U,\ \|\beta-\beta_0\|\le r_N}
\left\|
\frac1N\sum_{i=1}^N
\delta_i
\frac{\ddot\mu_\beta(X_i)}{\hat\sigma^2_{-k(i)}(X_i)}
\{Y_i-\mu(X_i,\beta)\}
\right\|
=o_p(1),
\notag\\
&\sup_{\beta\in U,\ \|\beta-\beta_0\|\le r_N}
\left\|
\frac1N\sum_{i=1}^N
\delta_i
\frac{\dot\mu_\beta(X_i)\dot\mu_\beta(X_i)^\top}
{\hat\sigma^2_{-k(i)}(X_i)}
-I_\pi
\right\|
=o_p(1).
\label{eq-regression-local-gradient-components}
\end{align}
Equations~\eqref{eq-regression-estimated-eif-gradient}, \eqref{eq-regression-observed-information-consistency-proof}, and~\eqref{eq-regression-local-gradient-components} prove Condition~\ref{cond-general-jacobian}.

By \eqref{eq-regression-eifs},
\[
\phi_i
=
I_\pi^{-1}\delta_i
\frac{\dot\mu_0(X_i)}{\sigma^2(X_i)}\varepsilon_i.
\]
Therefore
\begin{align}
&\frac1{\sqrt N}\sum_{i=1}^N
\{\hat\phi_{i}(\beta_0)-\phi_i\}
\notag\\
&\quad=
(\hat I_{\pi,N}^{-1}-I_\pi^{-1})
\frac1{\sqrt N}\sum_{i=1}^N
\delta_i\frac{\dot\mu_0(X_i)}{\sigma^2(X_i)}\varepsilon_i
\notag\\
&\qquad+
\hat I_{\pi,N}^{-1}
\frac1{\sqrt N}\sum_{i=1}^N
\delta_i
\left\{
\frac1{\hat\sigma^2_{-k(i)}(X_i)}
-
\frac1{\sigma^2(X_i)}
\right\}
\dot\mu_0(X_i)\varepsilon_i.
\label{eq-regression-C3-decomposition}
\end{align}
The first centered sum is \(O_p(1)\). Equations~\eqref{eq-regression-observed-information-consistency-proof} and~\eqref{eq-regression-crossfit-weight-remainder} make both terms in \eqref{eq-regression-C3-decomposition} \(o_p(1)\), proving Condition~\ref{cond-general-score}.
\end{proof}

\begin{proposition}[Variance estimation for regression]
\label{prop-regression-variance-consistency}
Under the assumptions of Proposition~\ref{prop-regression-sufficient-C-conditions},
\[
\hat I_{\pi,N}\xrightarrow[]{p}I_\pi,
\qquad
\hat I_{F,N}\xrightarrow[]{p}I_F.
\]
Consequently, Theorem~\ref{thm-regression-efficient-estimation}\textup{(b)} holds.
\end{proposition}

\begin{proof}
The first convergence is \eqref{eq-regression-observed-information-consistency-proof}. For the second, \eqref{eq-regression-estimator-consistency}, \eqref{eq-regression-variance-L2}, and the ordinary law of large numbers give
\[
\hat I_{F,N}
=
\frac1N\sum_{i=1}^N
\frac{\dot\mu_0(X_i)\dot\mu_0(X_i)^\top}{\sigma^2(X_i)}
+o_p(1)
\xrightarrow[]{p}I_F.
\]
Continuity of matrix inversion proves Theorem~\ref{thm-regression-efficient-estimation}\textup{(b)}.
\end{proof}

\begin{proof}[Proof of Theorem~\ref{thm-regression-efficient-estimation}]
Lemma~\ref{lem-regression-finite-target-expansion} gives Assumption~\ref{ass-A3}, and Proposition~\ref{prop-regression-sufficient-C-conditions} verifies Conditions~\ref{cond-general-root}--\ref{cond-general-score}. Theorem~\ref{thm-efficient-estimation}\textup{(a)}, with the EIFs in \eqref{eq-regression-eifs}, proves part~\textup{(a)}. Proposition~\ref{prop-regression-variance-consistency} proves part~\textup{(b)}.
\end{proof}

\begin{theorem}[Efficiency under the sampling designs]
\label{thm-regression-sampling-designs}
Consider one of the sampling designs listed in Theorem~\ref{thm-mean-sampling-designs}, under its corresponding setting in Section~\ref{app-sampling-examples}. Assume Assumptions~\ref{ass-A1}--\ref{ass-A2} and Conditions~\ref{cond-regression-regularity}--\ref{cond-regression-root}. Then \eqref{eq-regression-score-ulln} and the conclusions of Theorem~\ref{thm-regression-efficient-estimation} hold, and \(\hat\beta_N\) is locally regular and efficient for the superpopulation and finite-population regression targets.
\end{theorem}

\begin{proof}[Proof of Theorem~\ref{thm-regression-sampling-designs}]
The results in Section~\ref{app-sampling-examples}, including Corollary~\ref{cor-self-weighted-two-stage} for self-weighted two-stage sampling, verify Assumptions~\ref{ass-B1}--\ref{ass-B4} for each design listed in Theorem~\ref{thm-mean-sampling-designs}. Lemma~\ref{lem-regression-score-ulln} gives \eqref{eq-regression-score-ulln}, and Theorem~\ref{thm-regression-efficient-estimation} applies.
\end{proof}

\section{Additional numerical and empirical details}
\label{app-additional-numerical-details}

\subsection{Simulation experiments}
\label{app-details-sections-five-six}

In all experiments, PSU partitions and fold assignments use covariates and auxiliary randomization only and are formed before sampling. Every repetition has its own pre-assigned random-number stream, so that the results do not depend on the number of parallel workers, and no repetition is discarded. For the mean, the reference variances on the $\sqrt N$ scale are
\[
\begin{aligned}
\Sigma_{\mathrm{bound},\theta_0,b}
&=\frac1N\sum_i\{m_0(X_i)-\mu_0\}^2+
  \frac1N\sum_i\frac{\sigma^2(X_i)}{\pi_{N,i}},\\
\Sigma_{\mathrm{bound},\theta_N,b}
&=\frac1N\sum_i\left(\frac1{\pi_{N,i}}-1\right)\sigma^2(X_i).
\end{aligned}
\]
Here $\sigma^2$ is the generating conditional variance and $\pi_{N,i}$ is the exact inclusion probability, and standard errors use \eqref{eq-sp-var-estimator}--\eqref{eq-fp-var-estimator-mean-positive}. For regression, the reference variances are the relevant diagonal entries of the empirical versions of $I_\pi^{-1}$ and $I_\pi^{-1}-I_F^{-1}$ in Example~\ref{ex-eif-regression}, with expectations replaced by full-population averages and $\pi(X_i)$ by $\pi_{N,i}$, and standard errors use the variance estimators of Section~\ref{sec-regression-target-estimation}.

\subsubsection{Design-invariance and group-size experiments}
\label{app-fixed-strata-implementation}

\emph{Design invariance.} The experiment uses 4,000 repetitions, and the same generated population is used for both PSU constructions in a repetition. Within a stratum of size $N_{N,h}$, the number of PSUs $M_{N,h}$ is the integer nearest $N_{N,h}/25$, and PSU sizes differ by at most one. Random PSUs partition a random permutation of the stratum units, and covariate-ordered PSUs partition the units sorted by $U_1$. Whole PSUs are assigned to five folds at random. The exact inclusion probabilities are
\[
\pi_{N,i}=\begin{cases}
\rho_h,&\text{Poisson and two-stage sampling},\\[2pt]
\operatorname{round}(\rho_hN_{N,h})/N_{N,h},&\text{stratified SRSWOR},\\[2pt]
\operatorname{round}(\rho_hM_{N,h})/M_{N,h},&\text{one-stage cluster SRSWOR},
\end{cases}\qquad H_i=h.
\]
In the two-stage design, PSU $a$ of stratum $h$, of size $N_{N,h,a}$, is selected independently with probability $\rho_hN_{N,h,a}/10$, and ten of its units are then drawn by SRSWOR. The outcome regression is inverse-inclusion-probability weighted least squares on an intercept, four stratum indicators, $U_1$, $X_2$, and $X_2^2$, fitted outside each validation fold with a small numerical ridge.

\emph{Group size.} The population model is unchanged. For a nominal PSU size $g$, stratum $h$ is divided into $M_{N,h}=\max\{5,\operatorname{round}(N_{N,h}/g)\}$ covariate-ordered PSUs, of which $\max\{1,\operatorname{round}(\rho_hM_{N,h})\}$ are selected under cluster sampling. Each $g\in\{5,10,25,50\}$ uses 2,000 repetitions at $N=10{,}000$, and each $g\in\{10,25,50,100\}$ uses 1,000 repetitions at $N=40{,}000$. The spline learner replaces $U_1$ in the regression above by 40 piecewise-linear basis functions with knots $\Phi^{-1}\{(j-0.5)/40\}$, $j=1,\ldots,40$, and constant extrapolation beyond the extreme knots; its coefficients carry a second-difference penalty with unit weight. This fixed dictionary serves to vary the regression error and is not claimed to be mean-square consistent. With the errors $e_i=\hat m_{-k(i)}(X_i)-m_0(X_i)$ held fixed, the design variance of $N^{-1/2}\sum_i(\delta_i-\pi_{N,i})e_i/\pi_{N,i}$ under SRSWOR of PSUs within strata is
\[
Q_N=\frac1N\sum_h\frac{1-f_h}{f_h}\,\frac{M_{N,h}}{M_{N,h}-1}\sum_{a=1}^{M_{N,h}}(S_{h,a}-\bar S_h)^2,
\qquad
S_{h,a}=\sum_{i\in G_{N,h,a}}e_i,
\]
where $f_h$ is the PSU sampling fraction and $\bar S_h$ is the mean of the $S_{h,a}$ in stratum $h$; under stratified sampling the groups are the individual units. The dashed curve in Figure~\ref{fig-group-size} is $\{1+\bar Q/\bar\Sigma_{\mathrm{bound},\theta_N}\}^{1/2}$, with bars denoting averages over repetitions. It ignores the dependence between the fitted errors and the validation-fold indicators and is an approximation rather than an exact variance decomposition.

\emph{Regression.} The regression counterpart uses the same covariates, designs, PSU constructions, and folds, with 2,000 repetitions. Outcomes are conditionally Gaussian with mean $Z^\top\beta_0$, where $Z=(1,U_1,X_2)^\top$ and $\beta_0=(0.5,1,-0.75)^\top$, and variance $\sigma^2(X)=\exp\{0.65\min(\max(U_1,-2),2)\}$. The feasible GLS estimator uses the cross-fitted Gamma variance learner of Section~\ref{app-api-regression-implementation} with dictionary $(1,\min(\max(U_1,-2),2))$. The reported coefficient is that of $U_1$, and the finite-population target is the solution of \eqref{eq-regression-finite-population-target} with the true variance. Across the eight design configurations and both targets, the SD ratio is 0.983--1.046 for feasible GLS, with coverage 0.939--0.956, and 1.305--1.464 for unweighted OLS.

\subsubsection{Efficiency-gain map}
\label{app-efficiency-gain-implementation}

In the mean experiment, $X_1$ and $X_2$ are independent $\operatorname{Unif}(-1,1)$ variables and
\[
f(X)=\frac{X_1+0.75(X_2^2-1/3)}{\{1/3+0.75^2(4/45)\}^{1/2}},
\]
so that $\mathbb E\{f(X)\}=0$ and $\Var\{f(X)\}=1$. Five balanced folds are assigned by independent randomization before sampling, and the feasible learner is inverse-inclusion-probability weighted least squares on an intercept, $X_1$, $X_2$, and $X_2^2$. In the regression experiment, $\beta_0=(0.5,1)^\top$. Within each training fold, squared residuals of a preliminary unweighted OLS fit are averaged separately in the two known variance regions, and the two estimated levels, with a lower floor of $10^{-4}$, are assigned to the validation units. The intervals in Figure~\ref{fig-efficiency-gain-map} use 500 paired bootstrap resamples of the Monte Carlo repetitions.

\subsection{California API data}
\label{app-real-data-implementation}\label{app-real-data-supplement}

\subsubsection{Implementation and variance estimation for the mean}
\label{app-api-mean-variance}

The records in \texttt{apipop} and \texttt{apistrat} are matched by the school identifier \texttt{cds}. Continuous covariates are median-imputed and standardized using the population. GREG is a full-sample survey-weighted linear fit. Cross-fitted linear augmentation uses the same dictionary outside each fold. The spline learner adds four-degree-of-freedom cubic B-splines for each continuous covariate, retaining the intercept, linear terms, and school-type indicators without penalization. Only spline coefficients are ridge-penalized, with the penalty selected by generalized cross-validation within each training fold. All basis construction uses covariates only.

Let $s_h$ be the sampled units in stratum $h$, $f_h=n_h/N_h$, and $s_{r,h}^2$ the usual sample variance of $r_i=Y_i-\hat m_i$ over $s_h$. Here $\hat m_i=0$ for HT, an in-sample prediction for GREG, and an out-of-fold prediction for the two cross-fitted estimators. The variance estimates on the original scale are
\[
\widehat{\Var}_{\mathrm{fp}}(\hat\mu)=
\sum_h\left(\frac{N_h}{N}\right)^2\frac{1-f_h}{n_h}s_{r,h}^2,
\qquad
\widehat{\Var}_{\mathrm{sp}}(\hat\mu)=
\widehat{\Var}_{\mathrm{fp}}(\hat\mu)+\frac{\hat V_Y}{N},
\]
where
\[
\hat V_Y=\max\left\{
\frac1N\sum_{i:\delta_i=1}\frac{Y_i^2}{\pi_{N,i}}
-\left(\frac1N\sum_{i:\delta_i=1}\frac{Y_i}{\pi_{N,i}}\right)^2,\,0\right\}.
\]
These formulas are not algebraically identical to \eqref{eq-sp-var-estimator}--\eqref{eq-fp-var-estimator-mean-positive}. Under fixed-stratum sampling and consistent conditional-mean estimation, multiplying them by $N$ yields the same probability limits; with an imperfect working regression, they remain residual-linearization variance estimates for the stratified design but need not estimate the efficiency bounds. For the cross-fitted linear augmentation estimator, the standard errors from these formulas differ from those from \eqref{eq-fp-var-estimator-mean-positive} and \eqref{eq-sp-var-estimator} by 0.6\% and 0.8\% with the fold assignment used in the paper, and by at most 0.6\% and 0.9\% across 200 additional random fold assignments.

The repeated-sampling diagnostic in Table~\ref{tab:api-mean-repeated-sampling} draws 1,000 independent stratified samples, using SRSWOR within each stratum, of sizes 100, 50, and 50 from the fixed elementary-, high-, and middle-school populations. Nonsampled outcomes are masked during each fit, and the complete outcome vector is used only to draw the samples and to evaluate the finite-population error.

\begin{table}[!tb]
\centering
\setlength{\tabcolsep}{4pt}
\caption{Repeated stratified sampling from the fixed California API population.}
\label{tab:api-mean-repeated-sampling}

\begin{tabularx}{\linewidth}{@{}Yrrrrrr@{}}
\toprule
Estimator & Bias & Emp. SD & Avg. est. SD & SE/SD & Coverage & Var./GREG \\
\midrule
HT & 13.047 & 804.519 & 775.896 & 0.964 & 0.945 & 24.094 \\
GREG & -3.320 & 163.902 & 153.213 & 0.935 & 0.924 & 1.000 \\
Cross-fitted linear augmentation & -2.351 & 165.391 & 166.994 & 1.010 & 0.936 & 1.018 \\
Cross-fitted spline EIF & -0.961 & 169.286 & 169.760 & 1.003 & 0.939 & 1.067 \\
\bottomrule
\end{tabularx}
\par\smallskip\parbox{0.95\linewidth}{Bias and standard deviations are on the $\sqrt{N}$ scale. The last column is empirical variance relative to GREG. This is a design-conditional finite-population diagnostic based on repeated samples from the fixed realized population. It is not an assessment of the joint superpopulation--design efficiency bound.}
\end{table}

\subsubsection{Conditional-mean regression}
\label{app-api-regression-implementation}

The working dictionary $Z$ has an intercept, the eight standardized continuous covariates, and school-type indicators. OLS is unweighted on the sampled units. The variance models are cross-fitted as follows. Within each training fold, an unweighted OLS fit of $Y$ on $Z$ gives residuals $e_i$ and their mean square $s^2$. A Gamma log-link regression of $\max(e_i^2,10^{-8}s^2)/s^2$ on the variance dictionary is fitted by damped Newton steps, and a failed convergence check stops the calculation. The fitted values are rescaled to match the mean response in the training fold, multiplied by $s^2$, and clipped to $[0.02s^2,50s^2]$. The dictionary of GLS-size consists of an intercept and standardized log enrollment; that of GLS-flex consists of an intercept, linear and quadratic terms of the continuous covariates, and school-type indicators. Sample fits never use nonsampled outcomes.

Writing $w_i=1$ for OLS and $w_i=1/\hat\sigma^2_{-k(i)}(X_i)$ for GLS, the sandwich covariance on the $\sqrt N$ scale is
\[
\left(\frac1N\sum_i\delta_iw_iZ_iZ_i^\top\right)^{-1}
\left(\frac1N\sum_i\delta_iw_i^2Z_iZ_i^\top
       (Y_i-Z_i^\top\hat\beta)^2\right)
\left(\frac1N\sum_i\delta_iw_iZ_iZ_i^\top\right)^{-1}.
\]
The plug-in covariance of GLS is $\{N^{-1}\sum_i\delta_iw_iZ_iZ_i^\top\}^{-1}$, the superpopulation variance estimator of Section~\ref{sec-regression-target-estimation}. The sandwich treats the fitted weights as given. Multiplying all fitted variances by a common constant leaves the GLS coefficient and the sandwich covariance unchanged and rescales only the plug-in covariance.

Fold sensitivity uses 200 additional random fold assignments within school type. The repeated-sampling diagnostic draws 5,000 stratified samples as for the mean, with a new fold assignment for each sample. Across the ten slope coefficients, the empirical variance relative to OLS is 0.878--0.980 for GLS-size and 1.463--2.137 for GLS-flex. The census OLS coefficient is the least-squares coefficient of the working regression in the full population; its $\pi$-weighted version uses the weights $\pi_{N,i}$ and is the population counterpart of the unweighted sample estimating equation. The residual standard deviations by enrollment decile and the intraclass correlation quoted in Sections~\ref{sec-real-data-regression} and~\ref{sec-discussion}, respectively, are computed from the residuals of the census least-squares fit; the latter is the one-way analysis-of-variance estimate with school districts as groups.

\FloatBarrier

\bibliographystyle{plainnat}
\bibliography{refs}

\end{document}